\documentclass[11pt,a4paper]{article}
\usepackage{color}
\usepackage{graphicx}
\usepackage{amsmath}
\usepackage{amssymb}
\usepackage{mathrsfs}
\usepackage[numbers,sort&compress]{natbib}
\usepackage{amsthm}
\usepackage{extarrows}
\usepackage{tikz}
\usepackage{float}
\usepackage{appendix}
\usepackage{hyperref}
\hypersetup{hidelinks}
\usepackage{setspace}
\usepackage{caption}
\usepackage{float}
\usepackage{subfigure}
\usepackage{geometry}
\usepackage{indentfirst}
\newtheorem{theorem}{Theorem}[section]

\newtheorem{corollary}[theorem]{Corollary}
\newtheorem{proposition}[theorem]{Proposition}

\newtheorem{remark}[theorem]{Remark}
\newtheorem{RHP}[theorem]{RH problem}
\newtheorem{Dbar}[theorem]{$\bar{\partial}$-Problem}
\makeatletter
\@addtoreset{equation}{section}
\makeatother
\renewcommand{\theequation}{\arabic{section}.\arabic{equation}}
\begin{document}
\title{On the Cauchy problem of defocusing mKdV equation with finite density initial data: long time asymptotics in soliton-less regions}
\author{Taiyang Xu$^1$\thanks{\ tyxu19@fudan.edu.cn} \ \ \ Zechuan Zhang$^2$\thanks{\ 17110180013@fudan.edu.cn} \ \ \ Engui Fan$^1$\thanks{\ faneg@fudan.edu.cn }}
\footnotetext[1]{\ School of Mathematical Sciences, Fudan University, Shanghai 200433, P.R. China.}

\date{ }
\baselineskip=16pt
\maketitle
\begin{abstract}
\baselineskip=16pt
We investigate the long time asymptotics  for the solutions to  the  Cauchy problem of defocusing modified Kortweg-de Vries (mKdV) equation
with finite density initial data. The present paper is the subsequent work of our previous paper [arXiv:2108.03650], which
gives the soliton resolution for  the defocusing mKdV equation in the central asymptotic sector $\{(x,t): \vert \xi \vert<6\}$ with $\xi:=x/t$.
In the present paper, via  the Riemann-Hilbert (RH) problem associated to the Cauchy problem, the long-time asymptotics in
 the soliton-less regions  $\{(x,t): \vert \xi \vert>6, |\xi|=\mathcal{O}(1)\}$ for  the defocusing mKdV equation are
further obtained. It is shown that the leading term of the asymptotics are in compatible with the ``background solution'' and the error terms are
derived via rigorous analysis.


\end{abstract}
\baselineskip=16pt

\begin{spacing}{1}
\setcounter{tocdepth}{2} \tableofcontents
\end{spacing}

\section{Introduction}
In the present work, we investigate the long time asymptotics in soliton-less regions for the defocusing modified Kortweg-de Vries (mKdV)
equation with finite density initial data:
\begin{align}
    &q_t(x,t)-6q^2(x,t)q_{x}(x,t)+q_{xxx}(x,t)=0, \quad (x,t)\in \mathbb{R}\times \mathbb{R}_{+},\label{dmkdv}\\
    &q(x,0)=q_{0}(x),\quad \lim_{x\rightarrow\pm\infty}q_{0}(x)=\pm1.\label{bdries}
\end{align}

\begin{remark}\rm 
    Generally, the finite type initial data is presented by the nonzero boundary condition 
    $\lim_{x\rightarrow\pm\infty}q_{0}(x)=q_{\pm}$, $\vert q_{\pm}\vert=q_{0}$. Taking the following transformation 
    \begin{equation*}
        u=q/q_{0}, \quad \tilde{x}=q_{0}x, \quad \tilde{t}=q_{0}^{3}t, 
    \end{equation*}
    we have 
    \begin{equation*}
        u_{\tilde{t}}(\tilde{x}, \tilde{t})-6u^2(\tilde{x}, \tilde{t})u_{\tilde{x}}(\tilde{x}, \tilde{t})+u_{\tilde{x}\tilde{x}\tilde{x}}(\tilde{x}, \tilde{t})=0,
    \end{equation*}
    with the normalized boundary conditions $\lim_{\tilde{x}\rightarrow\pm\infty}u=q_{\pm}/q_{0}$, $\vert q_{\pm}/q_{0} \vert=1$.
    Based on the analysis, we directly choose the boundary condition of initial data as \eqref{bdries} for convenience.
\end{remark}

\begin{remark}\rm
    The ``soliton-less regions'' doesn't represent that there exist no solitons in our present work. Indeed, we use a interpolation transformation to convert residue conditions for poles
    to jump conditions such that the jump matrices vanish as $t\rightarrow\infty$. In our result, the solitons make few
    contributions (exponential decay) for the obtained asymptotics.  
\end{remark}

The mKdV equation arises in various of physical fields, such as acoustic wave and phonons in a certain anharmonic
lattice \cite{Zab1967,Ono1992}, Alfv\'en wave in a cold collision-free plasma \cite{kaku1969, Kha1998}.
A considerable amount of work has been carried out around the long time asymptotics for defocusing mKdV equation \eqref{dmkdv}.
The earliest work can be traced back to Segur and Ablowitz \cite{AblowitzSegur}, who  extend  the  method
developed by Zakharov and Manakov \cite{Zak-Mana1976} to derive the leading asymptotics for the solution of the mKdV equation, including full
information on the phase. The most influential work to investigate the long time behavior of integrable PDEs
is the nonlinear steepest descent method which was firstly proposed by Deift and Zhou (Deift-Zhou method)
to study the defocusing mKdV equation \cite{DZAnn}. Lenells proves a nonlinear steepest descent theorem for RH problems
with Carleson jump contours, where jump matrices admit low regularity and slow decay \cite{lenellsmkdv}.
Recently, Chen and Liu extend the asymptotics to the solution for defocusing mKdV equation with initial data in lower regularity
spaces \cite{CandLdmkdv}. The works mentioned above refer to that initial data $q_0(x)$
admits zero boundary conditions (ZBCs, i.e., $q_0(x)\rightarrow 0$ as $x\rightarrow\pm\infty$).

Studies for the long time asymptotic behavior of the integrable systems with nonzero boundary conditions (NBCs) have been investigated in
a number previous articles. Specifically, the nonzero boundary conditions could be divided into the
asymmetric NBCs (i.e., $q_0(x)\rightarrow q_{\pm}$ with $|q_+|\neq |q_{-}|$, also be called ``step-like'' initial data) and
symmetric NBCs (i.e., $q_0(x)\rightarrow q_{\pm}$ with $|q_+|=|q_{-}|\neq 0$).
For the long time asymptotic behavior for integrable PDEs with asymmetric NBCs, refer \cite{MonLenSheCMP,MonLenSheCMP3,MonKotSheIMRN,KotMinakovJMP,KotMinakovJMPAG1,KotMinakovJMPAG2,MinakovJPAMaTheor}.
For the symmetric NBCs, a lot of works for long time asymptotics have been investigated around nonlinear Schr\"odinger (NLS) equation, see \cite{Vartan1,Vartan2,BiondiniManNLSNBCs1,BiondiniManNLSNBCs2}.
S. Cuccagna and R. Jenkins \cite{Cuc} develop the $\bar{\partial}$ generalization which was firstly proposed by McLaughlin and his collaborators \cite{M&M2006, M&M2008, Dieng, Bor} to verify
the soliton resolution for defocusing NLS equation with finite initial data in an asymptotic soliton regime $|x/2t|<1$. The method used in \cite{Cuc} is applied
to investigate the asymptotics for $|x/2t|>1$ by Wang and Fan \cite{WFDNLS2022}.

For the defocusing mKdV equation with finite density initial data defined by \eqref{dmkdv}-\eqref{bdries}, Zhang and Yan \cite{Z&Y} use the inverse scattering transform (IST) to express the solution
in terms of the associated RH problem  and prove that the discrete spectrum all locate on the unit circle in the complex plane.
For comparison, only focusing mKdV equation posses discrete spectrum under ZBCs.
In the presence of discrete spectrum for defocusing mKdV equation with finite density initial data, we exhibit
the soliton resolution and asymptotic stability in the previous article \cite{Zhang&Xu mkdv} for $|\xi|<6$, and the asymptotics for $\vert \xi\vert>6$, $\vert\xi\vert=\mathcal{O}(1)$ in the present work.



\subsection{Main results}
The main result of this work is exhibited in the following theorem that reveals the long time asymptotic behavior of the solution $q(x,t)$
of defocusing mKdV equation \eqref{dmkdv} in different asymptotic sectors (see Figure \ref{cone}), where
\begin{align}
  & \mathcal{R}_{L}=\left\{(x,t): \xi<-6, |\xi|=\mathcal{O}(1) \right\}, \quad \mathcal{R}_{M}=\left\{(x,t): -6<\xi<6 \right\}, \nonumber\\
  & \mathcal{R}_{R}=\left\{(x,t): \xi>6, |\xi|=\mathcal{O}(1) \right\}, \ \ \xi:=x/t.\nonumber
\end{align}

\begin{figure}[htbp]
    \begin{center}
    \begin{tikzpicture}[node distance=2cm]
    \draw[->](-5.5,0)--(5.5,0)node[right]{$x$};
    \draw[->](0,0)--(0,4)node[above]{$t$};
    \draw[red, dashed](0,0)--(-5,1.5)node[above,black]{$\xi=-6$};
    \draw[red, dashed](0,0)--( 5,1.5)node[above,black]{$\xi=6$};
    \node[below]{$0$};
    \coordinate (A) at (-5, 0.5);
	\fill (A) node[right] {$\mathcal{R}_{L}$};
    \coordinate (B) at (0,1.5);
    \fill (B) node[right] {$\mathcal{R}_{M}$};
    \coordinate (D) at (5, 0.5);
	\fill (D) node[left] {$\mathcal{R}_{R}$};
    \end{tikzpicture}
    \caption{\small Asymptotic sectors for the solution $q(x,t)$ for mKdV equation} \label{cone}
    \end{center}
    \end{figure}
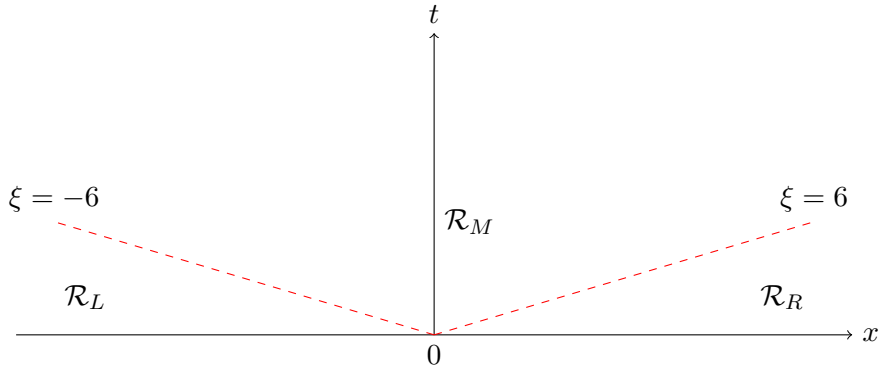

\begin{theorem}\label{mainthm}
    Let $q(x,t)$ be the solution for the Cauchy problem \eqref{dmkdv} with generic data $q_{0}(x)- \tanh(x)\in H^{4,4}(\mathbb{R})$ associated to
    scattering data $\Big{\{}r(z), \{\eta_n, c_n\}_{n=1}^{2N}\Big{\}}$. As $t\rightarrow+\infty$, the following three asymptotics are shown.
    \begin{itemize}
    \item[(a)] For $(x,t)\in\mathcal{R}_{L}$ (left field),
    \begin{equation}\label{mainthma}
    q(x,t)=-1+t^{-\frac{1}{2}}f(\xi)+\mathcal{O}(t^{-\frac{3}{4}}).
    \end{equation}
    where
    \begin{equation*}
        f(\xi):=\sum_{j=1}^{4}\epsilon_j\left(2\epsilon_j\theta''(\xi_j) \right)^{-\frac{1}{2}}\left(1-\xi_j^{-2}\right)^{-1}
    \cdot \left(\beta^{(\xi_j)}_{12}-\frac{1}{\xi_j^2}\beta^{(\xi_j)}_{21}\right),
    \end{equation*}
    with $\epsilon_j=(-1)^{j+1}$ for $j=1,2,3,4$, $\xi_j$ defined by \eqref{saddlexi1,4}-\eqref{saddlexi2,3} for $j=1,2,3,4$, $\beta_{12}^{(\xi_j)}$ and $\beta_{21}^{(\xi_j)}$
    defined by \eqref{betaxi1-1}-\eqref{betaxi1-3} for $j=1,3$, while by \eqref{betaxi2-1}-\eqref{betaxi2-3} for $j=2,4$.
    \item[(b)] For $(x,t)\in\mathcal{R}_{M}$
    \begin{equation}
    q(x,t)=- 1+\sum_{j=0}^{N}[sol(z_{j},x-x_j,t)+1] +\mathcal{O}(t^{-1}).
    \end{equation}
.
    \item[(c)] For $(x,t)\in\mathcal{R}_{R}$ (right field),
    \begin{equation}
        q(x,t)=1+\mathcal{O}(t^{-1}),
    \end{equation}

    \end{itemize}
\end{theorem}

\begin{remark}\rm
    Comparing to the results in \cite{Cuc}, \cite{WFDNLS2022} for defocusing NLS equation, the asymptotics in Theorem \ref{mainthm} is real-valued, which owes to the mKdV equation is a real-valued integrable PDE.
    It's find that the asymptotics \eqref{mainthma} is formally similar to the asymptotics in \cite{WFDNLS2022}, and the sub-leading term stems the contribution from four saddle points (in the case of mKdV equation) and
    two saddle points (in the case of NLS equation) respectively. The other difference between the present work and \cite{WFDNLS2022} is the asymptotics in right field, where the error bounds of the former
    mainly stems from the $\bar{\partial}$ estimation by $\mathcal{O}(t^{-1})$.  
\end{remark}

\begin{remark}\rm
    $q^{sol}(x,t)=\tanh(x)$ is the stationary solution of  \eqref{dmkdv}-\eqref{bdries}, which is called the dark  soliton.
\end{remark}

\begin{remark}\label{whyweneedM}\rm
$|\xi|=\mathcal{O}(1)$ is needed to ensure the following issues:\\
- The saddle points $\xi_1$, $\xi_4$ defined by \eqref{saddlexi1,4} are bounded;\\
- The  estimates for $\Im\theta(z)$, jump matrix and $\bar{\partial}$ derivatives are reasonable, see Proposition \ref{lemz=0},
Proposition \ref{lemz=xi1}, Proposition \ref{estopenlenssaddle},  Proposition \ref{outxi1}, Proposition \ref{outxi2}, Proposition \ref{estS} and Proposition \ref{estopenlesz=0xi>6};\\
- The higher power term of the expansion for $\theta(z)$ near saddle points could decay as $t\rightarrow\infty$, see Remark \ref{rmk:phasereduction}.\\
If removing the condition $\vert x/t \vert=\mathcal{O}(1)$, we can turn to study the large $x$ asymptotic behavior in a similar way to large $t$ asymptotics.
\end{remark}

\begin{remark}\rm
    In an early version of this paper, $\mathcal{R}_{R}$ is set by $\xi\in(-2,+\infty), |\xi|=\mathcal{O}(1)$ instead of $\xi\in(6,+\infty), |\xi|=\mathcal{O}(1)$.
    The reason why we modify this condition is that we find when $\xi\in(-2,-6)$, the asymptotics can be reduced to a specific case in the Theorem \ref{mainthm}(b), 
    which verifies the soliton resolution in \cite{Zhang&Xu mkdv}, by setting the index $\Lambda=\emptyset$ in \cite{Zhang&Xu mkdv}.
\end{remark}

\begin{remark}\rm
    The smoothness and decay properties of the reflection coefficient $r(z)$ are needed in our analysis.\\
    - Proposition \ref{rinH1} shows that: $q_{0}- \tanh(x)\in H^{3,3}(\mathbb{R})\Rightarrow q_{0}- \tanh(x)\in L^{1,2}(\mathbb{R})\Rightarrow r(z)\in H^{1}(\mathbb{R})$.\\
    - Eq.\eqref{asymptoticsofr} shows $r(z)=\mathcal{O}(z^{-2})$ as $z\rightarrow\infty$, and we can obtain that $r(z)$ also belongs to $L^{2,1}(\mathbb{R})$.
    Moreover, $r(z)\in H^{1,1}(\mathbb{R})=L^{2,1}(\mathbb{R})\cap H^{1}(\mathbb{R})$. It's Corollary \ref{rinH11}.\\
    The condition $q_{0}-\tanh(x)\in H^{4,4}(\mathbb{R})$ in Theorem \ref{mainthm} is needed to include all conditions to show that $r(z)\in H^{1}(\mathbb{R})$, which can help us bound the $\bar{\partial}$ derivatives of our extensions in
    Proposition \ref{estopenlenssaddle}, Proposition \ref{estS} and Proposition \ref{estopenlesz=0xi>6}, etc.
\end{remark}

\subsection{Outline of this paper}
The structure of this work is as follows.

Section \ref{RHconstruct} and Section \ref{disphasepoint} are the preliminary parts. In Section \ref{RHconstruct}, we review the elementary results on the associated RH problem formulation
of the Cauchy problem for the defocusing mKdV equation \eqref{dmkdv}, which is the basis to analyze the asymptotic behavior of the defocusing mKdV equation in our work.
In Section \ref{disphasepoint}, we present the distribution of phase points of $\theta(z)$ and depict the signature tables of $e^{2it\theta}$ by some numerical figures.

In Section \ref{sec:proofof1a}, we mainly deal with the asymptotics for $\xi\in\mathcal{R}_{L}$. In Subsection \ref{1stdeform}, jump matrix factorizations corresponding to this case are
given. In Subsection \ref{subsec:RLmtom1}, a scalar function $\delta(z)$ which could use the factorizations of the jump matrix along the real axis to deform the contours onto those which
the oscillatory jump on the real axis for exponential decay, and a interpolation function $G(z)$ which interpolates the poles by trading them for jumps along small closed circles around each poles
are introduced to make the first transformation $M(z)\rightarrow M^{(1)}(z)$. In Subsection \ref{subsec:RLopenlens}, we open $\bar{\partial}$ lenses to set up
a mixed $\bar{\partial}$-RH problem $M^{(2)}(z)$, which consists of a pure RH problem $M^{(PR)}(z)$ and a pure $\bar{\partial}$-problem $M^{(3)}(z)$. In Subsection \ref{subsec:RLpureRH},
analysis on pure RH problem $M^{(PR)}(z)$ is exhibited, which refer to two standard parts: global parametrix $M^{(\infty)}(z)$ and local parametrix $M^{(LC)}(z)$. Error analysis via using
small-norm RH theory is also given. In Subsection \ref{subsec: RLpuredbar}, we give the rigorous analysis for the pure $\bar{\partial}$-problem $M^{(3)}(z)$. In Subsection \ref{subsec:pfofa},
the asymptotics in Theorem \ref{mainthm}(a) is given by reviewing a series of transformations we use in this section.

Similar techniques to Section \ref{sec:proofof1a} are used to analyze the asymptotics for $\xi\in\mathcal{R}_{R}$ in Section \ref{sec:pfof1c}.

\subsection{Notations}
We conclude this section with some notations used throughout this paper.\\
- Japanese bracket $\langle x\rangle:=\sqrt{1+|x|^2}$ is widely used in some normed space.\\
- A weighted $L^{p,s}(\mathbb{R})$ is defined by $L^{p,s}(\mathbb{R})=\{u\in L^{p}(\mathbb{R}):\langle x\rangle^s u(x)\in L^{p}(\mathbb{R})\}$,
    with $\Vert u \Vert_{L^{p,s}(\mathbb{R})}:=\Vert\langle x \rangle^{s}u \Vert_{L^{p}(\mathbb{R})}$.\\
- A Sobolev space is defined by $W^{m,p}(\mathbb{R})=\{u\in L^{p}(\mathbb{R}): \partial ^{j}u(x)\in L^{p}(\mathbb{R}) \quad{\rm for}\quad j=0,1,2,\dots,m\}$,
    with $\Vert u \Vert_{W^{m,p}(\mathbb{R})}:=\sum_{j=0}^{m}\Vert \partial^{j}u \Vert_{L^{p}(\mathbb{R})}$. Usually, we are used to expressing $H^{m}(\mathbb{R}):=W^{m,2}(\mathbb{R})$.\\
- A weighted Sobolev space is defined by $H^{m,s}(\mathbb{R}):=L^{2,s}(\mathbb{R})\cap H^{m}(\mathbb{R})$.\\
- $\sigma_1, \sigma_2,\sigma_3$ are classical Pauli matrices as follows
    \begin{equation*}
        \sigma_1=\begin{pmatrix} 0 & 1 \\ 1 & 0 \end{pmatrix},\quad
        \sigma_2=\begin{pmatrix} 0 & -i \\ i & 0 \end{pmatrix}, \quad
        \sigma_3=\begin{pmatrix} 1 & 0 \\ 0 & -1 \end{pmatrix}.
    \end{equation*}
- $a\lesssim b$ i.e., $\exists c$, s.t. $a\leqslant cb$; \ \ \
  $\epsilon_j:=(-1)^{j+1}$. \\
- $\Re$ and $\Im$ represent real part and imaginary part of a complex variable respectively.

\section{Direct and inverse scattering transform}\label{RHconstruct}
\subsection{Lax pair and spectral analysis}
The defocusing mKdV equation \eqref{dmkdv}
admits the following Lax pair \cite{AKNS}
\begin{equation}\label{lax pair}
    \Phi_x=X\Phi, \quad \Phi_t=T\Phi,
\end{equation}
where
\begin{align*}
    &X=ik\sigma_3+Q,\ \ T=4k^2X-2ik\sigma_3(Q_x-Q^2)+2Q^3-Q_{xx}, \ \ Q=\begin{pmatrix} 0 & q(x,t) \\  q(x,t) & 0 \end{pmatrix},
\end{align*}
and $k \in \mathbb{C}$ is a spectral parameter.\\

By using the boundary condition of \eqref{bdries}, the Lax pair \eqref{lax pair} admits the following approximation
\begin{equation}\label{asy spec prob}
    \Phi_{\pm, x}\sim X_{\pm}\Phi_{\pm}, \quad \Phi_{\pm, t}\sim T_{\pm}\Phi_{\pm}, \quad  x\rightarrow\pm\infty,
\end{equation}
where
\begin{equation*}
    X_{\pm}=ik\sigma_3+Q_{\pm}, \quad T_{\pm}=(4k^2+2)X_{\pm},
\end{equation*}
with $Q_{\pm}=\pm\sigma_1$.

The eigenvalues of $X_{\pm}$ are $\pm i\lambda$, which satisfy the equality
\begin{equation}
    \lambda^2=k^2-1.
\end{equation}
Since $\lambda$ is multi-valued, we introduce the following uniformization variable to ensure that our discussion is
based on a complex plane rather than a Riemann surface
\begin{equation}
    z=k+\lambda,
\end{equation}
and obtain two single-valued functions
\begin{equation}
    \lambda(z)=\frac{1}{2}(z-\frac{1}{z}),\quad k(z)=\frac{1}{2}(z+\frac{1}{z}).
\end{equation}
\begin{remark}\rm
    In the present work, we use the uniformization technique. Indeed, the other techniques, such as \cite{BiondiniManNLSNBCs1, BiondiniManNLSNBCs2} can be taken, which
    should deal with a branch cut. However, we have to pay more attention to singular points which appears via uniformization method.
\end{remark}

Define two domains $D_{+}$, $D_{-}$ and their boundary $\Sigma$ on $z$-plane by
\begin{align*}
    D_{+}:=\{z\in \mathbb{C}, \Im\lambda(1+\frac{1}{|z|^2})>0\},\ D_{-}:=\{z\in \mathbb{C}, \Im\lambda(1+\frac{1}{|z|^2})<0\},\  \Sigma:=\mathbb{R}\backslash \{0\}.
\end{align*}
the ``background solution'' of the asymptotic spectral problem \eqref{asy spec prob} is given by
\begin{equation}
    \Phi_{\pm}\sim E_{\pm}(z)e^{i\lambda(z) x\sigma_3},
\end{equation}
where
\begin{equation*}
    E_{\pm}=\begin{pmatrix} 1 & \pm \frac{ i}{z} \\  \mp \frac{ i}{z} & 1 \end{pmatrix}.
\end{equation*}
Introducing the modified Jost solution
\begin{equation}\label{muPhirelation}
    \mu_{\pm}=\Phi_{\pm}e^{-i\lambda(z)x\sigma_3},
\end{equation}
then we have
\begin{align*}
    &\mu_{\pm}\sim E_{\pm}, \quad {\rm as} \quad x\rightarrow\pm\infty, \\
    &{\rm det}(\Phi_{\pm})={\rm det}(\mu_{\pm})={\rm det}(E_{\pm})=1-\frac{1}{z^2}.
\end{align*}
$\mu_{\pm}$ are defined by the Volterra type integral equations
\begin{align}
    &\mu_{\pm}(x;z)=E_{\pm}(z)+\int_{\pm\infty}^{x} E_{\pm}(z)e^{i\lambda(z)(x-y)\hat{\sigma}_3}\left[(E^{-1}_{\pm}(z)\Delta Q_{\pm}\left(y\right)\mu_{\pm}\left(y; z\right)\right]dy, \quad z\neq \pm 1,\label{Vo1}\\
    &\mu_{\pm}(x;z)=E_{\pm}(z)+\int_{\pm\infty}^{x} \left[I+\left(x-y\right)\left(Q_{\pm}\pm i\sigma_3\right)\right]\Delta Q_{\pm}\left(y\right)\mu_{\pm}\left(y; z\right)dy, \quad z=\pm 1,\label{Vo2}
\end{align}
where $\Delta Q=Q-Q_{\pm}$.

The properties of $\mu_{\pm}$ are conclude in the following proposition, of which proof is similar to \cite[Lemma 3.3]{Cuc} owing to defocusing mKdV equation and NLS equation admit
the same spatial spectrum problem ($t=0$).
\begin{proposition}\label{analydiff}
    Given $n\in\mathbb{N}_0$, let $q-\tanh(x)\in$ $L^{1,n+1}(\mathbb{R})$, $q'\in$ $W^{1,1}(\mathbb{R})$.
    Denote $\mu_{\pm,j}$ the $j$-th column of $\mu_{\pm}$.\\
    - For $z\in\mathbb{C}\setminus\{0\}$, $\mu_{+,1}(x,t;z)$ and $\mu_{-,2}(x,t;z)$ can be analytically extended to $\mathbb{C}_{+}$ and continuously extended to $\mathbb{C}_{+}\cup \Sigma$;
    $\mu_{-,1}(x,t;z)$ and $\mu_{+,2}(x,t;z)$ can be analytically extended to $\mathbb{C}_{-}$ and continuously extended to $\mathbb{C}_{-}\cup \Sigma$. \\
    - (Symmetry for $\mu_{\pm}$) \ $\mu_{\pm}(z)=\sigma_1\overline{\mu_{\pm}(\bar{z})}\sigma_1=\overline{\mu_{\pm}(-\bar{z})}$, $\mu_{\pm}(z)=\frac{\mp 1}{z}\mu_{\pm}(z^{-1})\sigma_2$.\\
    - (Asymptotic behavior of $\mu_{\pm}$ as $z\rightarrow\infty$) For $\Im z\geq 0$ as $z\rightarrow\infty$,
    \begin{align}
    &\mu_{+,1}(z)=e_1+\frac{1}{z}\left(\begin{array}{c}
                                    -i\int_{x}^{\infty}(q^2-1)dx\\
                                    -iq
                                    \end{array}\right)+\mathcal{O}(z^{-2}),\\
    &\mu_{-,2}(z)=e_2+\frac{1}{z}\left(\begin{array}{c}
                                    iq\\
                                    i\int_{x}^{\infty}(q^2-1)dx
                                    \end{array}\right)+\mathcal{O}(z^{-2}),
    \end{align}
    for $\Im z\leq0$, as $z\rightarrow\infty$
    \begin{align}
    &\mu_{-,1}(z)=e_1+\frac{1}{z}\left(\begin{array}{c}
                                    -i\int_{x}^{\infty}(q^2-1)dx\\
                                    -iq
                                    \end{array}\right)+\mathcal{O}(z^{-2}),\\
    &\mu_{+,2}(z)=e_2+\frac{1}{z}\left(\begin{array}{c}
                                    iq\\
                                    i\int_{x}^{\infty}(q^2-1)dx
                                    \end{array}\right)+\mathcal{O}(z^{-2}).
    \end{align}
    - (Asymptotic behavior of $\mu_{\pm}$ as $z\rightarrow 0$) For $z\in\mathbb{C}_{+}$, as $z\rightarrow0$,
    \begin{equation}
    \mu_{+,1}(z)=-\frac{i}{z}e_2+\mathcal{O}(1),\hspace{0.5cm}\mu_{-,2}(z)=-\frac{i}{z}e_1+\mathcal{O}(1);
    \end{equation}
    for $z\in\mathbb{C}_{-}$, as $z\rightarrow0$,
    \begin{equation}
    \mu_{-,1}(z)=\frac{i}{z}e_2+\mathcal{O}(1),\hspace{0.5cm}\mu_{+,2}(z)=\frac{i}{z}e_1+\mathcal{O}(1);
    \end{equation}
    where $e_1=(1,0)^{\rm T}$, $e_2=(0,1)^{\rm T}$.
\end{proposition}

Since $\Phi_{+}(z)$ and $\Phi_{-}(z)$ are two fundamental matrix-valued solutions of \eqref{asy spec prob} for $z\in \Sigma\backslash \{\pm 1\} $,
thus there exists a scattering $S(z)$ such that
\begin{equation}\label{linear relation}
    \Phi_{+}(x,t;z)=\Phi_{-}(x,t;z)S(z).
\end{equation}
Owing to \eqref{muPhirelation} and the symmetries of $\mu_{\pm}(z)$, $\Phi_{\pm}$ admit the following symmetry.
\begin{equation}
    \Phi_{\pm}(z)=\sigma_1\overline{\Phi_{\pm}(\bar{z})}\sigma_1=\overline{\Phi_{\pm}(-\bar{z})}, \quad \Phi_{\pm}(z)=\frac{\mp 1}{z}\Phi_{\pm}(z^{-1})\sigma_2.
\end{equation}
Then the symmetry $S(z)=\sigma_1\overline{S(\bar{z})}\sigma_1=\overline{S(-\bar{z})}=-\sigma_2S(z^{-1})\sigma_2$ follows immediately.
And $S(z)$ is given by
\begin{equation}
    S(z)=\begin{pmatrix} a(z) & \overline{b(z)} \\ b(z) & \overline{a(z)} \end{pmatrix},\quad z\in \Sigma\backslash\{\pm 1\},
\end{equation}
where $a(z)$, $b(z)$ are called scattering data.

Define
\begin{equation}\label{rz}
    r(z):=\frac{b(z)}{a(z)}, \quad \tilde{r}(z):=\frac{b(\bar{z})}{a(\bar{z})}.
\end{equation}
Several properties of $a(z)$, $b(z)$ and $r(z)$ are given as follows.
\begin{proposition}\label{spec distri} Let $z\in\Sigma \backslash \{\pm 1\}$, $a(z)$, $b(z)$ and $r(z)$ be the data as mentioned above. \\
    - The scattering coefficients can be expressed in terms of the Jost functions by
        \begin{equation}\label{azbz}
            a(z)=\frac{{\rm det}(\Phi_{+,1}, \Phi_{-,2})}{1-z^{-2}}, \quad b(z)=\frac{{\rm det}(\Phi_{-,1}, \Phi_{+,1})}{1-z^{-2}}.
        \end{equation}
    - For each $z\in \Sigma\backslash\{\pm 1\}$, we have
        \begin{equation}\label{relazbz}
            {\rm det}S(z)=|a(z)|^2-|b(z)|^2=1, \quad |r(z)|^2=1-|a(z)|^{-2}<1.
        \end{equation}
    - $a(z)$, $b(z)$ and the reflection coefficient $r(z)$ satisfy the symmetries
        \begin{align}
            &a(z)=\overline{a(-\bar{z})}=-\overline{a(\bar{z}^{-1})},\\
            &b(z)=\overline{b(-\bar{z})}=\overline{b(\bar{z}^{-1})},\\
            &r(z)=\overline{r(-\bar{z})}=-\overline{r(\bar{z}^{-1})}.
        \end{align}
    - The scattering data admit the asymptotics
        \begin{align}
        &\lim_{z\rightarrow\infty}(a(z)-1)z=i\int_{\mathbb{R}}(q^2-1)dx, \ z\in\overline{\mathbb{C}_{+}}, \label{asymptoticsfora1}\\
        &\lim_{z\rightarrow0}(a(z)+1)z^{-1}=i\int_{\mathbb{R}}(q^2-1)dx, \ z\in\overline{\mathbb{C}_{+}}, \label{asymptoticsfora2}
        \end{align}
        and
        \begin{align}
        &|b(z)|=\mathcal{O}(|z|^{-2}),\hspace{0.5cm} \text{as }|z|\rightarrow\infty,\ \ z\in\mathbb{R},\label{asymptoticsforb1}\\
        &|b(z)|=\mathcal{O}(|z|^{2}),\hspace{0.5cm} \text{as }|z|\rightarrow0, \ z\in\mathbb{R}.\label{asymptoticsforb2}
        \end{align}
        So that
    \begin{align}\label{asymptoticsofr}
    r(z)\sim z^{-2}, \hspace{0.2cm}|z|\rightarrow\infty;\hspace{0.5cm}r(z)\sim z^2,\hspace{0.2cm}|z|\rightarrow0.
    \end{align}
\end{proposition}
\begin{proof} The first item follows by applying Cramer's rule to \eqref{linear relation}. The second item can be obtained
    by direct calculation. The third item follows from the symmetry of $S(z)$. The fourth item follows from the first item and the asymptotics of $\mu_{\pm}$ in
    Proposition \ref{analydiff}.
\end{proof}
Though $a(z)$ and $b(z)$ have singularities at $\pm 1$, the reflection coefficient $r(z)$ remains bounded at $z=\pm 1$ with $|r(\pm 1)|=1$. Indeed, as $z\rightarrow \pm 1$
\begin{align}
    &a(z)=\frac{\pm S_{\pm}}{z\mp 1}+\mathcal{O}(1),\quad b(z)=\frac{\mp S_{\pm}}{z\mp 1}+\mathcal{O}(1)
\end{align}
where $S_{\pm}=\frac{1}{2}{\rm det}[\mu_{+.1}(\pm 1, x), \mu_{-,2}(\pm 1, x)]$. Then $\lim_{z\rightarrow\pm 1}r(z)=\mp i$ follows.

\begin{remark}\rm
    The above discussions suggest that scattering data exhibit singular behavior for $z$ at $\pm 1, 0$.
The singularities of these functions at $z=\pm 1$ can be removable, however, the singular behavior at $z=0$ plays a non-trivial and unavoidable role in our analysis.
\end{remark}

The next proposition shows that, given data $q_{0}$ with sufficient smoothness and decay properties,
the reflection coefficients will also be smooth and decaying.
\begin{proposition}\label{rinH1}
    For  given $q-\tanh(x)\in L^{1,2}(\mathbb{R})$, $q'\in W^{1,1}(\mathbb{R})$,  then $r(z)\in H^{1}(\mathbb{R})$.
\end{proposition}
\begin{proof}
    The proof is the same with \cite[Proposition 3.2]{Zhang&Xu mkdv}.
\end{proof}
\begin{remark}\rm
    $\Vert r\Vert_{H^{1}(\mathbb{R})}$ is widely used in the estimation below, such as Proposition \ref{deltapro},
    Proposition \ref{estopenlenssaddle}, Proposition \ref{estS}, Proposition \ref{estopenlesz=0xi>6}, etc.
    In fact, we can claim that $r(z)\in H^{1,1}(\mathbb{R})$, see the following corollary.
\end{remark}
\begin{corollary}\label{rinH11}
    For  given $q-\tanh(x)\in L^{1,2}(\mathbb{R})$, $q'\in W^{1,1}(\mathbb{R})$,  we have $r(z)\in H^{1,1}(\mathbb{R})$.
\end{corollary}
\begin{proof}
    Since $H^{1,1}(\mathbb{R})=L^{2,1}(\mathbb{R})\cap H^{1}(\mathbb{R})$, what we need to prove is that $r\in L^{2,1}(\mathbb{R})$.
    With \eqref{asymptoticsofr}, we can see that
    \begin{equation}
        \vert z\vert^2 r^{2}(z) \sim  \vert z\vert^{-2}, \quad |z|\rightarrow \infty
    \end{equation}
    Thus
    \begin{equation}
        \int_{\mathbb{R}}\vert \langle z\rangle r(z)\vert^2 <\infty,
    \end{equation}
    which implies the result.
\end{proof}

In a similar way \cite{ISTDNLS}, we can show that zeros of $a(z)$  are  finite and  simple,  all of which are placed on the unit circle $\{z: |z|=1\}$ (see Figure \ref{spectrumsdis}).
Suppose that $a(z)$ has finite $N$ simple zeros $z_1, z_2,\dots, z_N$ on $D_{+}\cap \{z:|z|=1, \Im z>0, \Re z>0\}$.

The symmetries of $S(z)$ imply that
\begin{equation*}
a(z_n)=0 \Leftrightarrow a(\bar{z}_n)=0 \Leftrightarrow a(-z_n)=0 \Leftrightarrow a(-\bar{z}_n)=0, \quad n=1, \dots, N.
\end{equation*}
Therefore we give the discrete spectrum by
\begin{equation}
    \mathcal{Z}=\{z_n, \bar{z}_n -\bar{z}_n, -z_n\}_{n=1}^{N},
\end{equation}
where $z_n$ satisfies that $|z_n|=1$, $\Re z_n>0$, $\Im z_n>0$. Moreover, it is convenient to define that
\begin{equation}
    \eta_n=\left\{
        \begin{aligned}
        &z_n,\quad &n=1, \dots, N, \\
        &-\bar{z}_{n-N},\quad &n=N+1 ,\dots, 2N,
        \end{aligned}
        \right.
\end{equation}
from which we express the set $\mathcal{Z}$ in terms of
\begin{equation}
\mathcal{Z}=\{\eta_n, \bar{\eta}_n\}_{n=1}^{2N}.
\end{equation}
Using trace formulae, $a(z)$ is given by
\begin{equation}
    a(z)=\prod_{n=1}^{2N}\left(\frac{z-\eta_n}{z-\bar{\eta}_n}\right){\rm exp}\left[-\frac{1}{2\pi i}\int_{\Sigma}\frac{{\rm log}(1-r(s)\tilde{r}(s))}{s-z}ds\right], \quad z\in\mathbb{C}_{+}.
\end{equation}

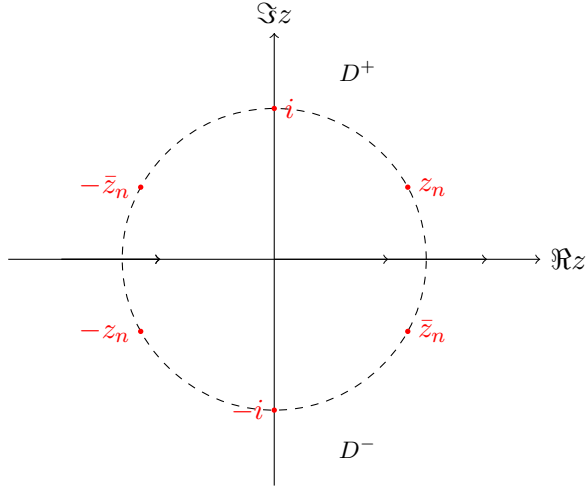
\begin{figure}[htbp]
	\centering
	\begin{tikzpicture}[node distance=2cm]
	\draw[->](-3.5,0)--(3.5,0)node[right]{$\Re z$};
	\draw[->](0,-3)--(0,3)node[above]{$\Im z$};
    \node   at (1.1,2.5) {\footnotesize $D^+$};
    \node   at (1.1,-2.5) {\footnotesize $D^-$};
	\draw[dashed] (2,0) arc (0:360:2);
	\draw[->](-2.8,0)--(-1.5,0);
	\draw[->](-2.8,0)--(-1.5,0);
	\draw[->](0,0)--(1.5,0);
	\draw[->](1.5,0)--(2.8,0);
	\coordinate (A) at (2,2.985);
    \coordinate (B) at (2,-2.985);
    \coordinate (C) at (-0.616996232,0.9120505887);
    \coordinate (D) at (-0.616996232,-0.9120505887);
    \coordinate (E) at (0.616996232,0.9120505887);
    \coordinate (F) at (0.616996232,-0.9120505887);
    \coordinate (G) at (-2,2.985);
    \coordinate (H) at (-2,-2.985);
	\coordinate (J) at (1.7570508075688774,0.956);
	\coordinate (K) at (1.7570508075688774,-0.956);
	\coordinate (L) at (-1.7570508075688774,0.956);
	\coordinate (M) at (-1.7570508075688774,-0.956);
    \coordinate (N) at (0, 2);
    \coordinate (O) at (0, -2);
	\fill[red] (J) circle (1pt) node[right] {$z_n$};
	\fill[red] (K) circle (1pt) node[right] {$\bar{z}_n$};
	\fill[red] (L) circle (1pt) node[left] {$-\bar{z}_n$};
	\fill[red] (M) circle (1pt) node[left] {$-z_n$};
    \fill[red] (N) circle (1pt) node[right]{$i$};
    \fill[red] (O) circle (1pt) node[left]{$-i$};
	\end{tikzpicture}
	\caption{\footnotesize The discrete spectrums distribute  on the unite circle $\{z:|z|=1\}$  on  the $z$-plane.
     }
	\label{spectrumsdis}
\end{figure}

Denoting norming constant $c_n=b_n/a'(\eta_n)$, the residue condition follows immediately
\begin{align}
    &\underset{z=\eta_n}{\rm Res} \left[\frac{\mu_{+,1}(z)}{a(z)}\right]=c_n e^{-2i\lambda(\eta_n)x}\mu_{-,2}(\eta_n),\\
    &\underset{z=\bar{\eta}_n}{\rm Res} \left[\frac{\mu_{+,2}(z)}{\overline{a(\bar{z})}}\right]=\bar{c}_n e^{2i\lambda(\bar{\eta}_n)x}\mu_{-,1}(\bar{\eta}_n).
\end{align}
Collect $\sigma_d=\{\eta_n, c_n\}_{n=1}^{2N}$ the scattering data. Now we try to carry out the time evolution of the scattering data. If $q$ also depends on time variable $t$,
we can obtain the functions $a(z)$ and $b(z)$ mentioned above for all times $t\in\mathbb{R}$. Applying $\partial_t$ to \eqref{lax pair} and taking some
standard arguments, such as \cite{DZtok, GPequ}, we know that time dependence of scattering data could be expressed in terms of the following replacement
\begin{align}
    &c(\eta_n)\rightarrow c(t,\eta_n)=c(0, \eta_n)e^{\lambda(\eta_n)(4k^2(\eta_n)+2)t},\\
    &r(z)\rightarrow r(t,z)=r(0,z)e^{\lambda (4k^2+2)t}.
\end{align}
\begin{remark}\rm
    At time $t=0$, the initial function $q(x,0)$ produces $4N$ simple zeros of $a(z,0)$. If $q$ evolves in terms of \eqref{dmkdv},
    then $q(x,t)$ will produce exactly the same $4N$ simple zeros at time $0\neq t\in\mathbb{R}$ for $a(z,t)$. And the scattering data with
    time variable $t$ can be given by
    \begin{equation*}
    \Big{\{}r(z)e^{\lambda (4k^2+2)t}, \{\eta_n, c(\eta_n)e^{\lambda(\eta_n)(4k^2(\eta_n)+2)t}\}_{n=1}^{2N}\Big{\}},
    \end{equation*}
    where $\Big{\{}r(z), \{\eta_n, c_n\}_{n=1}^{2N}\Big{\}}$ are corresponded to initial data $q_0(x)$.
\end{remark}

\subsection{Set up of the Riemann-Hilbert problem}
Define a sectionally meromorphic matrix as follows
\begin{equation}
    M(x,t;z):=\left\{
        \begin{aligned}
        \left(\frac{\mu_{+,1}(x,t;z)}{a(z)}, \mu_{-,2}(x,t;z)\right), \quad z\in \mathbb{C}_{+} \\
        \left(\mu_{-,1}(x,t;z), \frac{\mu_{+,2}(x,t;z)}{\overline{a(\bar{z})}}\right), \quad z\in \mathbb{C}_{-},
        \end{aligned}
        \right.
\end{equation}
which solves the following RH problem.
\begin{RHP}\label{originrhp}
Find a $2\times 2$ matrix-valued function $M(x,t;z)$ such that\\
- $M(z)$ is analytical in $\mathbb{C}\backslash (\Sigma\cup\mathcal{Z})$ and has simple poles in $\mathcal{Z}=\{\eta_n, \bar{\eta}_n\}_{n=1}^{2N}$. \\
- $M(z)=\sigma_1\overline{M(\bar{z})}\sigma_1=\overline{M(-\bar{z})}=\mp z^{-1}M(z^{-1})\sigma_2$. \\
- The non-tangential limits $M_{\pm}(z)=\lim_{s\rightarrow z}M(s), s\in\mathbb{C}_{\pm}$ exist for any $z\in\Sigma$ and
satisfy the jump relation $M_+(z)=M_-(z)V(z)$ where
\begin{equation}\label{rhp0jump}
    V(z)=\begin{pmatrix} 1-|r(z)|^2 & -\overline{r(z)}e^{2it\theta} \\ r(z)e^{-2it\theta} & 1 \end{pmatrix}, \quad z\in\Sigma,
\end{equation}
with $\theta(z)=\lambda(z)\left[\frac{x}{t}+4k^2(z)+2\right]$. \\
- Asymptotic behavior
\begin{align}
    &M(x,t;z)=I+\mathcal{O}(z^{-1}), \quad  z\rightarrow\infty,\\
    &M(x,t;z)=\frac{\sigma_2}{z}+\mathcal{O}(1), \quad  z\rightarrow 0.
\end{align}
- Residue conditions
\begin{align}
&\underset{z=\eta_n}{\rm Res}M(z)=\lim_{z\rightarrow\eta_{n}}M(z)\begin{pmatrix} 0 & 0 \\ c_n e^{-2it\theta(\eta_n)} & 0 \end{pmatrix},\label{rhp0resa}\\
&\underset{z=\bar{\eta}_n}{\rm Res}M(z)=\lim_{z\rightarrow\bar{\eta}_{n}}M(z)\begin{pmatrix} 0 & \bar{c}_ne^{2it\theta(\bar{\eta}_n)} \\ 0 & 0 \end{pmatrix},\label{rhp0resb}
\end{align}
\end{RHP}
The solution $q(x,t)$ of \eqref{dmkdv} can be expressed in terms of the solution $M$ of RH problem \ref{originrhp} via the following proposition.
\begin{proposition}\label{helprecover}Assuming $q-\tanh(x)\in L^{1,2}(\mathbb{R})$ and $q'(x)\in W^{1,1}(\mathbb{R})$,
    we have the following asymptotics of $M(z)$ as $z\rightarrow\infty$ and $z\rightarrow0$:
    \begin{align}
    &\lim_{z\rightarrow\infty}z(M(z)-I)=\left(\begin{array}{cc}
                                               -i\int_x^{\infty}(q^2-1)dx & iq\\
                                               -iq & i\int_x^{\infty}(q^2-1)dx
                                               \end{array}\right),\\
    &\lim_{z\rightarrow0}(M(z)-\frac{\sigma_2}{z})=\left(\begin{array}{cc}
                                               iq &-i\int_x^{\infty}(q^2-1)dx \\
                                               i\int_x^{\infty}(q^2-1)dx&-iq
                                               \end{array}\right).
    \end{align}
    And the solution $q(x,t)$ of \eqref{dmkdv}-\eqref{bdries} is given by
    \begin{equation}\label{recover}
        q(x,t)=-i(M_1)_{12}=-i\lim_{z\rightarrow\infty}(zM)_{12},
    \end{equation}
    where $M_1$ appears in the expansion of $M=I+z^{-1}M_1+\mathcal{O}(z^{-2})$ as $z\rightarrow\infty$.
\end{proposition}
\begin{proof}
    This proposition follows from the third item of Proposition \ref{analydiff}.
\end{proof}

\section{Distribution of saddle points and signature table }\label{disphasepoint}
The exponential term appeared in the jump matrix of RH problem \ref{originrhp} plays a key role in our analysis.
\begin{equation}\label{phasefunc}
e^{\pm 2it\theta}, \quad \theta(z)=\frac{1}{2}\left(z-\frac{1}{z}\right)\left[\frac{x}{t}+2+\left(z+\frac{1}{z}\right)^2\right].
\end{equation}
In this section, we present the analysis on the phase function $\theta(z)$, which include the saddle points (see Figure \ref{figsaddle})
and the signature tables for $e^{2it\theta(z)}$ (see Figure \ref{figtheta}).
Direct calculation shows that:
\begin{align}
    &\Im\theta(z)=\frac{\xi+3}{2}\Im z-\frac{\xi+3}{2|z|^2}\Im\bar{z}+\frac{1}{2}\Im(z^3)-\frac{1}{2|z|^6}\Im(\bar{z}^3),\label{imatheta}\\
    &\Re(2it\theta(z))=-t\left[(\xi+3)\Im z-(\xi+3)|z|^{-2}\Im\bar{z}+\Im(z^3)-|z|^{-6}\Im(\bar{z}^3)\right].
\end{align}
To find the stationary phase points (or saddle points), we need $\theta'(z)$
\begin{equation}
    \theta'(z)=\frac{3}{2}z^2+\frac{\xi+3}{2z^2}+\frac{3}{2z^4}+\frac{\xi+3}{2}.
\end{equation}
\begin{proposition}[Distribution of saddle points]\label{relation of xi}
    Besides two fixed saddle points $i, -i$, there exist four saddle points which satisfy the following properties for
    different $\xi$ (see Figure \ref{figsaddle}):\\
    - For $\xi<-6$, the four saddle points $\xi_j:=\xi_j(\xi)$, $j=1,2,3,4$ are located on the jump contour $\Sigma=\mathbb{R}\backslash\{0\}$.
        Moreover, we have $\xi_4<-1<\xi_3<0<\xi_2<1<\xi_1$ and $\xi_1=\frac{1}{\xi_2}=-\frac{1}{\xi_3}=-\xi_{4}$;\\
    - For $-6<\xi<6$, the four saddle points are away from the coordinate axis (both real and imaginary axis);\\
    - For $\xi>6$, the four saddle points are all located on the imaginary axis.
    Moreover, we have $\Im\xi_1>1>\Im\xi_2>0>\Im\xi_3>-1>\Im\xi_4$ and $\xi_1\xi_2=\xi_3\xi_4=-1$.
\end{proposition}
\begin{proof}
    From $\theta'(z)=0$, we have
    \begin{equation}
       3z^6+(\xi+3)z^4+(\xi+3)z^2+3=0.
    \end{equation}
    Using factorization technique, we obtain
    \begin{equation}
       (1+z^2)\left(3z^4+\xi z^2+3\right)=0.
    \end{equation}
    From the equality, we have two fixed saddle points $i, -i$. And we can solve that
    \begin{equation}
        z^2=-\frac{\xi+\sqrt{\xi^2-36}}{6}, \quad {\rm or}\quad  z^2=-\frac{\xi-\sqrt{\xi^2-36}}{6}.
    \end{equation}
    For $\xi<-6$, both $-\frac{\xi+\sqrt{\xi^2-36}}{6}$ and $-\frac{\xi-\sqrt{\xi^2-36}}{6}$ are greater than zero,
    and four roots are as follows.
    \begin{align}
    \xi_1=\sqrt{-\frac{\xi-\sqrt{\xi^2-36}}{6}}, \quad \xi_4=-\sqrt{-\frac{\xi-\sqrt{\xi^2-36}}{6}}, \label{saddlexi1,4}\\
    \xi_2=\sqrt{-\frac{\xi+\sqrt{\xi^2-36}}{6}}, \quad \xi_3=-\sqrt{-\frac{\xi+\sqrt{\xi^2-36}}{6}}.\label{saddlexi2,3}
    \end{align}
    with relation $\xi_4<-1<\xi_3<0<\xi_2<1<\xi_1$ and $\xi_1=\frac{1}{\xi_2}=-\frac{1}{\xi_3}=-\xi_{4}$.

    For $-6<\xi<6$, the discriminant $\xi^2-36$ is less than zero. We can know that there exist four saddle points
    $\xi_j=\Re(\xi_j)+i\Im(\xi_j)$, where $\Re(\xi_j), \Im(\xi_j)\neq 0$, $j=1,2,3,4$.

    For $\xi>6$, both $-\frac{\xi+\sqrt{\xi^2-36}}{6}$ and $-\frac{\xi-\sqrt{\xi^2-36}}{6}$ are less than zero. And
    four pure imaginary saddle points are as follows
    \begin{align}
    \xi_1=i\sqrt{\frac{\xi+\sqrt{\xi^2-36}}{6}}, \quad \xi_4=-i\sqrt{\frac{\xi+\sqrt{\xi^2-36}}{6}}, \\
    \xi_2=i\sqrt{\frac{\xi-\sqrt{\xi^2-36}}{6}}, \quad \xi_3=-i\sqrt{\frac{\xi-\sqrt{\xi^2-36}}{6}},
    \end{align}
    with $\Im\xi_1>1>\Im\xi_2>0>\Im\xi_3>-1>\Im\xi_4$ and $\xi_1\xi_2=\xi_3\xi_4=-1$.
\end{proof}
\begin{remark}\rm
    We can see that, for example, $\xi_1$, $\xi_4$ defined by \eqref{saddlexi1,4} can't be $\infty$ because of the finite $\xi=\mathcal{O}(1)$.
\end{remark}
According to the Figure \ref{figsaddle} and Figure \ref{figtheta}. We can find that:
for $\xi<-6$, there exist four stationary phase points besides $i,-i$, which are all located on the jump contour $\Sigma$ as shown in Figure \ref{xiaoyu} with signature table shown in Figure \ref{figxi4sp}.
For $-6<\xi<-2$, The distribution of phase points is shown in Figure \ref{jieyu} and signature table is shown in Figure \ref{figxi0sp}.
For $\xi>-2$, there exist four stationary phase points besides $i,-i$.
When $-2<\xi<6$, the four saddle points are away from the coordinate axis (both real and imaginary axis), which is corresponded to Figure \ref{jieyu} and the signature table is shown in Figure \ref{figxi00sp}.
The asymptotics for $-2<\xi<6$ could be seen as a specific case of the asymptotics for $-6<\xi<-2$.
For $\xi>6$, the four saddle points are all distributed on the imaginary axis as shown in Figure \ref{dayu} and the signature table is still shown in Figure \ref{figxi00sp}.
\begin{figure}[h]
	\centering
	\subfigure[]{\includegraphics[width=0.3\linewidth]{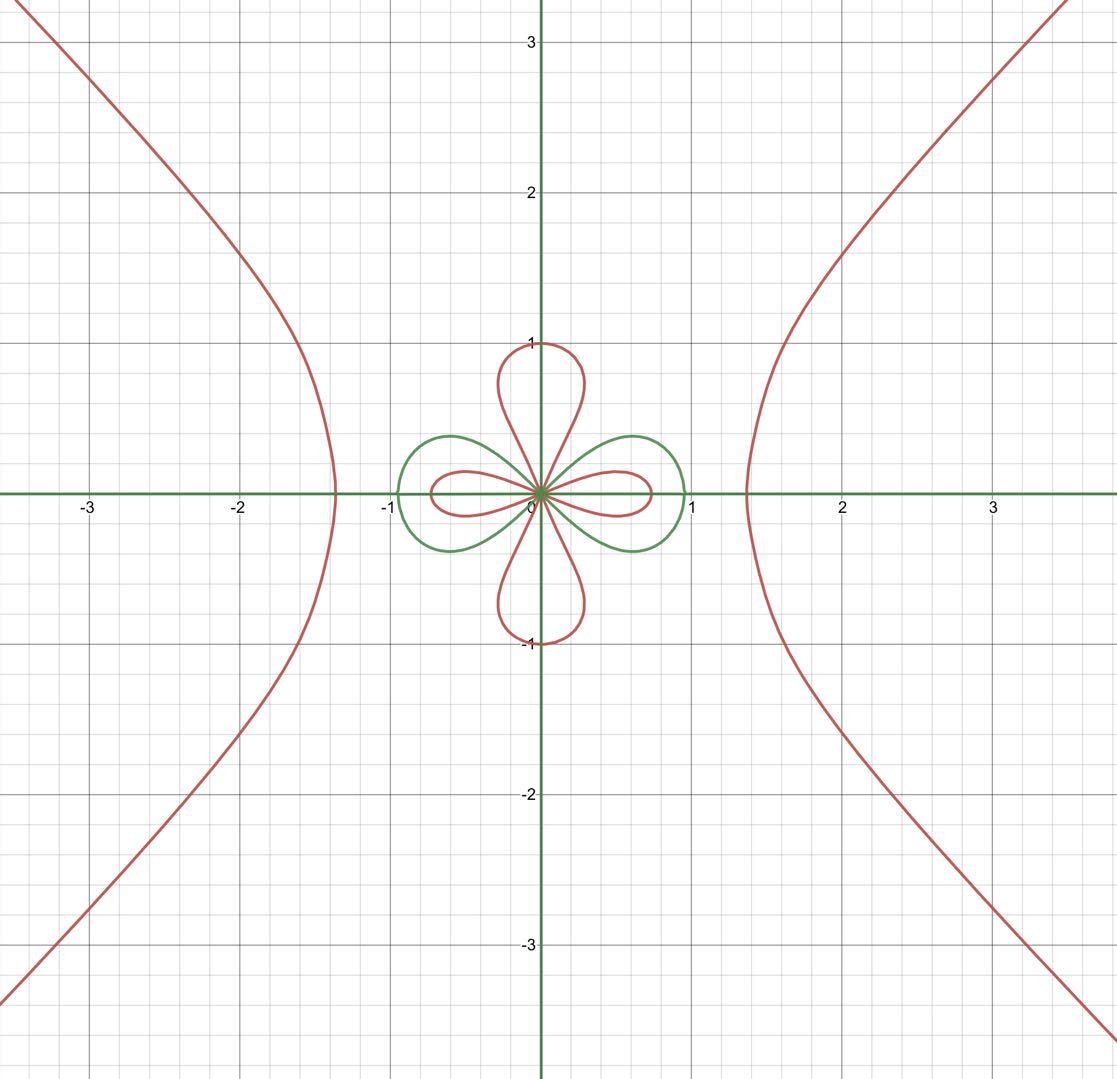}\hspace{0.5cm}\
	\label{xiaoyu}}
	\subfigure[]{\includegraphics[width=0.3\linewidth]{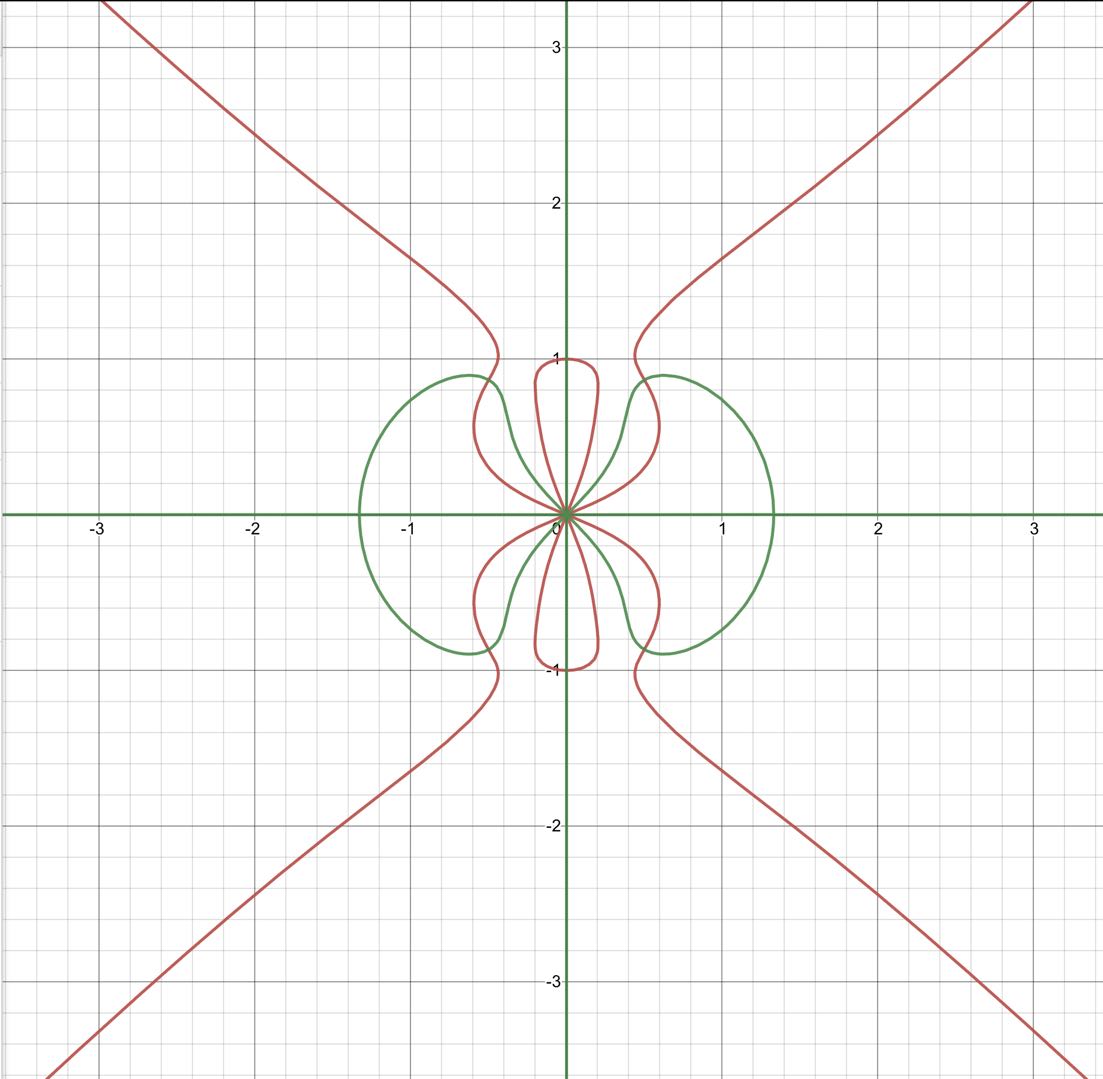}
	\label{jieyu}}	
	\subfigure[]{\includegraphics[width=0.3\linewidth]{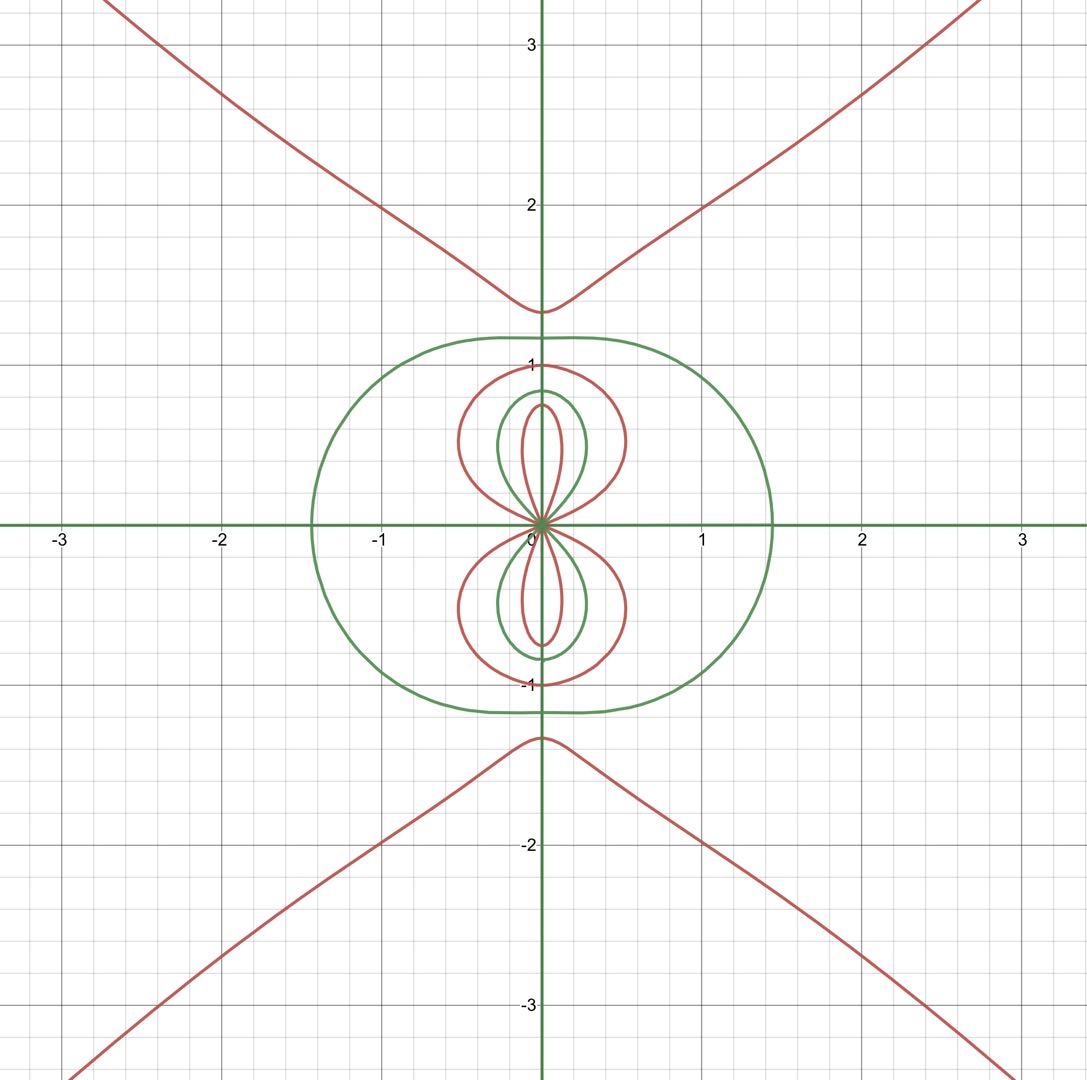}
	\label{dayu}}
	\caption{\footnotesize Plots of the distributions for saddle points:
    $\textbf{(a)}$ $\xi<-6$,
    $\textbf{(b)}$ $-6<\xi<6$,
    $\textbf{(c)}$ $\xi>6$. The red curve represents $\Re \theta'(z)=0$, and the green curve represents $\Im \theta'(z)=0$. The intersection
    points are the saddle points which represent the zeros of $\theta'(z)=0$.}
	\label{figsaddle}
\end{figure}

\begin{figure}[h]
	\centering
	\subfigure[]{\includegraphics[width=0.3\linewidth]{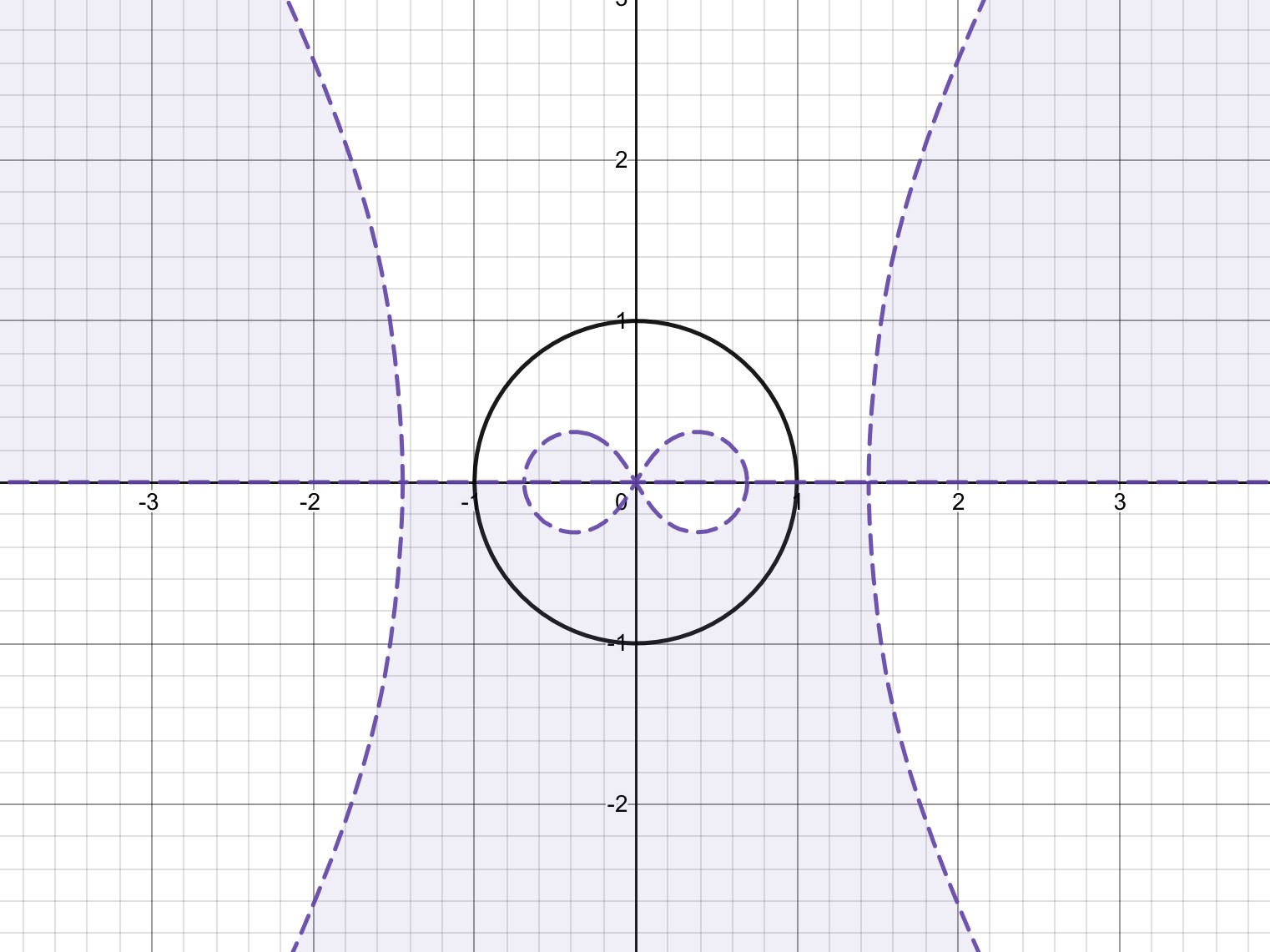}\hspace{0.5cm}
	\label{figxi4sp}}
	\subfigure[]{\includegraphics[width=0.3\linewidth]{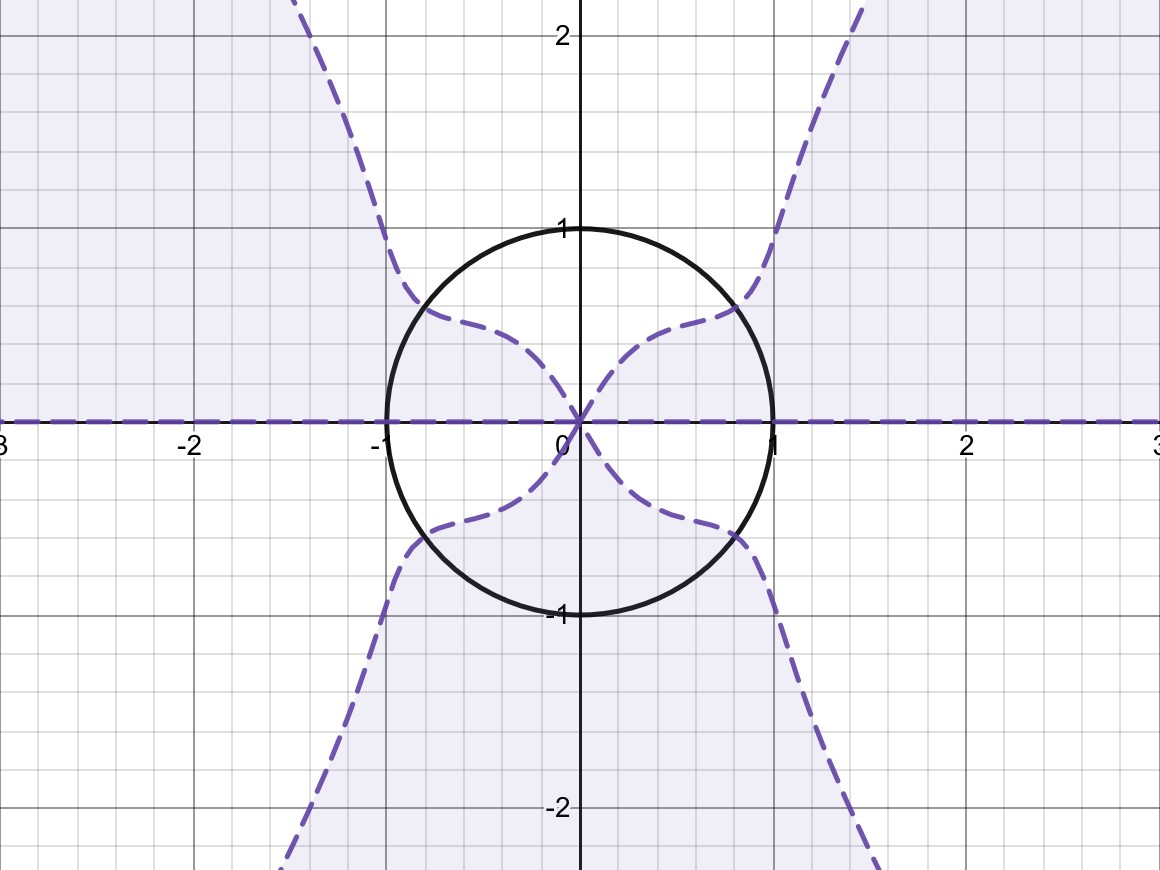}
	\label{figxi0sp}}	
	\subfigure[]{\includegraphics[width=0.3\linewidth]{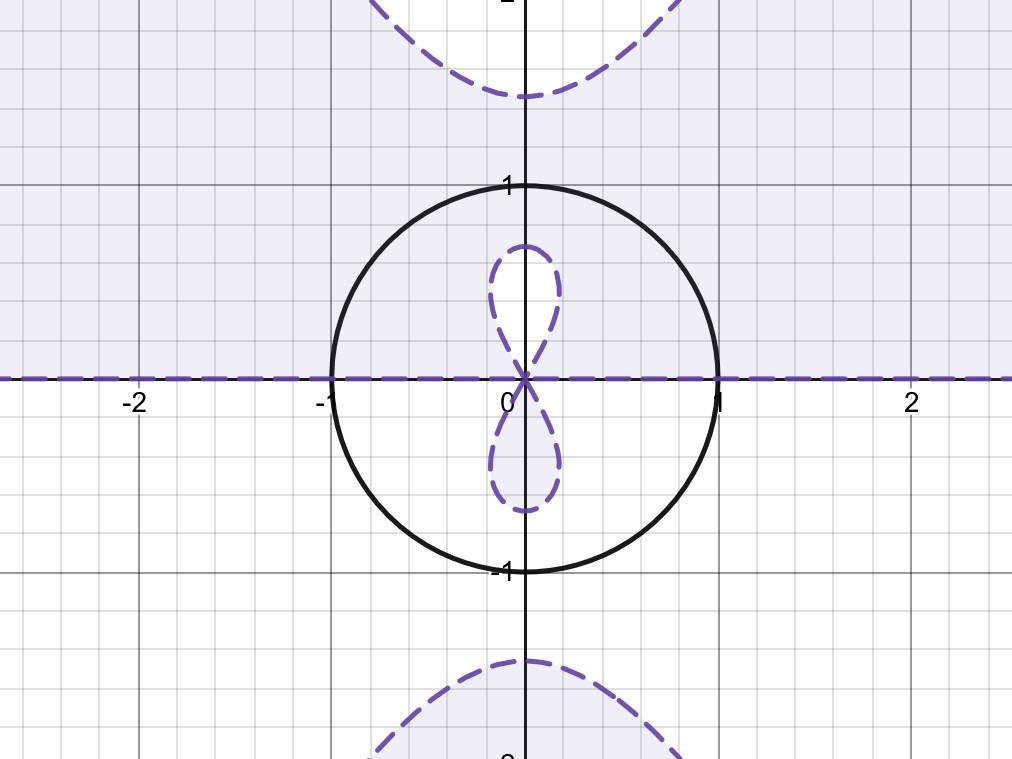}
	\label{figxi00sp}}
	\caption{\footnotesize Plots of the $\Im\theta$ with different $\xi=x/t$:
    $\textbf{(a)}$ $\xi<-6$,
    $\textbf{(b)}$ $-6<\xi<-2$,
    $\textbf{(c)}$ $\xi>6$. The black curve is unit circle. In the purple region, $\Im\theta>0$ ($|e^{2it\theta}|\rightarrow 0$ as $t\rightarrow\infty$),
    and $\Im\theta<0$ ($|e^{-2it\theta}|\rightarrow 0$ as $t\rightarrow\infty$) in the white region. The purple dotted curve represents $\Im\theta=0$. }
	\label{figtheta}
\end{figure}

\section{Asymptotics for $\xi\in\mathcal{R}_L$: left field}\label{sec:proofof1a}
\subsection{Jump matrix  factorizations}\label{1stdeform}
Now we use factorizations of the jump matrix along the real axis to deform the contours onto those on which the oscillatory jump on
the real axis is traded for exponential decay. This step is aided by two well known factorizations of the jump matrix $V(z)$ in \eqref{rhp0jump}:
\begin{align}
    V(z)&=\begin{pmatrix} 1 & -\overline{r(z)}e^{2it\theta} \\ 0 & 1 \end{pmatrix}\begin{pmatrix} 1 & 0 \\ r(z)e^{-2it\theta} & 1 \end{pmatrix}, \quad z\in\tilde{\Gamma},\\
        &=\begin{pmatrix} 1 & 0 \\ \frac{r(z)e^{-2it\theta}}{1-|r(z)|^2} & 1 \end{pmatrix}\left(1-|r(z)|^2\right)^{\sigma_3}\begin{pmatrix} 1 & -\frac{\overline{r(z)}e^{2it\theta}}{1-|r(z)|^2}\\ 0 & 1\end{pmatrix}, \quad z\in\Gamma,
\end{align}
where
\begin{align*}
&\Gamma:=(-\infty, \xi_4)\cup(\xi_3,0)\cup(0,\xi_2)\cup(\xi_1, +\infty),\\
&\tilde{\Gamma}:=(\xi_4,\xi_3)\cup(\xi_2, \xi_1).
\end{align*}
The leftmost term of the factorization can be deformed into $\mathbb{C}_-$, the rightmost term can be deformed into $\mathbb{C}_+$, while
any central terms remain on the real axis. These deformations are useful when they deformed the factors into regions in which the corresponding
off-diagonal exponential terms $e^{\pm 2it\theta}$ are decaying as $t\rightarrow\infty$.

\subsection{First transformation: $M\rightarrow M^{(1)}$}\label{subsec:RLmtom1}
Define the function
\begin{align}
    \delta(z):=\delta(z,\xi)={\rm exp}\left[-\frac{1}{2\pi i}\int_{\Gamma}{\rm log}\left(1-|r(s)|^2\right)\left(\frac{1}{s-z}-\frac{1}{2s}\right)ds\right].
\end{align}
Taking $\nu(z)=-\frac{1}{2\pi}{\rm log}(1-|r(z)|^2)$, then we can express
\begin{equation}\label{delta}
    \delta(z)={\rm exp}\left(-i\int_{\Gamma}\nu(s)\left(\frac{1}{s-z}-\frac{1}{2s}\right)ds\right).
\end{equation}
In the above formulaes, we choose the principal branch of power and logarithm functions.
\begin{proposition}\label{deltapro}The function defined by \eqref{delta} admits following properties:
    \begin{itemize}
        \item[(i)] $\delta(z)$ is analytical for $z\in\mathbb{C}\backslash\Gamma$;
        \item[(ii)] $\delta_{-}(z,\xi)=\delta_{+}(z,\xi)\left(1-|r(z)|^2\right)$, $z\in\Gamma$;
        \item[(iii)] $\delta(z)=\overline{\delta^{-1}(\bar{z})}=\overline{\delta(-\bar{z})}=\delta(z^{-1})^{-1}$;
        \item[(iv)] $\delta(\infty):=\lim_{z\rightarrow\infty}\delta(z)=1$. And $\delta(z)$ is continuous at $z=0$ with
        $\delta(0)=\delta(\infty)=1$;
        \item[(v)] $\delta(z)$ is uniformly bounded in $\mathbb{C}\setminus\mathbb{R}$
        \begin{equation}
            (1-\rho^2)^{1/2}\leqslant|\delta(z)|\leqslant(1-\rho^2)^{-1/2},
        \end{equation}
        where $||r||_{L^{\infty}}\leqslant\rho<1$.
        \item[(vi)] As $z\rightarrow\xi_j$ along any ray $\xi_j+e^{i\phi}\mathbb{R}_{+}$ with $|\phi|<\pi$, we have
        \begin{align}
            &\vert \delta(z)-(z-\xi_j)^{i\epsilon_j\nu(\xi_j)}e^{i\beta(z,\xi_j)} \vert\leqslant c\Vert r\Vert_{H^1}\vert z-\xi_j\vert^{\frac{1}{2}},\\
            &||\beta||_{L^{\infty}}\leqslant\frac{c||r||_{H^{1,0}}}{1-\rho},\label{betaest1}\\
            &|\beta(z,\xi)-\beta(\xi_j,\xi)|<\frac{c||r||_{H^{1,0}}}{1-\rho}|z-\xi_j|^{\frac{1}{2}},\label{betaest2}
        \end{align}
        where
        \begin{equation}
            \beta(z,\xi_j)=\int_{\Gamma}\frac{\nu(s)}{s-z}ds+\epsilon_j\nu(\xi_j){\rm log}(z-\xi_j),
        \end{equation}
    \end{itemize}
\end{proposition}
\begin{proof}
Properties of (i), (iv) can be obtained by simple calculation from the definition of $\delta(z)$.
The jump relation (ii) follows from the Plemelj formulae. As for the property (iii), the symmetry comes from the
symmetry of $r(z)$. We specially point out that the third symmetry follows from the symmetry of $r(z)$ as well as the following equality
\begin{equation}
    {\rm exp}\left(i\int_{\Gamma}\frac{\nu(s)}{s-z}ds\right)={\rm exp}\left[i\int_{\Gamma}\nu(s)\left(\frac{1}{s-z}-\frac{1}{2s}\right)ds\right], \quad {\rm for} \quad z\in\Gamma.
\end{equation}
For the item (v) and item (vi), the analysis is similar to \cite[Lemma 3.1]{Dieng}.
\end{proof}
Furthermore, we can rewrite
\begin{equation}
    \delta(z,\xi)=(z-\xi_j)^{i\nu(\xi_j)}{\rm exp}(i\beta(z,\xi)).
\end{equation}

\begin{remark} \rm
    We notice that all discrete spectrums $\eta_n\in\mathbb{C}_{+}\cap\{z:|z|=1\}$ satisfy $\Im\theta(\eta_n)<0$, all discrete
    spectrums $\bar{\eta}_n\in\mathbb{C}_{-}\cap\{z:|z|=1\}$ satisfy $\Im\theta(\bar{\eta}_n)>0$. Owe to this good property,
    that's why we do not classify the discrete spectrum by $\delta(z,\xi)$, which is different from the $T(z)$ we use in \cite{Zhang&Xu mkdv}.
\end{remark}

Introduce the interpolation functions which can convert the residue conditions \eqref{rhp0resa} and \eqref{rhp0resb} into the jump condition. For all poles
$\eta_j\in\mathcal{Z}$, we define a constant $h$ as follows
\begin{equation}
    h:=\frac{1}{2}\min\left\{\min_{i\neq j}\vert\eta_i-\eta_j\vert, \ \min_{j\in\mathcal{Z}}\vert \Im\eta_j \vert\right\}.
\end{equation}
We can see that the disk $D(\eta_i,h)\cap D(\eta_j,h)=\emptyset$ for $i\neq j$, and  $D(\eta_i,h)\cap\mathbb{R}=\emptyset$. To be brief, we
define a new path
\begin{align}
    \Sigma^{pole}=\bigcup_{n=1}^N\left\{z\in\mathbb{C}: z\in\partial D(\eta_n,h) \ {\rm or} \  z\in\partial D(\bar{\eta}_n,h)   \right\}.
\end{align}
The interpolation function $G(z)$ is introduced by
\begin{align}
    G(z)=\left\{
    \begin{aligned}
    &\begin{pmatrix} 1 & 0 \\ -\frac{c_ne^{-2it\theta(\eta_n)}}{z-\eta_n} & 1\end{pmatrix}, & z\in D(\eta_n,h),\\
    &\begin{pmatrix} 1 & -\frac{\bar{c}_ne^{2it\theta(\bar{\eta}_n)}}{z-\bar{\eta}_n} \\ 0  & 1\end{pmatrix}, & z\in D(\bar{\eta}_n,h),\\
    &I  &elsewhere.
    \end{aligned}
    \right.
\end{align}

By using $\delta(z, \xi)$ and the interpolation function $G(z)$, the new matrix-valued function $M^{(1)}(z)$ is defined by
\begin{equation}\label{trans1}
    M^{(1)}(x,t;z):=M^{(1)}(z)=M(z)G(z)\delta(z)^{\sigma_3},
\end{equation}
which satisfies the following regular RH problem.
\begin{RHP}\label{m1rhp}
Find a $2\times 2$ matrix-valued function $M^{(1)}(x,t;z)$ such that \\
    - $M^{(1)}(z)$ is analytic for $\mathbb{C}\backslash \Sigma^{(1)}$, where $\Sigma^{(1)}=\Sigma\cup\Sigma^{pole}$. \\
    - $M^{(1)}(z)=\sigma_1\overline{M^{(1)}(\bar{z})}\sigma_1=\overline{M^{(1)}(-\bar{z})}=\mp z^{-1}M^{(1)}(z^{-1})\sigma_2$. \\
    - The non-tangential limits $M^{(1)}_{\pm}(z)$ exist for any $z\in\Sigma^{(1)}$ and
    satisfy the jump relation $M_+^{(1)}(z)=M_-^{(1)}(z)V^{(1)}(z)$ where
    \begin{align}\label{rhp1jump}
    V^{(1)}(z)=\left\{
        \begin{aligned}
        &\begin{pmatrix} 1 & -\overline{r(z)}\delta(z)^{-2}e^{2it\theta} \\ 0 & 1 \end{pmatrix}\begin{pmatrix} 1 & 0 \\ r(z)\delta^2(z)e^{-2it\theta} & 1 \end{pmatrix}, & z\in\tilde{\Gamma},\\
        &\begin{pmatrix} 1 & 0 \\ \frac{r(z)\delta_{-}^{2}(z)}{1-|r(z)|^2}e^{-2it\theta} & 1 \end{pmatrix}\begin{pmatrix} 1 & -\frac{\overline{r(z)}\delta_{+}^{-2}(z)}{1-|r(z)|^2}e^{2it\theta}\\ 0 & 1\end{pmatrix}, & z\in\Gamma,\\
        &\begin{pmatrix} 1 & 0 \\ -\frac{c_ne^{-2it\theta(\eta_n)}\delta^{-2}}{z-\eta_n} & 1 \end{pmatrix}, & z\in \partial D(\eta_n,h) \ {\rm oriented \ counterclockwise},\\
        &\begin{pmatrix} 1 & \frac{\bar{c}_ne^{2it\theta(\bar{\eta}_n)}\delta^2(z)}{z-\bar{\eta}_n} \\ 0 & 1 \end{pmatrix}, & z\in \partial D(\bar{\eta}_n,h) \ {\rm oriented \ clockwise}.\\
        \end{aligned}
        \right.
    \end{align}
    - Asymptotic behavior
    \begin{align}
        &M^{(1)}(x,t;z)=I+\mathcal{O}(z^{-1}), \quad  z\rightarrow\infty,\\
        &M^{(1)}(x,t;z)=\frac{\sigma_2}{z}+\mathcal{O}(1), \quad  z\rightarrow 0.
    \end{align}
\end{RHP}

\subsection{Second transformation by opening $\bar{\partial}$ lenses: $M^{(1)}\rightarrow M^{(2)}$}\label{subsec:RLopenlens}
In this subsection, we make continuous extension for the jump matrix $V^{(1)}(z)$ to remove the jump from the real axis in such a
way that the new problem takes advantage of the decay of exp$(\pm 2it\theta)$ for $z\notin\mathbb{R}$.

\subsubsection{Characteristic lines}
The aim of this subsubsection is to denote some characteristic lines which are the jump contours of the RH problem $M^{(2)}$ defined below. To avoid these characteristic lines
intersect with discrete spectrum located on the unit circle, we fix a small enough angle $\theta_0$ which satisfies that the following three conditions.\\
- Let the set
\begin{equation}\label{angle0}
\left\{z\in\mathbb{C}: {\rm tan}\theta_0<\left\lvert\frac{\Im z}{\Re z-\xi_j} \right\rvert \right\}
\end{equation}
do not intersect the set $\mathcal{Z}$,  $j=2,3$ for $\xi\in\mathcal{R}_L$;\\
- The following regions $\Omega_{jk}, j=2,3, k=3,4$ do not intersect discrete spectrums, which implies that
\begin{equation}\label{angle1}
    \Upsilon(\xi)={\rm min}\left\{\theta_0, \frac{\pi}{4} \right\},
\end{equation}
- Recalling Proposition \ref{relation of xi}, we make
\begin{equation}\label{angle2}
d\in\left(0, \frac{\xi_2}{2{\rm cos}\Upsilon}\right), \quad \tilde{d}\in\left(0, \frac{\xi_1-\xi_2}{2{\rm cos}\Upsilon}\right).
\end{equation}
With these conditions, some characteristic lines are given as following items (For convenience, see Figure \ref{fig:regionomega}).
\begin{itemize}
\item [(i)]For an angle $\phi$ satisfies the above conditions \eqref{angle0}, \eqref{angle1} and  \eqref{angle2}, we denote the characteristic lines near saddle points presented by
\begin{align}
    &\Sigma_{j1}=\left\{
        \begin{aligned}
        &\xi_j+e^{i\left[\left(j-1\right)\pi+(-1)^{j-1}\phi\right]}\mathbb{R}_{+}, &j=1, 4,\\
        &\xi_j+e^{i\left[\left(j-1\right)\pi+(-1)^{j-1}\phi\right]}d,  &j=2, 3,
        \end{aligned}
        \right.,
    &\Sigma_{j2}=\left\{
        \begin{aligned}
        &\xi_j+e^{-i\left[\left(j-1\right)\pi+(-1)^{j-1}\phi\right]}\mathbb{R}_{+}, & j=1, 4,\\
        &\xi_j+e^{-i\left[\left(j-1\right)\pi+(-1)^{j-1}\phi\right]}d, & j=2, 3,
        \end{aligned}
        \right.\\
    &\Sigma_{j3}=\left\{
        \begin{aligned}
        &\xi_j+e^{-i\left[j\pi+(-1)^{j}\phi\right]}\tilde{d}, & j=1, 4,\\
        &\xi_j+e^{-i\left[j\pi+(-1)^{j}\phi\right]}\tilde{d}, & j=2, 3,
        \end{aligned}
        \right.
    &\Sigma_{j4}=\left\{
        \begin{aligned}
        &\xi_j+e^{i\left[j\pi+(-1)^{j}\phi\right]}\tilde{d}, &j=1, 4,\\
        &\xi_j+e^{i\left[j\pi+(-1)^{j}\phi\right]}\tilde{d}, &j=2, 3.
        \end{aligned}
        \right.
\end{align}
\item[(ii)]Characteristic lines near $z=0$ are defined as follows
\begin{align}
    &\Sigma_{j1}=\left\{
        \begin{aligned}
        &e^{i\phi}d,\quad j=0^{+},\\
        &e^{i(\pi-\phi)}d,\quad j=0^{-},
        \end{aligned}
        \right.
    &\Sigma_{j2}=\left\{
        \begin{aligned}
        &e^{-i\phi}d,\quad j=0^{+},\\
        &e^{-i(\pi-\phi)}d,\quad j=0^{-},
        \end{aligned}
        \right.
\end{align}
\item[(iii)]Meanwhile, there exist vertical jumps $\Sigma_{j\pm}^{(1/2)}$, $j=1,2,3,4$.
\end{itemize}

The complex plane $\mathbb{C}$ is separated by these contours, which is shown in Figure \ref{phase}, Figure \ref{fig:regionomega}. Denote
\begin{align}
    &\Omega=\left(\bigcup_{j,k=1,2,3,4}\Omega_{jk}\right)\bigcup\left(\bigcup_{k=1,2}\Omega_{0^{\pm}k}\right),\quad \Omega_{\pm}=\mathbb{C}\backslash\Omega,\\
    &\Sigma^{(2)}:=\left(\bigcup_{j,k=1,2,3,4}\Sigma_{jk}\right)\bigcup\left(\bigcup_{k=1,2}\Sigma_{0^{\pm}k}\right)\bigcup\left(\bigcup_{j=1,2,3,4}\Sigma_{j\pm}^{(1/2)}\right)\bigcup\Sigma^{pole}. \label{Sigma^(2)con}
\end{align}
\begin{figure}[htbp]
    \begin{center}
        \begin{tikzpicture}
               \draw[dotted] (2.05,0) arc (0:360:2.05);
               \draw[dotted] (1,0) arc (0.5:360:0.5);
               \draw[dotted] (-1,0) arc (0.5:360:-0.5);
               \draw[->,dotted](-5.5,0)--(5.5,0)node[right]{ \textcolor{black}{$\Re z$}};
               \draw[->](0,-3)--(0,3)node[right]{\textcolor{black}{$\Im z$}};
               \draw(2.5,0)--(2.5,0.1)node[below]{\scriptsize$\frac{\xi_1+\xi_2}{2}$};
               \draw(2.5,0)--(2.5,-0.1);
               \draw(0.5,0)--(0.5,0.1)node[below]{\scriptsize$\frac{\xi_2}{2}$};
               \draw(0.5,0)--(0.5,-0.1);
               \draw(-2.5,0)--(-2.5,0.1)node[below]{\scriptsize$\frac{\xi_3+\xi_4}{2}$};
               \draw(-2.5,0)--(-2.5,-0.1);
               \draw(-0.5,0)--(-0.5,0.1)node[below]{\scriptsize$\frac{\xi_3}{2}$};
               \draw(-0.5,0)--(-0.5,-0.1);
               \coordinate (I) at (0,0);
               \coordinate (A) at (-3.8,0);
               \fill (A) circle (0pt) node[blue,below] {\footnotesize $\xi_4$};
               \coordinate (b) at (-1.2,0);
               \fill (b) circle (0pt) node[blue,below] {\footnotesize $\xi_3$};
               \coordinate (e) at (4.2,0);
               \fill (e) circle (0pt) node[blue,below] {\footnotesize $\xi_1$};
               \coordinate (f) at (1.2,0);
               \fill (f) circle (0pt) node[blue,below] {\footnotesize $\xi_2$};
               \draw [dotted ] (4.5,2.6) to [out=-105,in=90] (4,0);
               \draw [dotted ]  (4.5,-2.6) to [out=105,in=-90] (4,0);
               \draw [dotted ] (-4.5,2.6) to [out=-70,in=90] (-4,0);
               \draw [dotted ] (-4.5,-2.6) to [out=70,in=-95] (-4,0);
               \coordinate (J) at (1.78,0.98);
	           \coordinate (K) at (1.78,-0.98);
	           \coordinate (L) at (-1.78,0.98);
	           \coordinate (M) at (-1.78,-0.98);
               \coordinate (N) at (0, 2.05);
               \coordinate (O) at (0, -2.05);
	           \fill[red] (J) circle (1.2pt) node[left] {\textcolor{black}{$z_n$}};
	           \fill[red] (K) circle (1.2pt) node[left] {\textcolor{black}{$\bar{z}_n$}};
	           \fill[red] (L) circle (1.2pt) node[right] {\textcolor{black}{$-\bar{z}_n$}};
	           \fill[red] (M) circle (1.2pt) node[right] {\textcolor{black}{$-z_n$}};
               \fill[red] (N) circle (1.2pt);
               \node at (0.25,2.3)  {\textcolor{black}{$i$}};
               \fill[red] (O) circle (1.2pt);
                \node at (0.3,-2.2) {\textcolor{black}{$-i$}};
               \coordinate (ke) at (4.8,1.5);
               \fill (ke) circle (0pt) node[below] {\color{blue}$+$};
               \coordinate (k1e) at (4.8,-1.5);
               \fill (k1e) circle (0pt) node[above] {\color{red}$-$};
               \coordinate (le) at (3,1.5);
               \fill (le) circle (0pt) node[below] {\color{red}$-$};
               \coordinate (l1e) at (3,-1.5);
               \fill (l1e) circle (0pt) node[above] {\color{blue}$+$};
               \coordinate (n2) at (0.4,1.5);
               \fill (n2) circle (0pt) node[below] {\color{red}$-$};
               \coordinate (n12) at (0.4,-1.5);
               \fill (n12) circle (0pt) node[above] {\color{blue}$+$};
               \coordinate (m2) at (0.5,0.6 );
               \fill (m2) circle (0pt) node[below] {\footnotesize\color{blue}$+$};
               \coordinate (m12) at (0.5,-0.68);
               \fill (m12) circle (0pt) node[above] {\footnotesize\color{red}$-$};
               \coordinate (k) at (-4.8,1.5);
               \fill (k) circle (0pt) node[below] {\color{blue}$+$};
               \coordinate (k1) at (-4.8,-1.5);
               \fill (k1) circle (0pt) node[above] {\color{red}$-$};
               \coordinate (l) at (-3,1.5);
               \fill (l) circle (0pt) node[below] {\color{red}$-$};
               \coordinate (l1) at (-3,-1.5);
               \fill (l1) circle (0pt) node[above] {\color{blue}$+$};
               \coordinate (n) at (-0.5,0.6);
               \fill (n) circle (0pt) node[below] {\footnotesize\color{blue}$+$};
               \coordinate (n1) at (-0.5,-0.68);
               \fill (n1) circle (0pt) node[above] {\footnotesize\color{red}$-$};
               \coordinate (c) at (-2.05,0);
               \fill[red] (c) circle (1pt) node[below] {\scriptsize$-1$};
               \coordinate (d) at (2.05,0);
               \fill[red] (d) circle (1pt) node[below] {\scriptsize$1$};
               \draw[blue](0,0)--(0.5,0.25);
               \draw[->,blue](0,0)--(0.25,0.125);
               \draw[blue](0.5,0.25)--(1,0);
               \draw[->,blue](0.5,0.25)--(0.75,0.125);
               \draw[red](1,0)--(2.5,0.5);
               \draw[->,red](1,0)--(1.75,0.25);
               \draw[red](2.5,0.5)--(4,0);
               \draw[->,red](2.5,0.5)--(3.25,0.25);
               \draw[blue](4,0)--(5.5, 0.5);
               \draw[->,blue](4,0)--(4.75,0.25);
               \draw[red](0,0)--(0.5,-0.25);
               \draw[->,red](0,0)--(0.25,-0.125);
               \draw[red](0.5,-0.25)--(1,0);
               \draw[->,red](0.5,-0.25)--(0.75,-0.125);
               \draw[blue](1,0)--(2.5,-0.5);
               \draw[->,blue](1,0)--(1.75,-0.25);
               \draw[blue](2.5,-0.5)--(4,0);
               \draw[->,blue](2.5,-0.5)--(3.25,-0.25);
               \draw[red](4,0)--(5.5, -0.5);
               \draw[->,red](4,0)--(4.75,-0.25);
               \draw[blue](0,0)--(-0.5,0.25);
               \draw[->,blue] (-0.5,0.25)--(-0.25,0.125);
               \draw[blue](-0.5,0.25)--(-1,0);
               \draw[<-,blue](-0.75,0.125)--(-1,0);
               \draw[red](-1,0)--(-2.5,0.5);
               \draw[<-,red](-1.75,0.25)--(-2.5,0.5);
               \draw[red](-2.5,0.5)--(-4,0);
               \draw[<-,red](-3.25,0.25)--(-4,0);
               \draw[blue](-4,0)--(-5.5, 0.5);
               \draw[<-,blue](-4.75,0.25)--(-5.5, 0.5);
               \draw[red](0,0)--(-0.5,-0.25);
               \draw[->,red](-0.5,-0.25)--(-0.25,-0.125);
               \draw[red](-0.5,-0.25)--(-1,0);
               \draw[<-,red](-0.75,-0.125)--(-1,0);
               \draw[blue](-1,0)--(-2.5,-0.5);
               \draw[<-,blue](-1.75,-0.25)--(-2.5,-0.5);
               \draw[blue](-2.5,-0.5)--(-4,0);
               \draw[<-,blue](-3.25,-0.25)--(-4,0);
               \draw[red](-4,0)--(-5.5, -0.5);
               \draw[<-,red](-4.75,-0.25)--(-5.5, -0.5);
               \draw[->,red] (0.2,2.05) arc(0:360:0.2);
               \draw[->,red] (1.98,0.98) arc(0:360:0.2);
               \draw[->,red] (-1.58,0.98) arc(0:360:0.2);
               \draw[->,blue] (0.2,-2.05) arc(360:0:0.2);
               \draw[->,blue] (-1.58,-0.98) arc(360:0:0.2);
               \draw[->,blue] (1.98,-0.98) arc(360:0:0.2);
               \end{tikzpicture}
       \caption{\footnotesize There are four stationary phase points $\xi_1,...,\xi_4$ with $\xi_1=-\xi_4=1/\xi_2=-1/\xi_3$ for $\xi\in\mathcal{R}_L$.
       Open jump contour $\mathbb{R}\backslash\{0\}$ such that red and blue lines don't intersect the discrete spectrum on the unite circle $|z|=1$.
       Additionally, the blue ``$+$'' implies that $e^{2it\theta}\rightarrow 0$ as $t\rightarrow+\infty$, on the other side,
       the red ``$-$'' implies that $e^{-2it\theta}\rightarrow 0$ as $t\rightarrow+\infty$.}
       \label{phase}
    \end{center}
   \end{figure}
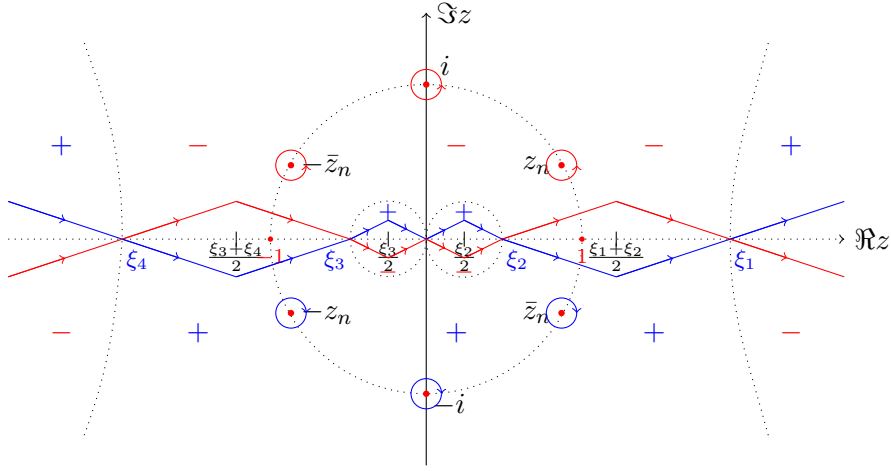

   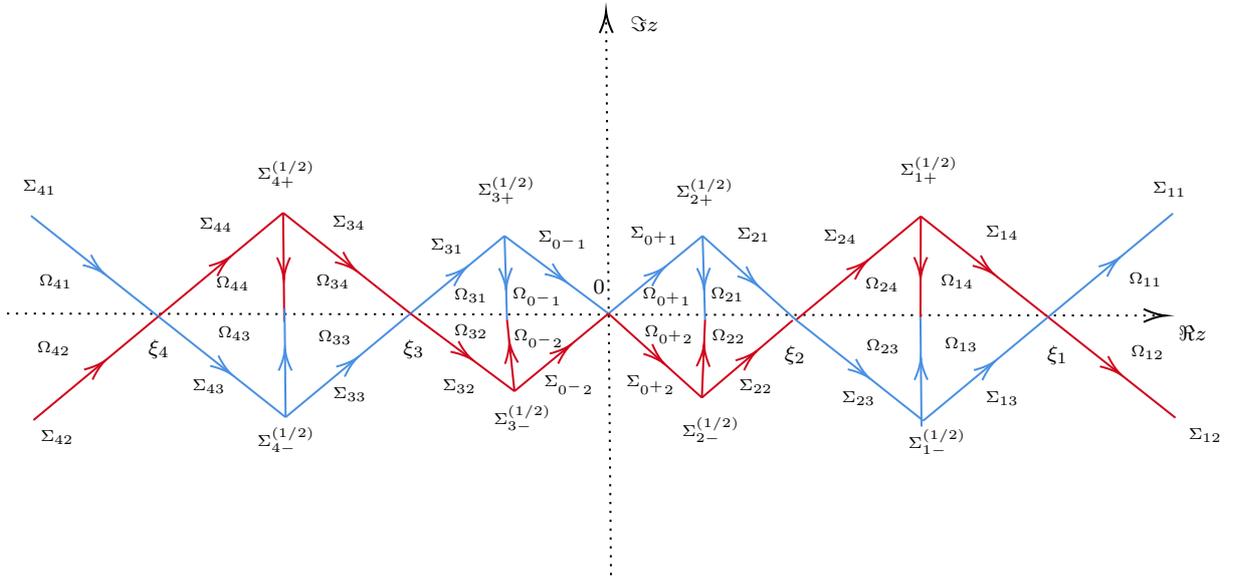
\begin{figure}[htbp]
    \begin{center}
        \tikzset{every picture/.style={line width=0.75pt}} 
        \begin{tikzpicture}[x=0.75pt,y=0.75pt,yscale=-1,xscale=1]
        \draw  [dash pattern={on 0.84pt off 2.51pt}]  (14.43,155.72) -- (590.74,156.75) ;
        \draw [shift={(592.74,156.75)}, rotate = 180.1] [color={rgb, 255:red, 0; green, 0; blue, 0 }  ][line width=0.75]    (10.93,-3.29) .. controls (6.95,-1.4) and (3.31,-0.3) .. (0,0) .. controls (3.31,0.3) and (6.95,1.4) .. (10.93,3.29)   ;
        \draw  [dash pattern={on 0.84pt off 2.51pt}]  (315.94,287.29) -- (313.49,5.61) ;
        \draw [shift={(313.47,3.61)}, rotate = 89.5] [color={rgb, 255:red, 0; green, 0; blue, 0 }  ][line width=0.75]    (10.93,-3.29) .. controls (6.95,-1.4) and (3.31,-0.3) .. (0,0) .. controls (3.31,0.3) and (6.95,1.4) .. (10.93,3.29)   ;
        \draw [color={rgb, 255:red, 74; green, 144; blue, 226 }  ,draw opacity=1 ][fill={rgb, 255:red, 255; green, 255; blue, 255 }  ,fill opacity=1 ]   (534.13,157.49) -- (596.45,105.36) ;
        \draw [shift={(569.89,127.57)}, rotate = 140.09] [color={rgb, 255:red, 74; green, 144; blue, 226 }  ,draw opacity=1 ][line width=0.75]    (10.93,-3.29) .. controls (6.95,-1.4) and (3.31,-0.3) .. (0,0) .. controls (3.31,0.3) and (6.95,1.4) .. (10.93,3.29)   ;
        \draw [color={rgb, 255:red, 208; green, 2; blue, 27 }  ,draw opacity=1 ]   (534.13,157.49) -- (597.68,208.14) ;
        \draw [shift={(570.6,186.55)}, rotate = 218.56] [color={rgb, 255:red, 208; green, 2; blue, 27 }  ,draw opacity=1 ][line width=0.75]    (10.93,-3.29) .. controls (6.95,-1.4) and (3.31,-0.3) .. (0,0) .. controls (3.31,0.3) and (6.95,1.4) .. (10.93,3.29)   ;
        \draw [color={rgb, 255:red, 208; green, 2; blue, 27 }  ,draw opacity=1 ]   (470.58,106.83) -- (534.13,157.49) ;
        \draw [shift={(507.05,135.9)}, rotate = 218.56] [color={rgb, 255:red, 208; green, 2; blue, 27 }  ,draw opacity=1 ][line width=0.75]    (10.93,-3.29) .. controls (6.95,-1.4) and (3.31,-0.3) .. (0,0) .. controls (3.31,0.3) and (6.95,1.4) .. (10.93,3.29)   ;
        \draw [color={rgb, 255:red, 74; green, 144; blue, 226 }  ,draw opacity=1 ]   (471.82,209.61) -- (534.13,157.49) ;
        \draw [shift={(507.58,179.7)}, rotate = 140.09] [color={rgb, 255:red, 74; green, 144; blue, 226 }  ,draw opacity=1 ][line width=0.75]    (10.93,-3.29) .. controls (6.95,-1.4) and (3.31,-0.3) .. (0,0) .. controls (3.31,0.3) and (6.95,1.4) .. (10.93,3.29)   ;
        \draw [color={rgb, 255:red, 208; green, 2; blue, 27 }  ,draw opacity=1 ]   (408.27,158.95) -- (470.58,106.83) ;
        \draw [shift={(444.03,129.04)}, rotate = 140.09] [color={rgb, 255:red, 208; green, 2; blue, 27 }  ,draw opacity=1 ][line width=0.75]    (10.93,-3.29) .. controls (6.95,-1.4) and (3.31,-0.3) .. (0,0) .. controls (3.31,0.3) and (6.95,1.4) .. (10.93,3.29)   ;
        \draw [color={rgb, 255:red, 74; green, 144; blue, 226 }  ,draw opacity=1 ]   (408.27,158.95) -- (471.82,209.61) ;
        \draw [shift={(444.74,188.02)}, rotate = 218.56] [color={rgb, 255:red, 74; green, 144; blue, 226 }  ,draw opacity=1 ][line width=0.75]    (10.93,-3.29) .. controls (6.95,-1.4) and (3.31,-0.3) .. (0,0) .. controls (3.31,0.3) and (6.95,1.4) .. (10.93,3.29)   ;
        \draw [color={rgb, 255:red, 208; green, 2; blue, 27 }  ,draw opacity=1 ]   (361.31,198.01) -- (406.68,159.1) ;
        \draw [shift={(388.55,174.65)}, rotate = 139.38] [color={rgb, 255:red, 208; green, 2; blue, 27 }  ,draw opacity=1 ][line width=0.75]    (10.93,-3.29) .. controls (6.95,-1.4) and (3.31,-0.3) .. (0,0) .. controls (3.31,0.3) and (6.95,1.4) .. (10.93,3.29)   ;
        \draw [color={rgb, 255:red, 74; green, 144; blue, 226 }  ,draw opacity=1 ]   (361.66,116.67) -- (408.27,158.95) ;
        \draw [shift={(389.41,141.84)}, rotate = 222.22] [color={rgb, 255:red, 74; green, 144; blue, 226 }  ,draw opacity=1 ][line width=0.75]    (10.93,-3.29) .. controls (6.95,-1.4) and (3.31,-0.3) .. (0,0) .. controls (3.31,0.3) and (6.95,1.4) .. (10.93,3.29)   ;
        \draw [color={rgb, 255:red, 74; green, 144; blue, 226 }  ,draw opacity=1 ]   (314.71,155.72) -- (361.66,116.67) ;
        \draw [shift={(342.8,132.36)}, rotate = 140.25] [color={rgb, 255:red, 74; green, 144; blue, 226 }  ,draw opacity=1 ][line width=0.75]    (10.93,-3.29) .. controls (6.95,-1.4) and (3.31,-0.3) .. (0,0) .. controls (3.31,0.3) and (6.95,1.4) .. (10.93,3.29)   ;
        \draw [color={rgb, 255:red, 208; green, 2; blue, 27 }  ,draw opacity=1 ]   (314.71,155.72) -- (361.31,198.01) ;
        \draw [shift={(342.45,180.9)}, rotate = 222.22] [color={rgb, 255:red, 208; green, 2; blue, 27 }  ,draw opacity=1 ][line width=0.75]    (10.93,-3.29) .. controls (6.95,-1.4) and (3.31,-0.3) .. (0,0) .. controls (3.31,0.3) and (6.95,1.4) .. (10.93,3.29)   ;
        \draw [color={rgb, 255:red, 74; green, 144; blue, 226 }  ,draw opacity=1 ]   (262.81,116.67) -- (314.71,155.72) ;
        \draw [shift={(293.55,139.8)}, rotate = 216.96] [color={rgb, 255:red, 74; green, 144; blue, 226 }  ,draw opacity=1 ][line width=0.75]    (10.93,-3.29) .. controls (6.95,-1.4) and (3.31,-0.3) .. (0,0) .. controls (3.31,0.3) and (6.95,1.4) .. (10.93,3.29)   ;
        \draw [color={rgb, 255:red, 208; green, 2; blue, 27 }  ,draw opacity=1 ]   (267.75,194.78) -- (314.71,155.72) ;
        \draw [shift={(295.84,171.42)}, rotate = 140.25] [color={rgb, 255:red, 208; green, 2; blue, 27 }  ,draw opacity=1 ][line width=0.75]    (10.93,-3.29) .. controls (6.95,-1.4) and (3.31,-0.3) .. (0,0) .. controls (3.31,0.3) and (6.95,1.4) .. (10.93,3.29)   ;
        \draw [color={rgb, 255:red, 74; green, 144; blue, 226 }  ,draw opacity=1 ]   (215.85,155.72) -- (262.81,116.67) ;
        \draw [shift={(243.94,132.36)}, rotate = 140.25] [color={rgb, 255:red, 74; green, 144; blue, 226 }  ,draw opacity=1 ][line width=0.75]    (10.93,-3.29) .. controls (6.95,-1.4) and (3.31,-0.3) .. (0,0) .. controls (3.31,0.3) and (6.95,1.4) .. (10.93,3.29)   ;
        \draw [color={rgb, 255:red, 208; green, 2; blue, 27 }  ,draw opacity=1 ]   (215.85,155.72) -- (267.75,194.78) ;
        \draw [shift={(246.59,178.86)}, rotate = 216.96] [color={rgb, 255:red, 208; green, 2; blue, 27 }  ,draw opacity=1 ][line width=0.75]    (10.93,-3.29) .. controls (6.95,-1.4) and (3.31,-0.3) .. (0,0) .. controls (3.31,0.3) and (6.95,1.4) .. (10.93,3.29)   ;
        \draw [color={rgb, 255:red, 208; green, 2; blue, 27 }  ,draw opacity=1 ]   (152.3,105.07) -- (215.85,155.72) ;
        \draw [shift={(188.77,134.14)}, rotate = 218.56] [color={rgb, 255:red, 208; green, 2; blue, 27 }  ,draw opacity=1 ][line width=0.75]    (10.93,-3.29) .. controls (6.95,-1.4) and (3.31,-0.3) .. (0,0) .. controls (3.31,0.3) and (6.95,1.4) .. (10.93,3.29)   ;
        \draw [color={rgb, 255:red, 74; green, 144; blue, 226 }  ,draw opacity=1 ]   (153.53,207.85) -- (215.85,155.72) ;
        \draw [shift={(189.29,177.94)}, rotate = 140.09] [color={rgb, 255:red, 74; green, 144; blue, 226 }  ,draw opacity=1 ][line width=0.75]    (10.93,-3.29) .. controls (6.95,-1.4) and (3.31,-0.3) .. (0,0) .. controls (3.31,0.3) and (6.95,1.4) .. (10.93,3.29)   ;
        \draw [color={rgb, 255:red, 208; green, 2; blue, 27 }  ,draw opacity=1 ]   (89.98,157.19) -- (152.3,105.07) ;
        \draw [shift={(125.74,127.28)}, rotate = 140.09] [color={rgb, 255:red, 208; green, 2; blue, 27 }  ,draw opacity=1 ][line width=0.75]    (10.93,-3.29) .. controls (6.95,-1.4) and (3.31,-0.3) .. (0,0) .. controls (3.31,0.3) and (6.95,1.4) .. (10.93,3.29)   ;
        \draw [color={rgb, 255:red, 74; green, 144; blue, 226 }  ,draw opacity=1 ]   (89.98,157.19) -- (153.53,207.85) ;
        \draw [shift={(126.45,186.26)}, rotate = 218.56] [color={rgb, 255:red, 74; green, 144; blue, 226 }  ,draw opacity=1 ][line width=0.75]    (10.93,-3.29) .. controls (6.95,-1.4) and (3.31,-0.3) .. (0,0) .. controls (3.31,0.3) and (6.95,1.4) .. (10.93,3.29)   ;
        \draw [color={rgb, 255:red, 74; green, 144; blue, 226 }  ,draw opacity=1 ]   (26.43,106.54) -- (89.98,157.19) ;
        \draw [shift={(62.9,135.6)}, rotate = 218.56] [color={rgb, 255:red, 74; green, 144; blue, 226 }  ,draw opacity=1 ][line width=0.75]    (10.93,-3.29) .. controls (6.95,-1.4) and (3.31,-0.3) .. (0,0) .. controls (3.31,0.3) and (6.95,1.4) .. (10.93,3.29)   ;
        \draw [color={rgb, 255:red, 208; green, 2; blue, 27 }  ,draw opacity=1 ]   (27.67,209.32) -- (89.98,157.19) ;
        \draw [shift={(63.43,179.41)}, rotate = 140.09] [color={rgb, 255:red, 208; green, 2; blue, 27 }  ,draw opacity=1 ][line width=0.75]    (10.93,-3.29) .. controls (6.95,-1.4) and (3.31,-0.3) .. (0,0) .. controls (3.31,0.3) and (6.95,1.4) .. (10.93,3.29)   ;
        \draw [color={rgb, 255:red, 208; green, 2; blue, 27 }  ,draw opacity=1 ]   (152.3,105.07) -- (152.83,159.84) ;
        \draw [shift={(152.62,138.45)}, rotate = 269.45] [color={rgb, 255:red, 208; green, 2; blue, 27 }  ,draw opacity=1 ][line width=0.75]    (10.93,-3.29) .. controls (6.95,-1.4) and (3.31,-0.3) .. (0,0) .. controls (3.31,0.3) and (6.95,1.4) .. (10.93,3.29)   ;
        \draw [color={rgb, 255:red, 74; green, 144; blue, 226 }  ,draw opacity=1 ]   (153,153.08) -- (153.53,207.85) ;
        \draw [shift={(153.2,173.47)}, rotate = 89.45] [color={rgb, 255:red, 74; green, 144; blue, 226 }  ,draw opacity=1 ][line width=0.75]    (10.93,-3.29) .. controls (6.95,-1.4) and (3.31,-0.3) .. (0,0) .. controls (3.31,0.3) and (6.95,1.4) .. (10.93,3.29)   ;
        \draw [color={rgb, 255:red, 74; green, 144; blue, 226 }  ,draw opacity=1 ]   (262.81,116.67) -- (264.04,158.81) ;
        \draw [shift={(263.6,143.74)}, rotate = 268.32] [color={rgb, 255:red, 74; green, 144; blue, 226 }  ,draw opacity=1 ][line width=0.75]    (10.93,-3.29) .. controls (6.95,-1.4) and (3.31,-0.3) .. (0,0) .. controls (3.31,0.3) and (6.95,1.4) .. (10.93,3.29)   ;
        \draw [color={rgb, 255:red, 208; green, 2; blue, 27 }  ,draw opacity=1 ]   (264.04,158.81) -- (267.75,194.78) ;
        \draw [shift={(265.18,169.83)}, rotate = 84.12] [color={rgb, 255:red, 208; green, 2; blue, 27 }  ,draw opacity=1 ][line width=0.75]    (10.93,-3.29) .. controls (6.95,-1.4) and (3.31,-0.3) .. (0,0) .. controls (3.31,0.3) and (6.95,1.4) .. (10.93,3.29)   ;
        \draw [color={rgb, 255:red, 74; green, 144; blue, 226 }  ,draw opacity=1 ]   (361.66,116.67) -- (362.9,158.81) ;
        \draw [shift={(362.46,143.74)}, rotate = 268.32] [color={rgb, 255:red, 74; green, 144; blue, 226 }  ,draw opacity=1 ][line width=0.75]    (10.93,-3.29) .. controls (6.95,-1.4) and (3.31,-0.3) .. (0,0) .. controls (3.31,0.3) and (6.95,1.4) .. (10.93,3.29)   ;
        \draw [color={rgb, 255:red, 208; green, 2; blue, 27 }  ,draw opacity=1 ]   (362.9,158.81) -- (361.31,198.01) ;
        \draw [shift={(362.39,171.42)}, rotate = 92.32] [color={rgb, 255:red, 208; green, 2; blue, 27 }  ,draw opacity=1 ][line width=0.75]    (10.93,-3.29) .. controls (6.95,-1.4) and (3.31,-0.3) .. (0,0) .. controls (3.31,0.3) and (6.95,1.4) .. (10.93,3.29)   ;
        \draw [color={rgb, 255:red, 208; green, 2; blue, 27 }  ,draw opacity=1 ]   (470.58,106.83) -- (470.41,157.78) ;
        \draw [shift={(470.47,138.3)}, rotate = 270.2] [color={rgb, 255:red, 208; green, 2; blue, 27 }  ,draw opacity=1 ][line width=0.75]    (10.93,-3.29) .. controls (6.95,-1.4) and (3.31,-0.3) .. (0,0) .. controls (3.31,0.3) and (6.95,1.4) .. (10.93,3.29)   ;
        \draw [color={rgb, 255:red, 74; green, 144; blue, 226 }  ,draw opacity=1 ]   (470.41,157.78) -- (470.94,212.55) ;
        \draw [shift={(470.6,178.16)}, rotate = 89.45] [color={rgb, 255:red, 74; green, 144; blue, 226 }  ,draw opacity=1 ][line width=0.75]    (10.93,-3.29) .. controls (6.95,-1.4) and (3.31,-0.3) .. (0,0) .. controls (3.31,0.3) and (6.95,1.4) .. (10.93,3.29)   ;
        \draw (324.3,4.88) node [anchor=north west][inner sep=0.75pt]  [font=\scriptsize]  {$\Im z$};
        \draw (597.81,160.38) node [anchor=north west][inner sep=0.75pt]  [font=\scriptsize]  {$\Re z$};
        \draw (532.19,170.43) node [anchor=north west][inner sep=0.75pt]  [font=\scriptsize]  {$\xi _{1}$};
        \draw (401.21,170.43) node [anchor=north west][inner sep=0.75pt]  [font=\scriptsize]  {$\xi _{2}$};
        \draw (210.91,167.35) node [anchor=north west][inner sep=0.75pt]  [font=\scriptsize]  {$\xi _{3}$};
        \draw (83.63,167.35) node [anchor=north west][inner sep=0.75pt]  [font=\scriptsize]  {$\xi _{4}$};
        \draw (305.71,137.47) node [anchor=north west][inner sep=0.75pt]  [font=\scriptsize]  {$0$};
        \draw (584.92,88.19) node [anchor=north west][inner sep=0.75pt]  [font=\tiny]  {$\Sigma _{11}$};
        \draw (602.75,211.82) node [anchor=north west][inner sep=0.75pt]  [font=\tiny]  {$\Sigma _{12}$};
        \draw (501.42,193.32) node [anchor=north west][inner sep=0.75pt]  [font=\tiny]  {$\Sigma _{13}$};
        \draw (501.42,110.07) node [anchor=north west][inner sep=0.75pt]  [font=\tiny]  {$\Sigma _{14}$};
        \draw (377.32,110.8) node [anchor=north west][inner sep=0.75pt]  [font=\tiny]  {$\Sigma _{21}$};
        \draw (378.38,187.29) node [anchor=north west][inner sep=0.75pt]  [font=\tiny]  {$\Sigma _{22}$};
        \draw (429.75,193.32) node [anchor=north west][inner sep=0.75pt]  [font=\tiny]  {$\Sigma _{23}$};
        \draw (420.51,111.9) node [anchor=north west][inner sep=0.75pt]  [font=\tiny]  {$\Sigma _{24}$};
        \draw (324.01,110.44) node [anchor=north west][inner sep=0.75pt]  [font=\tiny]  {$\Sigma _{0^{+}1}$};
        \draw (322.01,185.44) node [anchor=north west][inner sep=0.75pt]  [font=\tiny]  {$\Sigma _{0^+2}$};
        \draw (278.01,112.44) node [anchor=north west][inner sep=0.75pt]  [font=\tiny]  {$\Sigma _{0^-1}$};
        \draw (281.01,185.44) node [anchor=north west][inner sep=0.75pt]  [font=\tiny]  {$\Sigma _{0^-2}$};
        \draw (224.32,116.8) node [anchor=north west][inner sep=0.75pt]  [font=\tiny]  {$\Sigma _{31}$};
        \draw (230.38,187.29) node [anchor=north west][inner sep=0.75pt]  [font=\tiny]  {$\Sigma _{32}$};
        \draw (175.75,191.32) node [anchor=north west][inner sep=0.75pt]  [font=\tiny]  {$\Sigma _{33}$};
        \draw (175.32,104.8) node [anchor=north west][inner sep=0.75pt]  [font=\tiny]  {$\Sigma _{34}$};
        \draw (108.92,106.19) node [anchor=north west][inner sep=0.75pt]  [font=\tiny]  {$\Sigma _{44}$};
        \draw (20.92,87.19) node [anchor=north west][inner sep=0.75pt]  [font=\tiny]  {$\Sigma _{41}$};
        \draw (29.67,212.72) node [anchor=north west][inner sep=0.75pt]  [font=\tiny]  {$\Sigma _{42}$};
        \draw (105.42,187.32) node [anchor=north west][inner sep=0.75pt]  [font=\tiny]  {$\Sigma _{43}$};
        \draw (573,133.4) node [anchor=north west][inner sep=0.75pt]  [font=\tiny]  {$\Omega _{11}$};
        \draw (574,170.4) node [anchor=north west][inner sep=0.75pt]  [font=\tiny]  {$\Omega _{12}$};
        \draw (481,166.4) node [anchor=north west][inner sep=0.75pt]  [font=\tiny]  {$\Omega _{13}$};
        \draw (479,134.4) node [anchor=north west][inner sep=0.75pt]  [font=\tiny]  {$\Omega _{14}$};
        \draw (441.43,136.29) node [anchor=north west][inner sep=0.75pt]  [font=\tiny]  {$\Omega _{24}$};
        \draw (442,167.4) node [anchor=north west][inner sep=0.75pt]  [font=\tiny]  {$\Omega _{23}$};
        \draw (364.28,141.14) node [anchor=north west][inner sep=0.75pt]  [font=\tiny]  {$\Omega _{21}$};
        \draw (364.9,162.21) node [anchor=north west][inner sep=0.75pt]  [font=\tiny]  {$\Omega _{22}$};
        \draw (330.28,141.14) node [anchor=north west][inner sep=0.75pt]  [font=\tiny]  {$\Omega _{0^+1}$};
        \draw (331.28,160.14) node [anchor=north west][inner sep=0.75pt]  [font=\tiny]  {$\Omega _{0^+2}$};
        \draw (265.42,141.14) node [anchor=north west][inner sep=0.75pt]  [font=\tiny]  {$\Omega _{0^-1}$};
        \draw (266.04,162.21) node [anchor=north west][inner sep=0.75pt]  [font=\tiny]  {$\Omega _{0^-2}$};
        \draw (236.33,141.6) node [anchor=north west][inner sep=0.75pt]  [font=\tiny]  {$\Omega _{31}$};
        \draw (236.33,159.6) node [anchor=north west][inner sep=0.75pt]  [font=\tiny]  {$\Omega _{32}$};
        \draw (167.33,134.6) node [anchor=north west][inner sep=0.75pt]  [font=\tiny]  {$\Omega _{34}$};
        \draw (168.33,162.6) node [anchor=north west][inner sep=0.75pt]  [font=\tiny]  {$\Omega _{33}$};
        \draw (117.33,135.6) node [anchor=north west][inner sep=0.75pt]  [font=\tiny]  {$\Omega _{44}$};
        \draw (118.33,160.6) node [anchor=north west][inner sep=0.75pt]  [font=\tiny]  {$\Omega _{43}$};
        \draw (29,134.4) node [anchor=north west][inner sep=0.75pt]  [font=\tiny]  {$\Omega _{41}$};
        \draw (28,168.4) node [anchor=north west][inner sep=0.75pt]  [font=\tiny]  {$\Omega _{42}$};
        \draw (458.92,75.19) node [anchor=north west][inner sep=0.75pt]  [font=\tiny]  {$\Sigma _{1+}^{( 1/2)}$};
        \draw (462.92,212.19) node [anchor=north west][inner sep=0.75pt]  [font=\tiny]  {$\Sigma _{1-}^{( 1/2)}$};
        \draw (346.92,86.19) node [anchor=north west][inner sep=0.75pt]  [font=\tiny]  {$\Sigma _{2+}^{( 1/2)}$};
        \draw (349.92,206.19) node [anchor=north west][inner sep=0.75pt]  [font=\tiny]  {$\Sigma _{2-}^{( 1/2)}$};
        \draw (247.92,85.19) node [anchor=north west][inner sep=0.75pt]  [font=\tiny]  {$\Sigma _{3+}^{( 1/2)}$};
        \draw (255.92,200.19) node [anchor=north west][inner sep=0.75pt]  [font=\tiny]  {$\Sigma _{3-}^{( 1/2)}$};
        \draw (137.92,77.19) node [anchor=north west][inner sep=0.75pt]  [font=\tiny]  {$\Sigma _{4+}^{( 1/2)}$};
        \draw (137.92,211.19) node [anchor=north west][inner sep=0.75pt]  [font=\tiny]  {$\Sigma _{4-}^{( 1/2)}$};
        \end{tikzpicture}
	\caption{\footnotesize $\Sigma_{jk}$ separate complex $\mathbb{C}$ into some regions denoted by $\Omega_{jk}$.}
    \label{fig:regionomega}
\end{center}
\end{figure}

\subsubsection{Some estimations for $\Im\theta(z)$}
In this subsubsection, we give some estimations for $\Im\theta(z)$ in different regions.
\begin{proposition}[near $z=0$]\label{lemz=0}For a fixed small angle $\phi$ which satisfies \eqref{angle0}, \eqref{angle1} and \eqref{angle2}, the imaginary part
of phase function $\theta(z)$ defined by \eqref{phasefunc} has following estimations:
\begin{align}
    &\Im\theta(z)\geqslant c|{\rm sin}\phi|\sqrt{\alpha}, \quad {\rm as} \quad z\in\Omega_{0^{\pm}1},\\
    &\Im\theta(z)\leqslant -c|{\rm sin}\phi|\sqrt{\alpha}, \quad {\rm as} \quad z\in\Omega_{0^{\pm}2},
\end{align}
where
\begin{equation}
    c=c(\xi)>0, \quad \alpha=3-\frac{\xi+3}{1+2{\rm cos}(2\phi)}.
\end{equation}
\end{proposition}
\begin{proof}
    We present the details for $z\in\Omega_{0^{+}1}$, the others are similar.
    Taking $z=le^{i\phi}$,  we can rewrite the \eqref{imatheta} as
    \begin{equation}\label{imthetaxingtai1}
        \Im\theta(z)=\frac{1}{2}F(l){\rm sin}\phi\left[\xi-6{\rm cos}(2\phi)+(2{\rm cos}(2\phi)+1)F^2(l)\right],
    \end{equation}
    where $F(l)=l+l^{-1}\geqslant 2$. Firstly we calculate the critical situation $\Im\theta(z)=0$. Taking \eqref{imthetaxingtai1},
    $F(l)\geqslant2$ as well as sin$\phi>0$, we have
    \begin{equation}
        \xi-6{\rm cos}\left(2\phi\right)+\left[2{\rm cos}\left(2\phi\right)+1\right]F^2(l)=0.
    \end{equation}
    Thus
    \begin{equation}
        F^2(l)=3-\frac{\xi+3}{1+2{\rm cos}(2\phi)}=:\alpha>4.
    \end{equation}
    Moreover, by $F(l)=\sqrt{\alpha}$, we have $l^2-\sqrt{\alpha}l+1=0$. Solving this quadratic equation, we obtain two roots
    \begin{equation}\label{todopara}
        l_1=\frac{\sqrt{\alpha}-\sqrt{\alpha-4}}{2}<l_2=\frac{\sqrt{\alpha}+\sqrt{\alpha-4}}{2}.
    \end{equation}
    We claim that: $\Im\theta(z)>0$ as $l<l_1$ (corresponding to $z\in\Omega_{0^{+}1}$). It's easy to check
    that $h(l):=l^2-\sqrt{\alpha}l+1$ is monotonically increasing on the $(\sqrt{\alpha}/2, +\infty)$, while
    monotonically decreasing  on the $(-\infty, \sqrt{\alpha}/2)$. Since $l_1<\sqrt{\alpha}/2$, $h(l)$ is monotonically decreasing
    on the $(0, l_1)$.
    Thus we have $h(l)>h(l_1)=0$ and $F(l)>\sqrt\alpha$, which implies that
    \begin{equation}
        \Im\theta(z)>\frac{1}{2}\sqrt\alpha{\rm sin}\phi\left[\xi-6{\rm cos}(2\phi)+(2{\rm cos}(2\phi)+1)\alpha\right]=0.
    \end{equation}
    Thus we bring this proof to an end.
\end{proof}
\begin{corollary}\label{corz=0}
    $\Im\theta(z)$ has following evaluation for $z=le^{i\phi}:=u_0+iv$
    \begin{align}
        &\Im\theta(z)\geqslant cv, \quad {\rm for} \quad z\in\Omega_{0^{\pm}1},\\
        &\Im\theta(z)\leqslant -cv, \quad {\rm for} \quad z\in\Omega_{0^{\pm}2}.
    \end{align}
    where $c=c(\xi)>0$.
\end{corollary}

\begin{proposition}[Near $z=\xi_j, j=2,3$]\label{lemz=xi2} For a fixed small angle $\phi$ (the same in Proposition \ref{lemz=0}) which
satisfies \eqref{angle0}, \eqref{angle1} and \eqref{angle2}, the imaginary part of phase function $\theta(z)$ defined by \eqref{phasefunc} has following estimations:
\begin{align}
    &\Im\theta(z)\leqslant -c\left(1+|z|^{-2}\right)v^{2}, \quad z\in\Omega_{jk},\quad j=2,3, \quad k=2,4\\
    &\Im\theta(z)\geqslant c\left(1+|z|^{-2}\right)v^{2}, \quad z\in\Omega_{jk},\quad j=2,3, \quad k=1,3,
\end{align}
where $c=c(\xi)>0$.
\end{proposition}
\begin{proof}
    Taking $z\in\Omega_{24}$ as an example, the proof for the other regions is similar. Denoting $z=\xi_2+le^{i\phi}:=\xi_2+u_2+iv$,
we can rewrite \eqref{imatheta} as
    \begin{align}\label{imthetaxingtai2}
        \Im\theta(z)&=\frac{v}{2}\left(1+|z|^{-2}\right)\left[\xi+3+3(|z|^2+|z|^{-2}-1)-4v^{2}(1+|z|^{-4}-|z|^{-2})\right]\nonumber\\
        &\leqslant c\left(1+|z|^{-2}\right)\left[\xi+3+3(|z|^2+|z|^{-2}-1)-4v^{2}(1+|z|^{-4}-|z|^{-2})\right],
    \end{align}
    the second step we used $v\leqslant\frac{\xi_1-\xi_2}{2}$.

    Consider
    \begin{equation}
        h\left(|z|^2\right):=\left[\xi+3+3(|z|^2+|z|^{-2}-1)-4v^{2}(1+|z|^{-4}-|z|^{-2})\right].
    \end{equation}
    Taking $\tau=|z|^2\in\left(\xi_2^2,\xi_1^2\right)$, we obtain that
    \begin{equation}
        h(\tau)=3(\tau+\tau^{-1}-1)-4v^2(1+\tau^{-2}-\tau^{-1})+\xi+3.
    \end{equation}
    It is not difficult to verify that $h'(\tau)<0$ for $\tau\in(\xi_2^2, \xi_1^2)$, thus
    \begin{align}
        h(\tau)&\leqslant h(\xi_2^2)\nonumber\\
        &=3\left(\xi_2^2+\xi_2^{-2}-1\right)-4v^2\left(1+\xi_2^{-4}-\xi_2^{-2}\right)+\xi+3\nonumber\\
        &\overset{\xi_2=1/\xi_1}{=}3\left(\xi_2^2+\xi_1^2-1\right)-4v^2\left(1+\xi_1^4-\xi_1^2\right)+\xi+3.
    \end{align}
    Since $\xi_1$ is the saddle point, we have $\theta'(\xi_1)=0$. Using $\xi_1\xi_2=1$ again, we can obtain the following
    relation from $\theta'(\xi_1)=0$ such that
    \begin{equation}\label{luck}
        \xi+3=\frac{-3\left(\xi_1^2+\xi_2^4\right)}{1+\xi_2^2}.
    \end{equation}
    With \eqref{luck}, we are lucky enough to find that
    \begin{equation}
        3\left(\xi_2^2+\xi_1^2-1\right)+\xi+3=0.
    \end{equation}
    Then we obtain
    \begin{equation}
        h(\tau)\leqslant-4v^2\left(1+\xi_1^4-\xi_1^2\right).
    \end{equation}
    As a consequence,
    \begin{equation}
    \Im\theta(z)\leqslant -cv^{2} \left(1+|z|^{-2}\right)\left(1+\xi_1^4-\xi_1^2\right)\leqslant-c(\xi)v^2\left(1+|z|^{-2}\right)<0.
    \end{equation}
\end{proof}

\begin{proposition}[Near $z=\xi_j, j=1,4$]\label{lemz=xi1}  For a fixed small angle $\phi$ (the same in Proposition \ref{lemz=0}) which
satisfies \eqref{angle0}, \eqref{angle1} and \eqref{angle2}, the imaginary part of phase function $\theta(z)$ defined by \eqref{phasefunc} admits following estimations:
    \begin{align}
        &\Im\theta(z)\geqslant cv\vert \Re z-\xi_j \vert, \quad z\in\Omega_{jk},\quad j=1,4, \quad k=1,3\\
        &\Im\theta(z)\leqslant -cv\vert \Re z-\xi_j \vert, \quad z\in\Omega_{jk},\quad j=1,4, \quad k=2,4,
    \end{align}
    where $c=c(\xi)>0$.
\end{proposition}
\begin{proof}
    The proof is similar to Proposition \ref{lemz=xi2}.
\end{proof}

\subsubsection{Opening $\bar{\partial}$ lenses}\label{opendbarlensxi<-6}
Introduce the following functions: for $j=0^{\pm},1,2,3,4$
\begin{align}
    &p_{j1}(z):=p_{j1}(z,\xi)=\frac{\overline{r(z)}}{1-|r(z)|^2}, \quad p_{j3}(z):=-\overline{r(z)},\\
    &p_{j2}(z):=\frac{r(z)}{1-|r(z)|^2}, \quad p_{j4}(z):=-r(z).
\end{align}
Define $R^{(2)}(z):=R^{(2)}(z,\xi)$ by
\begin{equation}
    R^{(2)}(z)=\left\{
        \begin{aligned}
        &\begin{pmatrix} 1 & f_{j1}e^{2it\theta} \\ 0 & 1 \end{pmatrix}, \quad &z\in\Omega_{j1}, \quad j=0^{\pm},1,2,3,4\\
        &\begin{pmatrix} 1 & 0 \\ f_{j2}e^{-2it\theta} & 1 \end{pmatrix} \quad &z\in\Omega_{j2}, \quad j=0^{\pm},1,2,3,4\\
        &\begin{pmatrix} 1 & f_{j3}e^{2it\theta} \\ 0 & 1 \end{pmatrix}\quad &z\in\Omega_{j3}, \quad j=1,2,3,4\\
        &\begin{pmatrix} 1 & 0 \\ f_{j4}e^{-2it\theta} & 1 \end{pmatrix}\quad &z\in\Omega_{j4}, \quad j=1,2,3,4\\
        &I, \quad &elsewhere,
        \end{aligned}
        \right.
\end{equation}
where the functions $f_{jk}$ are given by the following two propositions.

\begin{proposition}[Opening lens at $z=0$]\label{estopenlesz=0}
    $f_{jk}: \overline{\Omega}_{jk}\rightarrow \mathbb{C}$, $j=0^{\pm}, k=1,2$ are continuous on $\overline{\Omega}_{jk}, j=0^{\pm}, k=1,2$
    with boundary values:
    \begin{align}
        f_{0^{\pm}1}(z)=\left\{
                        \begin{aligned}
                        &p_{j1}(z)\delta_{+}^{-2}(z), \quad &z\in\left(\frac{\xi_3}{2},0\right)\cup\left(0, \frac{\xi_2}{2}\right),\\
                        &0, \quad &z\in \Sigma_{0^{\pm}1}.
                        \end{aligned}
                        \right.\label{bdryz=01}\\
        f_{0^{\pm}2}(z)=\left\{
                        \begin{aligned}
                        &p_{j2}(z)\delta_{-}^{2}(z), \quad &z\in\left(\frac{\xi_3}{2},0\right)\cup\left(0, \frac{\xi_2}{2}\right),\\
                        &0, \quad &z\in \Sigma_{0^{\pm}2}.
                        \end{aligned}
                        \right.\label{bdryz=02}
        \end{align}
$f_{jk}, j=0^{\pm}, k=1,2$ have following properties:
\begin{equation}\label{dbarfor0no1}
    \vert \bar{\partial}f_{jk}(z)\vert \lesssim \vert p'_{jk}\left(|z|\right)\vert+|z|^{-1/2}, \quad z\in \Omega_{jk}, \quad j=0^{\pm}, k=1,2.
\end{equation}
Moreover
\begin{equation}\label{dbarfor0no2}
    \vert \bar{\partial}f_{jk}(z)\vert \lesssim \vert p'_{jk}\left(|z|\right)\vert+|z|^{-1}, \quad z\in \Omega_{jk}, \quad j=0^{\pm}, k=1,2.
\end{equation}
\end{proposition}
\begin{proof}
    We give the details for $f_{0^{+}1}(z)$, which can be constructed by
\begin{equation}
    f_{0^{+}1}(z)=p_{0^{+}1}(z)\delta^{-2}_{+}(z){\rm cos}(\kappa_{0}{\rm arg}z), \quad \kappa_{0}=\frac{\pi}{2\phi}.
\end{equation}
Denoting $z=le^{i\varphi}$, we have $\bar{\partial}$-derivative
$\bar{\partial}=\frac{1}{2}e^{i\varphi}(\partial_{l}+il^{-1}\partial_{\varphi})$. Hence
\begin{equation}
    \bar{\partial}f_{0^{+}1}(z)=\frac{e^{i\varphi}}{2}\delta^{-2}_{+}(z)\left[p'_{0^{+}1}\left(l\right){\rm cos}\left(\kappa_{0}\varphi\right)
    -\frac{i}{l}\kappa_{0}{\rm sin}\left(\kappa_0\varphi\right)p_{0^{+}1}(l)\right].
\end{equation}
Using Cauchy-Schwarz inequality, we have
\begin{equation}
    \vert p_{0^{+}1}(l) \vert=\vert p_{0^{+}1}(l)-p_{0^{+}1}(0) \vert=\left\vert \int_{0}^{l}p'_{0^{+}1}(s)ds \right\vert
    \leqslant\Vert p'_{0^{+}1} \Vert_{L^2}l^{1/2} \lesssim l^{1/2}.
\end{equation}
Meanwhile, the boundedness of $\delta^2_{+}(z)$ is guaranteed by the property (v) of Proposition \ref{deltapro}. Thus \eqref{dbarfor0no1} comes true.
As for \eqref{dbarfor0no2}, we just notice $p_{0^{+}1}(l)\in L^{\infty}$.
\end{proof}

\begin{proposition}[Opening lens at saddle points]\label{estopenlenssaddle}
    $f_{jk}: \overline{\Omega}_{jk}\rightarrow \mathbb{C}, j, k=1, 2, 3, 4$ are continuous on $\overline{\Omega}_{jk}, j, k=1,2,3,4$
    with boundary values:
    \begin{align}
        &f_{j1}(z)=\left\{
        \begin{aligned}
        &p_{j1}(z)\delta_{+}^{-2}(z), \quad &z\in I_{j1},\\
        &p_{j1}(\xi_j)e^{-2i\beta\left(\xi_j, \xi\right)}\left(z-\xi_j\right)^{-2i\epsilon_j\nu(\xi_j)}, \quad &z\in \Sigma_{j1},
        \end{aligned}
        \right.\\
        &f_{j2}(z)=\left\{
        \begin{aligned}
        &p_{j2}(z)\delta_{-}^{2}(z), \quad &z\in I_{j2},\\
        &p_{j2}(\xi_j)e^{2i\beta\left(\xi_j, \xi\right)}\left(z-\xi_j\right)^{2i\epsilon_j\nu(\xi_j)}, \quad &z\in \Sigma_{j2},
        \end{aligned}
        \right.\\
        &f_{j3}(z)=\left\{
        \begin{aligned}
        &p_{j3}(z)\delta^{-2}(z), \quad &z\in I_{j3},\\
        &p_{j3}(\xi_j)e^{-2i\beta\left(\xi_j, \xi\right)}\left(z-\xi_j\right)^{-2i\epsilon_j\nu(\xi_j)}, \quad &z\in \Sigma_{j3},
        \end{aligned}
        \right.\\
        &f_{j4}(z)=\left\{
        \begin{aligned}
        &p_{j4}(z)\delta^{2}(z), \quad &z\in I_{j4},\\
        &p_{j4}(\xi_j)e^{2i\beta\left(\xi_j, \xi\right)}\left(z-\xi_j\right)^{2i\epsilon_j\nu(\xi_j)}, \quad &z\in \Sigma_{j4},
        \end{aligned}
        \right.
    \end{align}
where
\begin{align}
    &I_{11}=I_{12}:=\left(\xi_1, +\infty\right), \quad I_{21}=I_{22}:=\left(\frac{\xi_2}{2},\xi_2\right),\nonumber\\
    &I_{31}=I_{32}:=\left(\xi_3,\frac{\xi_3}{2}\right), \quad I_{41}=I_{42}:=\left(-\infty,\xi_4\right),\\
    &I_{13}=I_{14}:=\left(\frac{\xi_2+\xi_1}{2}, \xi_1\right), \quad I_{23}=I_{24}:=\left(\xi_2, \frac{\xi_2+\xi_1}{2}\right),\nonumber\\
    &I_{33}=I_{34}:=\left(\frac{\xi_4+\xi_3}{2},\xi_3\right),\quad I_{43}=I_{44}:=\left(\xi_4,\frac{\xi_4+\xi_3}{2}\right).
\end{align}
And $f_{jk}, j, k=1,2,3,4$ have following properties:
\begin{align}
 & \vert \bar{\partial}f_{jk}(z)\vert \lesssim \vert p'_{jk}\left(\Re z\right)\vert+|z-\xi_j|^{-1/2}, \quad z\in \Omega_{jk}, \quad j, k=1,2,3,4,\label{dbarforxino1}\\
 & \vert f_{jk}(z) \vert \lesssim {\rm sin}^2\left(\kappa_{0}{\rm arg}\left(z-\xi_j\right)\right)+\langle \Re z\rangle^{-1}, \quad  z\in \Omega_{jk}, \quad j, k=1,2,3,4.\label{dbarforxino2}
\end{align}
Moreover, as $z\rightarrow 1$,
\begin{align}
    &\vert \bar{\partial}f_{jk}(z)\vert \lesssim |p'_{jk}||z-1|, \quad z\in \Omega_{24},\Omega_{23},\label{estbalance}\\
    &\vert \bar{\partial}f_{jk}(z)\vert \lesssim |p'_{jk}||z+1|, \quad z\in \Omega_{34},\Omega_{33}.
\end{align}
\end{proposition}
\begin{proof}
    We take $f_{11}(z)$ and $f_{24}(z)$ as examples to present this proof.
    The continuous extension of $f_{11}(z)$ on $\Omega_{11}$ can be constructed by
    \begin{align}
        f_{11}(z)&=p_{11}(\xi_1)e^{-2i\beta\left(\xi_1, \xi\right)}\left(z-\xi_1\right)^{-2i\nu(\xi_1)}\left[1-{\rm cos}
        \left(\kappa_0{\rm arg}\left(z-\xi_1\right)\right)\right] \nonumber\\
        &\quad+p_{11}(\Re z)\delta^{-2}_{+}(z){\rm cos}\left(\kappa_0{\rm arg}\left(z-\xi_1\right)\right),
    \end{align}
Where $\kappa_{0}=\frac{\pi}{2\phi}$. Denote $z=\xi_1+le^{i\varphi}:=\xi_1+u+iv$, where $l, \varphi, u, v\in \mathbb{R}$. Firstly, we have
$|p_{11}(\Re z)|=\frac{|r(\Re z)|}{1-|r(\Re z)|^2}\leqslant \frac{|r(\Re z)|}{1-\rho^2}\lesssim \langle \Re z\rangle^{-1}$. Recalling \eqref{betaest1}, we obtain
\eqref{dbarforxino2}. Applying $\bar{\partial}=\frac{1}{2}e^{i\varphi}(\partial_{l}+il^{-1}\partial_{\varphi})$ to $f_{11}$, we have
\begin{align}
    &\bar{\partial}f_{11}=\left[p_{11}\delta^{-2}_{+}(z)-
    p_{11}(\xi_1)e^{-2i\beta\left(\xi_1, \xi\right)}\left(z-\xi_1\right)^{-2i\nu(\xi_1)}\right]\bar{\partial}{\rm cos}\left(\kappa_0\varphi\right)\nonumber\\
    &\quad+\frac{1}{2}\delta^{-2}_{+}(z)p'_{11}(u,\xi){\rm cos}\left(\kappa_0\varphi\right).
\end{align}
Recalling \eqref{betaest2}, we get \eqref{dbarforxino1} at once.

For $f_{24}$, taking the same method to $f_{11}$, we have
\begin{align}
    &\bar{\partial}f_{24}=\left[p_{24}\delta^2(z)-
    p_{24}(\xi_2)e^{2i\beta\left(\xi_2, \xi\right)}\left(z-\xi_2\right)^{-2i\nu(\xi_2)}\right]\bar{\partial}{\rm cos}\left(\kappa_0\varphi\right)\nonumber\\
    &\quad+\frac{1}{2}\delta^{2}(z)p'_{24}(u,\xi){\rm cos}\left(\kappa_0\varphi\right).
\end{align}
Finally $z$ near $1$, we have $\varphi\rightarrow 0$, thus we obtain
\begin{equation}
    \vert \bar{\partial}f_{24} \vert \lesssim |p'_{24}|{\rm cos}(\kappa_0\varphi)\lesssim|p'_{24}| |z-1|.
\end{equation}
Estimation for the other $f_{jk}$ could be given via similar techniques.
\end{proof}

Define the second transformation
\begin{equation}\label{trans2}
    M^{(2)}(z):=M^{(2)}(x,t;z)=M^{(1)}(z)R^{(2)}(z),
\end{equation}
which constructs the mixed $\bar{\partial}$-RH problem as follows
\begin{RHP}\label{m2rhp}
Find a $2\times 2$ matrix-valued function $M^{(2)}(x,t;z)$ such that\\
    - $M^{(2)}(z)$ is continuous in $\mathbb{C}\backslash \Sigma^{(2)}$, where $\Sigma^{(2)}$ is defined by \eqref{Sigma^(2)con}. \\
    - $M^{(2)}(z)$ takes continuous boundary values $M^{(2)}_{\pm}(z)$ on $\Sigma^{(2)}$ with jump relation
    \begin{equation}
        M^{(2)}_{+}(z)=M^{(2)}_{-}(z)V^{(2)}(z),
    \end{equation}
    where
    \begin{align}\label{rhp2jump}
        V^{(2)}(z)=\left\{
            \begin{aligned}
            &R^{(2)}(z)^{-1}\vert_{\Sigma_{j1}} \quad z\in\Sigma_{j1},\quad j=0^{\pm}, 1, 2, 3, 4, \\
            &R^{(2)}(z)^{-1}\vert_{\Sigma_{j4}} \quad z\in\Sigma_{j4},\quad j=1, 2, 3, 4, \\
            &R^{(2)}(z)\vert_{\Sigma_{j2}} \quad z\in\Sigma_{j2},\quad j=0^{\pm}, 1, 2, 3, 4, \\
            &R^{(2)}(z)\vert_{\Sigma_{j3}} \quad z\in\Sigma_{j3},\quad j=1, 2, 3, 4, \\
            &R^{(2)}(z)^{-1}\vert_{\Omega_{jk}^{\frac{1}{2}}}R^{(2)}(z)\vert_{\Omega_{lm}^{\frac{1}{2}}}, \quad z\in\Sigma^{(1/2)}_{n\pm},\quad n=1,2,3,4.
            \end{aligned}
            \right.
    \end{align}
    - Asymptotic behavior
    \begin{align}
        &M^{(2)}(x,t;z)=I+\mathcal{O}(z^{-1}), \quad  z\rightarrow\infty,\\
        &M^{(2)}(x,t;z)=\frac{\sigma_2}{z}+\mathcal{O}(1), \quad  z\rightarrow 0.
    \end{align}
    - For $z\in\mathbb{C}$, we have $\bar{\partial}$-derivative equality $\bar{\partial}M^{(2)}=M^{(2)}\bar{\partial}R^{(2)}$, where
    \begin{equation}
        \bar{\partial}R^{(2)}=\left\{
            \begin{aligned}
            &\begin{pmatrix} 1 & \bar{\partial}f_{j1}e^{2it\theta} \\ 0 & 1 \end{pmatrix}, \quad &z\in\Omega_{j1}, \quad j=0^{\pm},1,2,3,4,\\
            &\begin{pmatrix} 1 & 0 \\ \bar{\partial}f_{j2}e^{-2it\theta} & 1 \end{pmatrix}, \quad &z\in\Omega_{j2}, \quad j=0^{\pm},1,2,3,4,\\
            &\begin{pmatrix} 1 & \bar{\partial}f_{j3}e^{2it\theta} \\ 0 & 1 \end{pmatrix}, \quad &z\in\Omega_{j3}, \quad j=1,2,3,4,\\
            &\begin{pmatrix} 1 & 0 \\ \bar{\partial}f_{j4}e^{-2it\theta} & 1 \end{pmatrix}, \quad &z\in\Omega_{j4}, \quad j=1,2,3,4,\\
            &0, \quad &elsewhere.
            \end{aligned}
            \right.
    \end{equation}
\end{RHP}
\begin{remark}\label{rmk:nearz=0}\rm
    Notice the boundaries of $f_{jk}, j=0^{\pm}, k=1,2$ defined by \eqref{bdryz=01}, \eqref{bdryz=02}, we actually know that
    $V^{(2)}(z)=I$, for  $z\in\Sigma_{0^{\pm}k}$, $k=1,2$.
\end{remark}

Aiming at solving the mixed $\bar{\partial}$-RH problem \ref{m2rhp}, we decompose it to a pure RH problem for $M^{(PR)}$ with $\bar{\partial}R^{(2)}\equiv 0$ as
well as a pure $\bar{\partial}$-problem $M^{(3)}$ with nonzero $\bar{\partial}R^{(2)}$ derivatives. This step can be shown as the following structure
\begin{equation}
    M^{(2)}=M^{(3)}M^{(PR)}\left\{
        \begin{aligned}
        &\bar{\partial}R^{(2)}\equiv 0\rightarrow M^{(PR)},\\
        &\bar{\partial}R^{(2)}\neq 0\rightarrow M^{(3)}=M^{(2)}{M^{(PR)}}^{-1}.
        \end{aligned}
        \right.
\end{equation}

\subsection{Analysis on pure RH problem}\label{subsec:RLpureRH}
In this subsection, we mainly focus on the analysis for pure RH problem $M^{(PR)}$, which include three parts: global parametrix, local parametrix as well
as small norm RH problem. Noticing that $M^{(PR)}$ is a RH problem with $\bar{\partial}R^{(2)}\equiv 0$, thus, RH conditions for $M^{(PR)}$ are as follows.
\begin{RHP}\label{mrhprhp}
    Find a $2\times 2$ matrix-valued function $M^{(PR)}(x,t;z)$ such that\\
        - $M^{(PR)}(z)$ is analytic in $\mathbb{C}\backslash \Sigma^{(2)}$.\\
        - $M^{(PR)}(z)$ takes continuous boundary values $M^{(PR)}_{\pm}(z)$ on $\Sigma^{(2)}$ with jump relation
        \begin{equation}
            M^{(PR)}_{+}(z)=M^{(PR)}_{-}(z)V^{(2)}(z),
        \end{equation}
        - Asymptotic behavior
        \begin{align}
            &M^{(PR)}(x,t;z)=I+\mathcal{O}(z^{-1}), \quad  z\rightarrow\infty,\\
            &M^{(PR)}(x,t;z)=\frac{\sigma_2}{z}+\mathcal{O}(1), \quad  z\rightarrow 0.
        \end{align}
    \end{RHP}
    Define $U(\xi)$ as the union set of neighborhood of saddle point $\xi_j$ for $j=1,2,3,4$.
    \begin{equation}\label{equ:neighborhood}
        U(\xi)=\bigcup_{j=1,2,3,4}U_{\varrho}(\xi_j), \ {\rm with} \ U_{\varrho}(\xi_j)=\left\{z: |z-\xi_j|<\varrho \right\},
    \end{equation}
    where
    \begin{equation}\label{varrhorestrict}
    \varrho<\frac{1}{3}{\rm min}\left\{{\rm min}\left\{\vert\Im\eta_{n}|\right\}_{n=1}^{2N},\quad
    \underset{k\neq l}{\rm min}\vert\eta_{l}-\eta_{k}\vert, \quad \frac{1}{2}\underset{j=1,2,3,4}{\rm min}\vert\xi_j\pm 1\vert
    \quad \frac{1}{2}\underset{j=1,2,3,4}{\rm min}\vert\xi_j\vert \right\}.
    \end{equation}
    \begin{remark}\rm
        The third and fourth restriction of \eqref{varrhorestrict} is to remove the singularity $z=0,\pm 1$ from local model
        which will be discussed in Subsection \ref{localpc}.
    \end{remark}

    \begin{proposition}\label{outxi1}
        For $1\leqslant p\leqslant +\infty$, there exists a constant $\hbar=\hbar(p)>0$,
        such that the jump matrix $V^{(2)}$ defined in \eqref{rhp2jump} admit the following estimation
        as $t\rightarrow+\infty$
        \begin{equation}
            \Vert V^{(2)}-I \Vert_{L^p(\Sigma_{jk}\backslash U_{\varrho}(\xi_j))}=\mathcal{O}(e^{-\hbar t}), \quad {\rm for}\quad j,k=1,2,3,4.
        \end{equation}
    \end{proposition}
    \begin{proof}
        We take $z\in\Sigma_{24}\backslash U_{\varrho}(\xi_2)$ as an example, the other cases can be proved in a
        similar way. For $z\in\Sigma_{24}\backslash U_{\varrho}(\xi_2)$, $1\leqslant p<+\infty$, by using \eqref{rhp2jump} and \eqref{dbarforxino2},
        we have
        \begin{align}
            \Vert V^{(2)}-I \Vert_{L^p(\Sigma_{24}\backslash U_{\varrho}(\xi_2))}&=
            \Vert p_{24}(\xi_2)e^{2i\beta(\xi_2,\xi)}(z-\xi_2)^{-2i\nu(\xi_2)}e^{-2it\theta} \Vert_{L^p(\Sigma_{24}\backslash U_{\varrho}(\xi_2))}\\
             &\lesssim\Vert e^{-2it\theta}\Vert_{L^p(\Sigma_{24}\backslash U_{\varrho}(\xi_2))}.
        \end{align}
        for $z\in\Sigma_{24}\backslash U_{\varrho}(\xi_2)$, we still denote $z=\xi_2+le^{i\varphi}=\xi_2+u+iv$, $l>\varrho$.
        With the help of Proposition \ref{lemz=xi2}, we have
        \begin{align}
            \Vert e^{-2it\theta}\Vert^{p}_{L^p(\Sigma_{24}\backslash U_{\varrho}(\xi_2))}
            &\lesssim\int_{\Sigma_{24}\backslash U_{\varrho}(\xi_2)}e^{-2tpc\left(1+|z|^{-2}\right)v^2}dz\nonumber\\
            (1+|z|^{-2}\geqslant 1)&\lesssim\int_{\varrho}^{\infty}e^{-pctl}dl\nonumber\\
            &\lesssim t^{-1}e^{-pct\varrho},
        \end{align}
        where the value of $c=c(\xi)$ above changes from line to line.
    \end{proof}

    \begin{proposition}\label{outxi2}
    For $1\leqslant p< +\infty$, there exists a constant $\hbar'=\hbar'(p)>0$,
    such that the jump matrix $V^{(2)}$ defined in \eqref{rhp2jump} admit the following estimate
    as $t\rightarrow+\infty$
    \begin{equation}
        \Vert V^{(2)}-I \Vert_{L^p(\Sigma_{j\pm}^{(1/2)})
        }=\mathcal{O}(e^{-\hbar' t}), \quad {\rm for} \quad j=1,2,3,4.
    \end{equation}
    \end{proposition}
    \begin{proof}
        We only give the details for $z\in\Sigma_{1+}^{(1/2)}$.
        \begin{align}
            \Vert V^{(2)}-I \Vert_{L^p(\Sigma_{1+}^{(1/2)})}
            &=\Vert(f_{24}-f_{14})e^{-2it\theta}\Vert_{L^p(\Sigma_{1+}^{(1/2)})}\nonumber\\
            &\overset{\eqref{dbarforxino2}}{\lesssim}\Vert e^{-2it\theta}\Vert_{L^p(\Sigma_{1+}^{(1/2)})}\\
            &\lesssim t^{-\frac{1}{p}}e^{-ct}.
        \end{align}
    \end{proof}

\subsubsection{Global parametrix: $M^{(\infty)}$}\label{guzijie}
The leading order of $M^{(PR)}$ is approximated by a global parametrix (denoted by $M^{(\infty)}$) with exponentially decaying on the jump of $M^{(PR)}(z)$ (see Proposition \ref{outxi1}
and Proposition \ref{outxi2}). Thus we consider the following RH problem
\begin{RHP}\label{msolrhp}
Find a $2\times 2$ matrix-valued function $M^{(\infty)}(x,t;z)$ which satisfies\\
- $M^{(\infty)}(z)$ is analytical in $\mathbb{C}\backslash\{0\}$.\\
- Asymptotic behavior
    \begin{align}
        &M^{(\infty)}(x,t;z)=I+\mathcal{O}(z^{-1}), \quad  z\rightarrow\infty,\\
        &M^{(\infty)}(x,t;z)=\frac{\sigma_2}{z}+\mathcal{O}(1), \quad  z\rightarrow 0.
    \end{align}
\end{RHP}
Then the following result is standard.
\begin{proposition}\label{prop:global}
    The unique solution of RH problem \ref{msolrhp} is given by
    \begin{equation}
        M^{(\infty)}(z)=I+\frac{\sigma_2}{z}.
    \end{equation}
\end{proposition}

\subsubsection{Local parametrix near saddle points: $M^{(LC)}(z)$}\label{localpc}
The sub-leading contribution stems form the local behavior near critical points $\xi_j$, $j=1,2,3,4$.
It turns out that the local parametrix (denoted by $M^{(LC, \xi_j)}$, $j=1,2,3,4$ below) can be constructed in terms of the
solution of the well-known Webb (parabolic cylinder) equation.
Denote $U_{\varrho}(\xi_j)$ the open disk of radius $\varrho$ defined by \eqref{equ:neighborhood} around $\xi_j$, $j=1,2,3,4$ respectively. And
define the contours $\Sigma^{(LC)}:=\left(\cup_{j,k=1,2,3,4}\Sigma_{jk}\right)\cap U(\xi)$ (see Figure \ref{localxi}).

\begin{figure}[htbp]
    \begin{center}
\tikzset{every picture/.style={line width=0.75pt}} 
\begin{tikzpicture}[x=0.75pt,y=0.75pt,yscale=-1,xscale=1]
\draw  [dash pattern={on 0.84pt off 2.51pt}]  (337.02,290) -- (334.57,8.32) ;
\draw [shift={(334.55,6.32)}, rotate = 89.5] [color={rgb, 255:red, 0; green, 0; blue, 0 }  ][line width=0.75]    (10.93,-3.29) .. controls (6.95,-1.4) and (3.31,-0.3) .. (0,0) .. controls (3.31,0.3) and (6.95,1.4) .. (10.93,3.29)   ;
\draw [color={rgb, 255:red, 0; green, 0; blue, 0 }  ,draw opacity=1 ]   (107.65,156.49) -- (143.02,124.87) ;
\draw [shift={(129.81,136.68)}, rotate = 138.2] [color={rgb, 255:red, 0; green, 0; blue, 0 }  ,draw opacity=1 ][line width=0.75]    (10.93,-3.29) .. controls (6.95,-1.4) and (3.31,-0.3) .. (0,0) .. controls (3.31,0.3) and (6.95,1.4) .. (10.93,3.29)   ;
\draw [color={rgb, 255:red, 0; green, 0; blue, 0 }  ,draw opacity=1 ]   (111.06,159.91) -- (146.43,189.29) ;
\draw [shift={(133.36,178.43)}, rotate = 219.72] [color={rgb, 255:red, 0; green, 0; blue, 0 }  ,draw opacity=1 ][line width=0.75]    (10.93,-3.29) .. controls (6.95,-1.4) and (3.31,-0.3) .. (0,0) .. controls (3.31,0.3) and (6.95,1.4) .. (10.93,3.29)   ;
\draw [color={rgb, 255:red, 0; green, 0; blue, 0 }  ,draw opacity=1 ]   (71.43,128.29) -- (111.06,159.91) ;
\draw [shift={(95.94,147.84)}, rotate = 218.58] [color={rgb, 255:red, 0; green, 0; blue, 0 }  ,draw opacity=1 ][line width=0.75]    (10.93,-3.29) .. controls (6.95,-1.4) and (3.31,-0.3) .. (0,0) .. controls (3.31,0.3) and (6.95,1.4) .. (10.93,3.29)   ;
\draw [color={rgb, 255:red, 0; green, 0; blue, 0 }  ,draw opacity=1 ]   (71.02,189.87) -- (107.65,156.49) ;
\draw [shift={(93.77,169.14)}, rotate = 137.66] [color={rgb, 255:red, 0; green, 0; blue, 0 }  ,draw opacity=1 ][line width=0.75]    (10.93,-3.29) .. controls (6.95,-1.4) and (3.31,-0.3) .. (0,0) .. controls (3.31,0.3) and (6.95,1.4) .. (10.93,3.29)   ;
\draw  [dash pattern={on 0.84pt off 2.51pt}] (60.43,156.49) .. controls (60.43,130.41) and (81.57,109.27) .. (107.65,109.27) .. controls (133.73,109.27) and (154.87,130.41) .. (154.87,156.49) .. controls (154.87,182.58) and (133.73,203.72) .. (107.65,203.72) .. controls (81.57,203.72) and (60.43,182.58) .. (60.43,156.49) -- cycle ;
\draw [color={rgb, 255:red, 0; green, 0; blue, 0 }  ,draw opacity=1 ]   (257.65,157.86) -- (293.02,126.24) ;
\draw [shift={(279.81,138.05)}, rotate = 138.2] [color={rgb, 255:red, 0; green, 0; blue, 0 }  ,draw opacity=1 ][line width=0.75]    (10.93,-3.29) .. controls (6.95,-1.4) and (3.31,-0.3) .. (0,0) .. controls (3.31,0.3) and (6.95,1.4) .. (10.93,3.29)   ;
\draw [color={rgb, 255:red, 0; green, 0; blue, 0 }  ,draw opacity=1 ]   (261.06,161.27) -- (296.43,190.65) ;
\draw [shift={(283.36,179.8)}, rotate = 219.72] [color={rgb, 255:red, 0; green, 0; blue, 0 }  ,draw opacity=1 ][line width=0.75]    (10.93,-3.29) .. controls (6.95,-1.4) and (3.31,-0.3) .. (0,0) .. controls (3.31,0.3) and (6.95,1.4) .. (10.93,3.29)   ;
\draw [color={rgb, 255:red, 0; green, 0; blue, 0 }  ,draw opacity=1 ]   (221.43,129.65) -- (261.06,161.27) ;
\draw [shift={(245.94,149.2)}, rotate = 218.58] [color={rgb, 255:red, 0; green, 0; blue, 0 }  ,draw opacity=1 ][line width=0.75]    (10.93,-3.29) .. controls (6.95,-1.4) and (3.31,-0.3) .. (0,0) .. controls (3.31,0.3) and (6.95,1.4) .. (10.93,3.29)   ;
\draw [color={rgb, 255:red, 0; green, 0; blue, 0 }  ,draw opacity=1 ]   (221.02,191.24) -- (257.65,157.86) ;
\draw [shift={(243.77,170.51)}, rotate = 137.66] [color={rgb, 255:red, 0; green, 0; blue, 0 }  ,draw opacity=1 ][line width=0.75]    (10.93,-3.29) .. controls (6.95,-1.4) and (3.31,-0.3) .. (0,0) .. controls (3.31,0.3) and (6.95,1.4) .. (10.93,3.29)   ;
\draw  [dash pattern={on 0.84pt off 2.51pt}] (210.43,157.86) .. controls (210.43,131.78) and (231.57,110.64) .. (257.65,110.64) .. controls (283.73,110.64) and (304.87,131.78) .. (304.87,157.86) .. controls (304.87,183.94) and (283.73,205.08) .. (257.65,205.08) .. controls (231.57,205.08) and (210.43,183.94) .. (210.43,157.86) -- cycle ;
\draw [color={rgb, 255:red, 0; green, 0; blue, 0 }  ,draw opacity=1 ]   (409.65,154.86) -- (445.02,123.24) ;
\draw [shift={(431.81,135.05)}, rotate = 138.2] [color={rgb, 255:red, 0; green, 0; blue, 0 }  ,draw opacity=1 ][line width=0.75]    (10.93,-3.29) .. controls (6.95,-1.4) and (3.31,-0.3) .. (0,0) .. controls (3.31,0.3) and (6.95,1.4) .. (10.93,3.29)   ;
\draw [color={rgb, 255:red, 0; green, 0; blue, 0 }  ,draw opacity=1 ]   (413.06,158.27) -- (448.43,187.65) ;
\draw [shift={(435.36,176.8)}, rotate = 219.72] [color={rgb, 255:red, 0; green, 0; blue, 0 }  ,draw opacity=1 ][line width=0.75]    (10.93,-3.29) .. controls (6.95,-1.4) and (3.31,-0.3) .. (0,0) .. controls (3.31,0.3) and (6.95,1.4) .. (10.93,3.29)   ;
\draw [color={rgb, 255:red, 0; green, 0; blue, 0 }  ,draw opacity=1 ]   (373.43,126.65) -- (413.06,158.27) ;
\draw [shift={(397.94,146.2)}, rotate = 218.58] [color={rgb, 255:red, 0; green, 0; blue, 0 }  ,draw opacity=1 ][line width=0.75]    (10.93,-3.29) .. controls (6.95,-1.4) and (3.31,-0.3) .. (0,0) .. controls (3.31,0.3) and (6.95,1.4) .. (10.93,3.29)   ;
\draw [color={rgb, 255:red, 0; green, 0; blue, 0 }  ,draw opacity=1 ]   (373.02,188.24) -- (409.65,154.86) ;
\draw [shift={(395.77,167.51)}, rotate = 137.66] [color={rgb, 255:red, 0; green, 0; blue, 0 }  ,draw opacity=1 ][line width=0.75]    (10.93,-3.29) .. controls (6.95,-1.4) and (3.31,-0.3) .. (0,0) .. controls (3.31,0.3) and (6.95,1.4) .. (10.93,3.29)   ;
\draw  [dash pattern={on 0.84pt off 2.51pt}] (362.43,154.86) .. controls (362.43,128.78) and (383.57,107.64) .. (409.65,107.64) .. controls (435.73,107.64) and (456.87,128.78) .. (456.87,154.86) .. controls (456.87,180.94) and (435.73,202.08) .. (409.65,202.08) .. controls (383.57,202.08) and (362.43,180.94) .. (362.43,154.86) -- cycle ;
\draw [color={rgb, 255:red, 0; green, 0; blue, 0 }  ,draw opacity=1 ]   (559.65,154.86) -- (595.02,123.24) ;
\draw [shift={(581.81,135.05)}, rotate = 138.2] [color={rgb, 255:red, 0; green, 0; blue, 0 }  ,draw opacity=1 ][line width=0.75]    (10.93,-3.29) .. controls (6.95,-1.4) and (3.31,-0.3) .. (0,0) .. controls (3.31,0.3) and (6.95,1.4) .. (10.93,3.29)   ;
\draw [color={rgb, 255:red, 0; green, 0; blue, 0 }  ,draw opacity=1 ]   (563.06,158.27) -- (598.43,187.65) ;
\draw [shift={(585.36,176.8)}, rotate = 219.72] [color={rgb, 255:red, 0; green, 0; blue, 0 }  ,draw opacity=1 ][line width=0.75]    (10.93,-3.29) .. controls (6.95,-1.4) and (3.31,-0.3) .. (0,0) .. controls (3.31,0.3) and (6.95,1.4) .. (10.93,3.29)   ;
\draw [color={rgb, 255:red, 0; green, 0; blue, 0 }  ,draw opacity=1 ]   (523.43,126.65) -- (563.06,158.27) ;
\draw [shift={(547.94,146.2)}, rotate = 218.58] [color={rgb, 255:red, 0; green, 0; blue, 0 }  ,draw opacity=1 ][line width=0.75]    (10.93,-3.29) .. controls (6.95,-1.4) and (3.31,-0.3) .. (0,0) .. controls (3.31,0.3) and (6.95,1.4) .. (10.93,3.29)   ;
\draw [color={rgb, 255:red, 0; green, 0; blue, 0 }  ,draw opacity=1 ]   (523.02,188.24) -- (559.65,154.86) ;
\draw [shift={(545.77,167.51)}, rotate = 137.66] [color={rgb, 255:red, 0; green, 0; blue, 0 }  ,draw opacity=1 ][line width=0.75]    (10.93,-3.29) .. controls (6.95,-1.4) and (3.31,-0.3) .. (0,0) .. controls (3.31,0.3) and (6.95,1.4) .. (10.93,3.29)   ;
\draw  [dash pattern={on 0.84pt off 2.51pt}] (512.43,154.86) .. controls (512.43,128.78) and (533.57,107.64) .. (559.65,107.64) .. controls (585.73,107.64) and (606.87,128.78) .. (606.87,154.86) .. controls (606.87,180.94) and (585.73,202.08) .. (559.65,202.08) .. controls (533.57,202.08) and (512.43,180.94) .. (512.43,154.86) -- cycle ;
\draw  [dash pattern={on 0.84pt off 2.51pt}]  (36.43,156.29) -- (642.43,156.29) ;
\draw [shift={(644.43,156.29)}, rotate = 180] [color={rgb, 255:red, 0; green, 0; blue, 0 }  ][line width=0.75]    (10.93,-3.29) .. controls (6.95,-1.4) and (3.31,-0.3) .. (0,0) .. controls (3.31,0.3) and (6.95,1.4) .. (10.93,3.29)   ;
\draw (345.38,7.59) node [anchor=north west][inner sep=0.75pt]  [font=\scriptsize]  {$Im\ z$};
\draw (627.89,168.09) node [anchor=north west][inner sep=0.75pt]  [font=\scriptsize]  {$Re\ z$};
\draw (101.71,165.06) node [anchor=north west][inner sep=0.75pt]  [font=\scriptsize]  {$\xi _{4}$};
\draw (326.79,141.18) node [anchor=north west][inner sep=0.75pt]  [font=\scriptsize]  {$0$};
\draw (45,103.9) node [anchor=north west][inner sep=0.75pt]  [font=\tiny]  {$\Sigma _{41}^{( LC)}$};
\draw (142.5,187.04) node [anchor=north west][inner sep=0.75pt]  [font=\tiny]  {$\Sigma _{43}^{( LC)}$};
\draw (251.71,166.42) node [anchor=north west][inner sep=0.75pt]  [font=\scriptsize]  {$\xi _{3}$};
\draw (403.71,163.42) node [anchor=north west][inner sep=0.75pt]  [font=\scriptsize]  {$\xi _{2}$};
\draw (553.71,163.42) node [anchor=north west][inner sep=0.75pt]  [font=\scriptsize]  {$\xi _{1}$};
\draw (43,192.9) node [anchor=north west][inner sep=0.75pt]  [font=\tiny]  {$\Sigma _{42}^{( LC)}$};
\draw (142.5,100.04) node [anchor=north west][inner sep=0.75pt]  [font=\tiny]  {$\Sigma _{44}^{( LC)}$};
\draw (596.5,100.04) node [anchor=north west][inner sep=0.75pt]  [font=\tiny]  {$\Sigma _{11}^{( LC)}$};
\draw (602.5,186.04) node [anchor=north west][inner sep=0.75pt]  [font=\tiny]  {$\Sigma _{12}^{( LC)}$};
\draw (498.5,102.04) node [anchor=north west][inner sep=0.75pt]  [font=\tiny]  {$\Sigma _{14}^{( LC)}$};
\draw (492.5,189.04) node [anchor=north west][inner sep=0.75pt]  [font=\tiny]  {$\Sigma _{13}^{( LC)}$};
\draw (438.5,93.04) node [anchor=north west][inner sep=0.75pt]  [font=\tiny]  {$\Sigma _{24}^{( LC)}$};
\draw (436.5,192.04) node [anchor=north west][inner sep=0.75pt]  [font=\tiny]  {$\Sigma _{23}^{( LC)}$};
\draw (350.5,94.04) node [anchor=north west][inner sep=0.75pt]  [font=\tiny]  {$\Sigma _{21}^{( LC)}$};
\draw (349.5,190.04) node [anchor=north west][inner sep=0.75pt]  [font=\tiny]  {$\Sigma _{22}^{( LC)}$};
\draw (286.5,97.04) node [anchor=north west][inner sep=0.75pt]  [font=\tiny]  {$\Sigma _{31}^{( LC)}$};
\draw (293.5,187.04) node [anchor=north west][inner sep=0.75pt]  [font=\tiny]  {$\Sigma _{32}^{( LC)}$};
\draw (201.5,99.04) node [anchor=north west][inner sep=0.75pt]  [font=\tiny]  {$\Sigma _{34}^{( LC)}$};
\draw (196.5,195.04) node [anchor=north west][inner sep=0.75pt]  [font=\tiny]  {$\Sigma _{33}^{( LC)}$};
\end{tikzpicture}
        \caption{Jump contour $\Sigma^{(LC)}$ of $M^{(LC, \ \xi_j)}(z)$, $j=1,2,3,4$.}\label{localxi}
        \end{center}
    \end{figure}
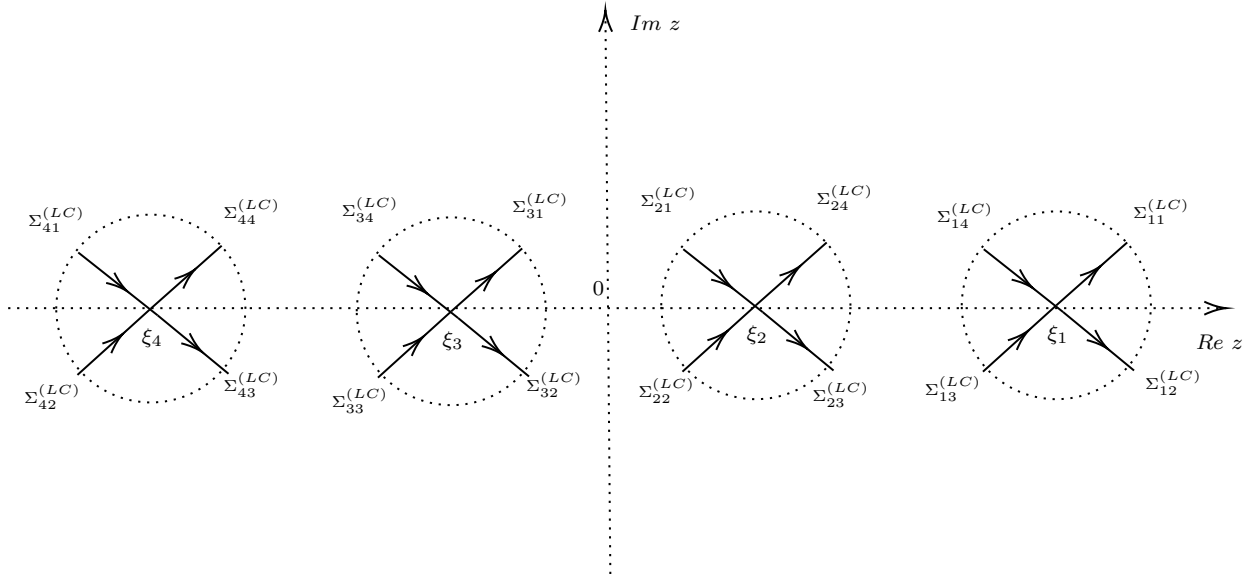

Now we turn to the following localized RH problem.
\begin{RHP}
    Find a $2\times 2$ matrix-valued function $M^{(LC,\xi_j)}(x,t;z)$ such that\\
    - $M^{(LC,\xi_j)}(z)$ is analytical in $\mathbb{C}\backslash \Sigma^{(LC,\xi_j)}$, where  $\Sigma^{(LC,\xi_j)}:=\Sigma^{(LC)}_{jk}$, $k=1,2,3,4$.\\
    - $M^{(LC,\xi_j)}(z)$ takes continuous boundary values $M^{(LC,\xi_j)}_{\pm}(z)$ on $\Sigma^{(LC,\xi_j)}$ with jump relation
        \begin{equation}
            M^{(LC,\xi_j)}_{+}(z)=M^{(LC,\xi_j)}_{-}(z)V^{(LC,\xi_j)}(z), \quad z\in \Sigma^{(LC,\xi_j)},
        \end{equation}
        where
        \begin{equation}
            V^{(LC,\xi_j)}(z)=\left\{
                \begin{aligned}
                &\begin{pmatrix} 1 & -\frac{\overline{r(\xi_j)}}{1-|r(\xi_j)|^2}e^{-2i\beta(\xi_j,\xi)}(z-\xi_j)^{-2i\epsilon_j\nu(\xi_j)}e^{2it\theta} \\ 0 & 1 \end{pmatrix}, \quad &z\in\Sigma^{(LC)}_{j1}, \\
                &\begin{pmatrix} 1 & 0 \\ \frac{r(\xi_j)}{1-|r(\xi_j)|^2}e^{2i\beta(\xi_j,\xi)}(z-\xi_j)^{2i\epsilon_j\nu(\xi_j)}e^{-2it\theta} & 1 \end{pmatrix}, \quad &z\in\Sigma^{(LC)}_{j2}, \\
                &\begin{pmatrix} 1 & -\overline{r(\xi_j)}e^{-2i\beta(\xi_j,\xi)}(z-\xi_j)^{-2i\epsilon_j\nu(\xi_j)}e^{2it\theta} \\ 0 & 1 \end{pmatrix}, \quad &z\in\Sigma^{(LC)}_{j3}, \\
                &\begin{pmatrix} 1 & 0 \\ r(\xi_j)e^{2i\beta(\xi_j,\xi)}(z-\xi_j)^{2i\epsilon_j\nu(\xi_j)}e^{-2it\theta} & 1 \end{pmatrix}, \quad &z\in\Sigma^{(LC)}_{j4}. \\
                \end{aligned}
                \right.
        \end{equation}
    - $M^{(LC, \xi_j)}(x,t;z)=I+\mathcal{O}(z^{-1}), \quad  z\rightarrow\infty$.
\end{RHP}

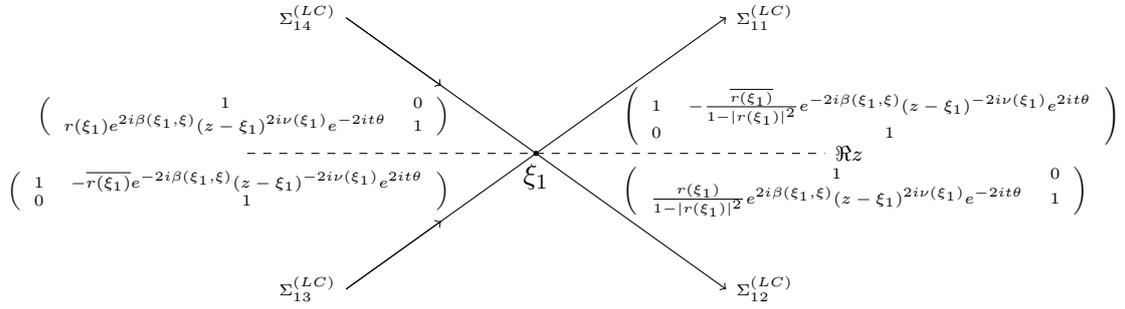
\begin{figure}[htbp]
	\centering
		\begin{tikzpicture}[node distance=2cm]
		\draw[->](0,0)--(2.5,1.8)node[right]{\tiny$\Sigma^{(LC)}_{11}$};
		\draw(0,0)--(-2.5,1.8)node[left]{\tiny$\Sigma^{(LC)}_{14}$};
		\draw(0,0)--(-2.5,-1.8)node[left]{\tiny$\Sigma^{(LC)}_{13}$};
		\draw[->](0,0)--(2.5,-1.8)node[right]{\tiny$\Sigma^{(LC)}_{12}$};
		\draw[dashed](-3.8,0)--(3.8,0)node[right]{\scriptsize $\Re z$};
		\draw[->](-2.5,-1.8)--(-1.25,-0.9);
		\draw[->](-2.5,1.8)--(-1.25,0.9);
		\coordinate (A) at (1,0.5);
		\coordinate (B) at (1,-0.5);
		\coordinate (G) at (-1,0.5);
		\coordinate (H) at (-1,-0.5);
		\coordinate (I) at (0,0);
		\fill (A) circle (0pt) node[right] {\tiny$\left(\begin{array}{cc}
		1 & -\frac{\overline{r(\xi_1)}}{1-|r(\xi_1)|^2}e^{-2i\beta(\xi_1,\xi)}(z-\xi_1)^{-2i\nu(\xi_1)}e^{2it\theta}\\
		0 & 1
		\end{array}\right)$};
	\fill (B) circle (0pt) node[right] {\tiny$\left(\begin{array}{cc}
		1 & 0\\
		\frac{r(\xi_1)}{1-|r(\xi_1)|^2}e^{2i\beta(\xi_1,\xi)}(z-\xi_1)^{2i\nu(\xi_1)}e^{-2it\theta} & 1
		\end{array}\right)$};
	\fill (G) circle (0pt) node[left] {\tiny$\left(\begin{array}{cc}
		1 & 0\\
		r(\xi_1)e^{2i\beta(\xi_1,\xi)}(z-\xi_1)^{2i\nu(\xi_1)}e^{-2it\theta} & 1
		\end{array}\right)$};
	\fill (H) circle (0pt) node[left] {\tiny$\left(\begin{array}{cc}
		1 & -\overline{r(\xi_1)}e^{-2i\beta(\xi_1,\xi)}(z-\xi_1)^{-2i\nu(\xi_1)}e^{2it\theta}\\
		0 & 1
		\end{array}\right)$};
		\fill (I) circle (1pt) node[below] {$\xi_1$};
		\end{tikzpicture}
	\caption{\footnotesize The contour $\Sigma^{(LC,\xi_1)}$ and the jump matrix on it.}
\end{figure}

\begin{figure}[htbp]
	\centering
		\begin{tikzpicture}[node distance=2cm]
		\draw[->](0,0)--(2.5,1.8)node[right]{\tiny$\Sigma^{(LC)}_{24}$};
		\draw(0,0)--(-2.5,1.8)node[left]{\tiny$\Sigma^{(LC)}_{21}$};
		\draw(0,0)--(-2.5,-1.8)node[left]{\tiny$\Sigma^{(LC)}_{22}$};
		\draw[->](0,0)--(2.5,-1.8)node[right]{\tiny$\Sigma^{(LC)}_{23}$};
		\draw[dashed](-3.8,0)--(3.8,0)node[right]{\scriptsize $\Re z$};
		\draw[->](-2.5,-1.8)--(-1.25,-0.9);
		\draw[->](-2.5,1.8)--(-1.25,0.9);
		\coordinate (A) at (1,0.5);
		\coordinate (B) at (1,-0.5);
		\coordinate (G) at (-1,0.5);
		\coordinate (H) at (-1,-0.5);
		\coordinate (I) at (0,0);
		\fill (A) circle (0pt) node[right] {\tiny$\left(\begin{array}{cc}
		1 & 0\\
		r(\xi_2)e^{2i\beta(\xi_2,\xi)}(z-\xi_2)^{-2i\nu(\xi_2)}e^{-2it\theta} & 1
		\end{array}\right)$};
	    \fill (B) circle (0pt) node[right] {\tiny$\left(\begin{array}{cc}
		1 & -\overline{r(\xi_2)}e^{-2i\beta(\xi_2,\xi)}(z-\xi_2)^{2i\nu(\xi_2)}e^{2it\theta}\\
		0 & 1
		\end{array}\right)$};
	    \fill (G) circle (0pt) node[left] {\tiny$\left(\begin{array}{cc}
		1 & -\frac{\overline{r(\xi_2)}}{1-|r(\xi_2)|^2}e^{-2i\beta(\xi_2,\xi)}(z-\xi_2)^{2i\nu(\xi_2)}e^{2it\theta}\\
		0 & 1
		\end{array}\right)$};
	    \fill (H) circle (0pt) node[left] {\tiny$\left(\begin{array}{cc}
		1 & 0\\
		\frac{r(\xi_2)}{1-|r(\xi_2)|^2}e^{2i\beta(\xi_2,\xi)}(z-\xi_2)^{-2i\nu(\xi_2)}e^{-2it\theta} & 1
		\end{array}\right)$};
		\fill (I) circle (1pt) node[below] {$\xi_2$};
		\end{tikzpicture}
	\caption{\footnotesize The contour $\Sigma^{(LC,\xi_2)}$ and the jump matrix on it.}
\end{figure}
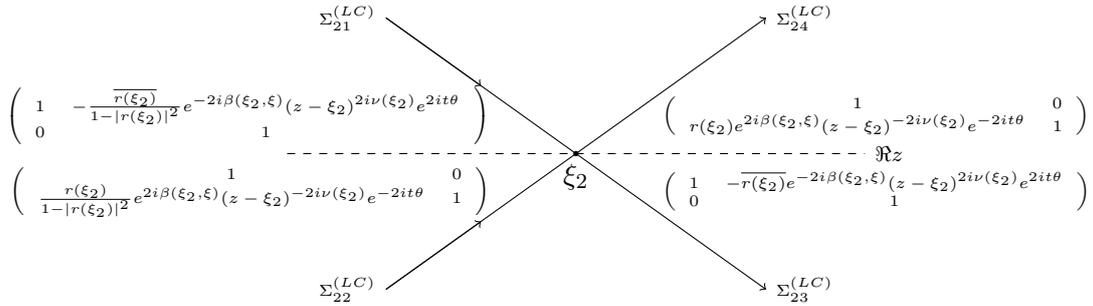

For $z$ near $\xi_j$, $j=1,2,3,4$, we have
\begin{equation}\label{equ:phasereduction}
    \theta(z)=\theta(\xi_j)+\frac{\theta''(\xi_j)}{2}(z-\xi_j)^2+\mathcal{O}(\vert z-\xi_j\vert^3), \quad z\rightarrow\xi_j, \ j=1,2,3,4.
\end{equation}

Thus, for $z\in U_{\varrho}(\xi_j)$, we define the rescaled variable $\zeta$ by
\begin{equation}\label{scale}
    \zeta(z)=\left(2t\epsilon_j\theta''\left(\xi_j\right)\right)^{\frac{1}{2}}\left(z-\xi_j\right), \ j=1,2,3,4,
\end{equation}
which is to match the standard model presented in Appendix \ref{pcmodel}.

And the scaling operator $N_{\xi_j}$ admits the following mapping
\begin{align}
    N_{\xi_j}:&U_{\varrho}(\xi_j)\longrightarrow U_{0},\ j=1,2,3,4,\\
    &z \longmapsto \zeta.
\end{align}
where $U_0$ is a neighborhood of $\zeta=0$.

Choose the free variable $r_{\xi_j}$ appeared in the Appendix \ref{pcmodel} by
\begin{equation}
    r_{\xi_j}=r(\xi_j)e^{2i\beta(\xi_j,\xi)-2it\theta(\xi_j)}{\rm exp}\left[-i\epsilon_j\nu(\xi_j){\rm log}(2t\epsilon_j\theta''\left(\xi_j\right))\right],
\end{equation}
with the equality $|r(\xi_j)|^2=|r_{\xi_j}|^2$, $j=1,2,3,4$.

In the above expression, the complex powers are defined by choosing the branch of the logarithm with $0<{\rm arg}\zeta<2\pi$ near
$\xi_1$, $\xi_3$, and the branch of the logarithm with $-\pi<{\rm arg}\zeta<\pi$ near $\xi_2$, $\xi_4$. Through this scaling of variable, the jump $V^{{LC,\xi_j}}(z)$ can be approximated by the
jump of a parabolic cylinder model which is shown in Appendix \ref{pcmodel}.

\begin{remark}\label{rmk:phasereduction}
    \rm  In the expansion \eqref{equ:phasereduction}, the higher order term as $z\rightarrow\xi_j$ could be ignored under the condition $\vert x/t \vert=\mathcal{O}(1)$.
    Without loss of generality, we take the neighborhood of $\xi_1$ as an example. we expand
\begin{equation}
    \theta(z)=\theta(\xi_1)+\frac{\theta''(\xi_1)}{2}(z-\xi_1)^2+\theta_c (z-\xi_1)^3,
\end{equation}
where $\theta_c=\frac{\theta'''(\kappa\xi_1+(1-\kappa) z)}{3!}$, $\kappa\in(0,1)$ is the coefficient of remainder.

Owe to the scaling \eqref{scale}, we have the following transformation
\begin{equation}
    N: g\longrightarrow (Ng)(\zeta):=g\left((2t\theta''(\xi_1))^{-\frac{1}{2}}\zeta+\xi_1\right),
\end{equation}
which acts on $e^{2it\theta(z)}$ to obtain
\begin{align}\label{thetascale}
    e^{2it\theta(z)}=e^{2it(N\theta)(\zeta)}\overset{\eqref{scale}}{=}e^{2it\theta(\xi_1)}\cdot
    e^{\frac{i}{2}\zeta^2}\cdot e^{2it\theta_c \cdot(2t\theta''(\xi_1))^{-\frac{3}{2}}\zeta^3}.
\end{align}
The first and second term of R.H.S of \eqref{thetascale} are used to match parabolic cylinder model.
What we care about is the third term.
Since $\zeta\in U_{0}$, the neighborhood of zero, we can set $\zeta=u+iv$, $|u|<\varepsilon$, $|v|<\varepsilon$. Thus
\begin{align}
    \vert e^{2it\theta_c (2t\theta''(\xi_1))^{-\frac{3}{2}}\zeta^3} \vert&=\vert e^{2it\theta_c \cdot(2t\theta''(\xi_1))^{-\frac{3}{2}}(u+iv)^3} \vert\nonumber\\
    &={\rm exp}\left((2t\theta''(\xi_1))^{-\frac{3}{2}}\Re \left(2it\left(\Re\left(\theta_c\right)+i \Im\left(\theta_c\right)\right) \cdot (u+iv)^3\right)\right)\nonumber\\
    &={\rm exp}\left[-(2t)^{-\frac{1}{2}}(\theta''(\xi_1))^{-\frac{3}{2}}\left(\Re(\theta_c)\left(3u^2v-v^3\right)+\Im(\theta_c)\left(u^3-3uv^2\right)\right)\right]\nonumber\\
    &\hspace{20em}\rightarrow 1 \quad {\rm as} \quad t\rightarrow+\infty,
\end{align}
by which the effects of the higher power could be ignored. The premise of this result is that $\Re(\theta_c)\left(3u^2v-v^3\right)+\Im(\theta_c)\left(u^3-3uv^2\right)$ is finite,
which follows from finite $\xi$.
\end{remark}

As a consequence, a standard local parametrix for $M^{(LC,\xi_j)}$, $j=1,2,3,4$ is constructed by
\begin{equation}
    M^{(LC,\xi_j)}(x,t;z)=M^{(\infty)} \cdot M^{(PC, \xi_j)}\left(\xi, \zeta(z)=\left(2t\epsilon_j\theta''\left(\xi_j\right)\right)^{\frac{1}{2}}\left(z-\xi_j\right)\right).
\end{equation}
where
\begin{equation}\label{mmodj}
    M^{(PC,\xi_j)}\left(\xi, \zeta(z)\right)=I+\frac{\left(2t\epsilon_j\theta''(\xi_j)\right)^{-\frac{1}{2}}}{z-\xi_j}\left(\begin{array}{cc}
        0 & \epsilon_j i\beta^{(\xi_j)}_{12}\\
        -\epsilon_j i\beta^{(\xi_j)}_{21} & 0
        \end{array}\right)
        +\mathcal{O}(\zeta^{-2}),
\end{equation}
where $\beta^{(\xi_j)}_{12}$ and $\beta^{(\xi_j)}_{21}$ are defined by \eqref{betaxi1-1}-\eqref{betaxi1-3} for $j=1,3$, while are defined by \eqref{betaxi2-1}-\eqref{betaxi2-3} for $j=2,4$.
\begin{remark}\rm
    Here we directly calculate $M^{(LC,\xi_j)}(z)$, $j=1,3$ and  $M^{(LC,\xi_j)}(z)$, $j=2,4$, because the original circular symmetry reduction  $M(z)=\mp z^{-1}M(z^{-1})\sigma_2$ is
    destroyed in these local models.
\end{remark}

Now we consider a new RH problem $M^{(LC)}(x,t;z)$, which include the contribution from $M^{(LC,\xi_j)}$, $j=1,2,3,4$.
\begin{RHP} Find a $2\times 2$ matrix-valued  function $M^{(LC)}(z)$ such that\\
        - $M^{(LC)}(z)$ is analytical off $\Sigma^{(LC)}$.\\
        - $M^{(LC)}(z)$ takes continuous boundary values $M^{(LC)}_{\pm}(z)$ on $\Sigma^{(LC)}$ with jump relation
        \begin{equation}
            M^{(LC)}_{+}(z)=M^{(LC)}_{-}(z)V^{(LC)}(z), \quad z\in\Sigma^{(LC)},
        \end{equation}
        where $V^{(LC)}(z)=V^{(2)}(z)$.\\
        - $ M^{(LC)}(x,t;z)=I+\mathcal{O}(z^{-1}), \quad  z\rightarrow\infty$.
\end{RHP}

$V^{(LC)}(z)$ admits a factorization
\begin{equation}
    V^{(LC)}(z)=\left(I-w_{jk}^{-}\right)^{-1}\left(I+w_{jk}^{+}\right),
\end{equation}
where
\begin{equation}\label{wjk}
    w_{jk}^{-}=0, \quad w_{jk}^{+}=V^{(LC)}-I,
\end{equation}
and the superscript $\pm$ indicate the analyticity in the positive and negative neighborhood of the contour respectively.

Recall the Cauchy projection operator $C_{\pm}$ on $\Sigma^{(LC)}_{jk}$, $j, k=1,2,3,4$
\begin{equation}
    C_{\pm}f(z)=\lim_{s\rightarrow z,\ z\in\Sigma^{(LC)}_{jk,\pm}}\frac{1}{2\pi i}\int_{\Sigma^{(LC)}_{jk}}\frac{f(s)}{s-z}ds.
\end{equation}
Define the following operator on $\Sigma^{(LC)}_{jk}$, $k=1,2,3,4$ as follows
\begin{equation}
    C_{w_{jk}}(f):=C_{-}(fw_{jk}^{+}).
\end{equation}
Then we review some notations as follows
\begin{align}
    &w_{j}=\sum_{k=1}^{4}w_{jk}, \quad \Sigma^{(LC, \xi_j)}=\bigcup_{k=1,2,3,4}\Sigma^{(LC)}_{jk},\\
    &w=\sum_{j=1}^{4}w_{j}=\sum_{j,k=1}^{4}w_{jk},\quad \Sigma^{(LC)}=\bigcup_{j=1,2,3,4}\Sigma^{(LC, \xi_j)}=\bigcup_{j,k=1,2,3,4}\Sigma_{jk}^{(LC)}.
\end{align}
It is obvious to find that $C_w=\sum_{j=1}^{4}C_{w_j}=\sum_{j,k=1}^{4}C_{w_{jk}}$. A direct calculation shows that
$\Vert w_{jk} \Vert_{L^{2}(\Sigma^{(LC)}_{jk})}=\mathcal{O}(t^{-1/2})$, which implies that $1-C_{w}$, $1-C_{w_j}$ and  $1-C_{w_{jk}}$ exist as $t\rightarrow+\infty$.

Via standard Beals-Coifman theory \cite{BCdecom}, we know that $M^{(LC)}$
can be uniquely shown by
\begin{equation}
    M^{(LC)}=I+C(\mu w),
\end{equation}
where $\mu\in I+L^{2}(\Sigma^{(LC)})$ is the solution of the singular integral equation $\mu=I+C_{w}(\mu)$. $M^{(LC)}(z)$ could be expressed in terms of the following integral.
\begin{equation}
    M^{(LC)}=I+\frac{1}{2\pi i}\int_{\Sigma^{(LC)}}\frac{(1-C_w)^{-1}Iw}{s-z}ds.
\end{equation}

\begin{proposition}\label{cwjcwk}
As $t\rightarrow +\infty$, for $j\neq k$
\begin{equation}
\Vert C_{w_j}C_{w_k} \Vert_{L^{2}(\Sigma^{(LC)})}=\mathcal{O}(t^{-1}), \quad \Vert C_{w_j}C_{w_k} \Vert_{L^{\infty}(\Sigma^{(LC)})\rightarrow L^{2}(\Sigma^{(LC)})}=\mathcal{O}(t^{-1}).
\end{equation}
\end{proposition}
\begin{proof}
    Thanks to the observation of Varzugin \cite{Var}, we have
    \begin{align}
        &1-\sum_{j\neq k}C_{w_j}C_{w_k}\left(1-C_{w_k}\right)^{-1}=\left(1-C_{w}\right)\left(1+\sum_{j=1}^{4}C_{w_j}\left(1-C_{w_j}\right)^{-1}\right),\\
        &1-\sum_{j\neq k}\left(1-C_{w_k}\right)^{-1}C_{w_j}C_{w_k}=\left(1+\sum_{j=1}^{4}C_{w_j}\left(1-C_{w_j}\right)^{-1}\right)\left(1-C_{w}\right).
    \end{align}
    The result follows from $\Vert w_{jk} \Vert_{L^{2}(\Sigma^{(LC)}_{jk})}=\mathcal{O}(t^{-1/2})$.
\end{proof}
The following proposition reveals that the contribution for $M^{(LC)}(z)$ could be separated by each $M^{(LC, \xi_j)}(z)$, $j=1,2,3,4$.
\begin{proposition} \label{prop:sepcontri}As $t\rightarrow +\infty$,
    \begin{equation}
        \int_{\Sigma^{(LC)}}\frac{\left(1-C_{w}\right)^{-1}Iw}{s-z}=\sum_{j=1}^{4}\int_{\Sigma^{(LC,\xi_j)}}\frac{\left(1-C_{w_{j}}\right)^{-1}Iw_{j}}{s-z}+\mathcal{O}(t^{-1}).
    \end{equation}
\end{proposition}
\begin{proof}
    Decompose the resolvent $(1-C_{w})^{-1}I$ as
\begin{equation}
    (1-C_{w})^{-1}I=I+\sum_{j=1}^{4}C_{w_j}(1-C_{w_j})^{-1}I+QPR I,
\end{equation}
where
\begin{align}
    &Q:=1+\sum_{j=1}^{4}C_{w_j}(1-C_{w_j})^{-1},\\
    &P:=\left(1-\sum_{j\neq k}C_{w_j}C_{w_k}\left(1-C_{w_k}\right)^{-1}\right)^{-1},\\
    &R:=\sum_{j\neq k}C_{w_{j}}C_{w_{k}}(1-C_{w_k})^{-1}.
\end{align}
Using Cauchy-Schwarz inequality and Proposition \ref{prop:sepcontri}, we have
\begin{equation}
\left\vert\int QPRIw \right\vert\leqslant \Vert Q\Vert_{L^{2}\left(\Sigma^{(LC, \xi_j)}\right)}
\Vert P\Vert_{L^{2}\left(\Sigma^{\left(LC, \xi_j\right)}\right)}\Vert R\Vert_{L^{2}\left(\Sigma^{\left(LC, \xi_j\right)}\right)}\Vert w \Vert_{L^{2}}\lesssim t^{-1}.
\end{equation}
\end{proof}

Combining \eqref{mmodj}, the following for $M^{(LC)}(z)$ proposition follows.
\begin{proposition}\label{mmodsol}As $t\rightarrow +\infty$, we have
    \begin{equation}
        M^{(LC)}(z)=M^{(\infty)}(z)\cdot M^{(PC)}(z),
    \end{equation}
    where
    \begin{equation}\label{compummod}
        M^{(PC)}(z)=I+t^{-\frac{1}{2}}\sum_{j=1}^{4}\frac{i\epsilon_jA^{mat}_j}{\left(2\epsilon_j\theta''(\xi_j)\right)^{\frac{1}{2}}\left(z-\xi_j\right)}+\mathcal{O}(t^{-1}),
    \end{equation}
    with
    \begin{equation}
        A^{mat}_j=\left(\begin{array}{cc}
            0 & \beta^{(\xi_j)}_{12}\\
            -\beta^{(\xi_j)}_{21} & 0
            \end{array}\right)
    \end{equation}
\end{proposition}

\subsubsection{Error: A small norm RH problem}\label{smallerro}
Define the error matrix $M^{(err)}$ by
\begin{equation}\label{equ:Merr}
    M^{(err)}(z)=\left\{
        \begin{aligned}
        &M^{(PR)}(z)M^{(\infty)}(z)^{-1}, & z\in \mathbb{C}\backslash U(\xi),\\
        &M^{(PR)}(z)M^{(LC)}(z)^{-1}, & z\in U(\xi), \\
        \end{aligned}
        \right.
\end{equation}
RH conditions for $M^{(err)}$ are as follows.
\begin{RHP}\label{errorrhp}
    Find a $2\times 2$ matrix-valued function  $M^{(err)}(z)$ such that\\
    - $M^{(err)}$ is analytical in $\mathbb{C}\backslash \Sigma^{(err)}$, where
    \begin{equation}
        \Sigma^{(err)}:=\partial U(\xi) \cup \left(\Sigma^{(2)}\backslash U(\xi)\right);
    \end{equation}
    - $M^{(err)}$ takes continuous boundary values $M^{(err)}_{\pm}(z)$ on $\Sigma^{(err)}$ and
    \begin{equation}
        M^{(err)}_{+}(z)=M^{(err)}_{-}(z)V^{(err)}(z),
    \end{equation}
    where
    \begin{equation}
    V^{(err)}(z)=\left\{
        \begin{aligned}
        &M^{(\infty)}(z)V^{(2)}(z){M^{(\infty)}(z)}^{-1}, \quad z\in \Sigma^{(2)}\backslash U(\xi),\\
        &M^{(\infty)}(z)M^{(PC)}(z){M^{(\infty)}(z)}^{-1}, \quad z\in \partial U(\xi).
        \end{aligned}
        \right.
    \end{equation}
    - $M^{(err)}=I+\mathcal{O}(z^{-1}), \quad z\rightarrow \infty$.
\end{RHP}
\begin{figure}[htbp]
    \begin{center}

        \tikzset{every picture/.style={line width=0.75pt}} 

        \begin{tikzpicture}[x=0.75pt,y=0.75pt,yscale=-1,xscale=1]

        \draw  [line width=0.75]  (142.43,139.74) .. controls (142.43,121.49) and (156.52,106.69) .. (173.91,106.69) .. controls (191.3,106.69) and (205.39,121.49) .. (205.39,139.74) .. controls (205.39,157.99) and (191.3,172.79) .. (173.91,172.79) .. controls (156.52,172.79) and (142.43,157.99) .. (142.43,139.74) -- cycle ;
        \draw   (242.43,140.7) .. controls (242.43,122.44) and (256.52,107.65) .. (273.91,107.65) .. controls (291.3,107.65) and (305.39,122.44) .. (305.39,140.7) .. controls (305.39,158.95) and (291.3,173.74) .. (273.91,173.74) .. controls (256.52,173.74) and (242.43,158.95) .. (242.43,140.7) -- cycle ;
        \draw   (343.76,138.6) .. controls (343.76,120.34) and (357.86,105.55) .. (375.24,105.55) .. controls (392.63,105.55) and (406.73,120.34) .. (406.73,138.6) .. controls (406.73,156.85) and (392.63,171.64) .. (375.24,171.64) .. controls (357.86,171.64) and (343.76,156.85) .. (343.76,138.6) -- cycle ;
        \draw [color={rgb, 255:red, 0; green, 0; blue, 0 }  ,draw opacity=1 ]   (501.52,156.38) -- (538.43,171.79) ;
        \draw [shift={(525.51,166.39)}, rotate = 202.66] [color={rgb, 255:red, 0; green, 0; blue, 0 }  ,draw opacity=1 ][line width=0.75]    (10.93,-3.29) .. controls (6.95,-1.4) and (3.31,-0.3) .. (0,0) .. controls (3.31,0.3) and (6.95,1.4) .. (10.93,3.29)   ;
        \draw   (443.76,138.6) .. controls (443.76,120.34) and (457.86,105.55) .. (475.24,105.55) .. controls (492.63,105.55) and (506.73,120.34) .. (506.73,138.6) .. controls (506.73,156.85) and (492.63,171.64) .. (475.24,171.64) .. controls (457.86,171.64) and (443.76,156.85) .. (443.76,138.6) -- cycle ;
        \draw  [dash pattern={on 0.84pt off 2.51pt}]  (126.43,139.59) -- (529.76,139.59) ;
        \draw [shift={(531.76,139.59)}, rotate = 180] [color={rgb, 255:red, 0; green, 0; blue, 0 }  ][line width=0.75]    (10.93,-3.29) .. controls (6.95,-1.4) and (3.31,-0.3) .. (0,0) .. controls (3.31,0.3) and (6.95,1.4) .. (10.93,3.29)   ;
        \draw [color={rgb, 255:red, 0; green, 0; blue, 0 }  ,draw opacity=1 ]   (504,121.39) -- (535.76,106.7) ;
        \draw [shift={(525.32,111.53)}, rotate = 155.19] [color={rgb, 255:red, 0; green, 0; blue, 0 }  ,draw opacity=1 ][line width=0.75]    (10.93,-3.29) .. controls (6.95,-1.4) and (3.31,-0.3) .. (0,0) .. controls (3.31,0.3) and (6.95,1.4) .. (10.93,3.29)   ;
        \draw [color={rgb, 255:red, 0; green, 0; blue, 0 }  ,draw opacity=1 ]   (402.47,121.39) -- (423.1,106.7) ;
        \draw [shift={(417.67,110.56)}, rotate = 144.54] [color={rgb, 255:red, 0; green, 0; blue, 0 }  ,draw opacity=1 ][line width=0.75]    (10.93,-3.29) .. controls (6.95,-1.4) and (3.31,-0.3) .. (0,0) .. controls (3.31,0.3) and (6.95,1.4) .. (10.93,3.29)   ;
        \draw [color={rgb, 255:red, 0; green, 0; blue, 0 }  ,draw opacity=1 ]   (423.1,106.7) -- (448.43,120) ;
        \draw [shift={(441.07,116.14)}, rotate = 207.69] [color={rgb, 255:red, 0; green, 0; blue, 0 }  ,draw opacity=1 ][line width=0.75]    (10.93,-3.29) .. controls (6.95,-1.4) and (3.31,-0.3) .. (0,0) .. controls (3.31,0.3) and (6.95,1.4) .. (10.93,3.29)   ;
        \draw [color={rgb, 255:red, 0; green, 0; blue, 0 }  ,draw opacity=1 ]   (402.47,156.38) -- (427.76,165.49) ;
        \draw [shift={(420.76,162.97)}, rotate = 199.8] [color={rgb, 255:red, 0; green, 0; blue, 0 }  ,draw opacity=1 ][line width=0.75]    (10.93,-3.29) .. controls (6.95,-1.4) and (3.31,-0.3) .. (0,0) .. controls (3.31,0.3) and (6.95,1.4) .. (10.93,3.29)   ;
        \draw [color={rgb, 255:red, 0; green, 0; blue, 0 }  ,draw opacity=1 ]   (427.76,165.49) -- (447.1,153.59) ;
        \draw [shift={(442.54,156.39)}, rotate = 148.39] [color={rgb, 255:red, 0; green, 0; blue, 0 }  ,draw opacity=1 ][line width=0.75]    (10.93,-3.29) .. controls (6.95,-1.4) and (3.31,-0.3) .. (0,0) .. controls (3.31,0.3) and (6.95,1.4) .. (10.93,3.29)   ;
        \draw [color={rgb, 255:red, 0; green, 0; blue, 0 }  ,draw opacity=1 ]   (298.47,121.39) -- (324.43,103.9) ;
        \draw [shift={(316.43,109.29)}, rotate = 146.03] [color={rgb, 255:red, 0; green, 0; blue, 0 }  ,draw opacity=1 ][line width=0.75]    (10.93,-3.29) .. controls (6.95,-1.4) and (3.31,-0.3) .. (0,0) .. controls (3.31,0.3) and (6.95,1.4) .. (10.93,3.29)   ;
        \draw [color={rgb, 255:red, 0; green, 0; blue, 0 }  ,draw opacity=1 ]   (324.43,103.9) -- (348.43,118.6) ;
        \draw [shift={(341.55,114.38)}, rotate = 211.48] [color={rgb, 255:red, 0; green, 0; blue, 0 }  ,draw opacity=1 ][line width=0.75]    (10.93,-3.29) .. controls (6.95,-1.4) and (3.31,-0.3) .. (0,0) .. controls (3.31,0.3) and (6.95,1.4) .. (10.93,3.29)   ;
        \draw [color={rgb, 255:red, 0; green, 0; blue, 0 }  ,draw opacity=1 ]   (326.43,171.79) -- (347.76,155.69) ;
        \draw [shift={(341.88,160.12)}, rotate = 142.96] [color={rgb, 255:red, 0; green, 0; blue, 0 }  ,draw opacity=1 ][line width=0.75]    (10.93,-3.29) .. controls (6.95,-1.4) and (3.31,-0.3) .. (0,0) .. controls (3.31,0.3) and (6.95,1.4) .. (10.93,3.29)   ;
        \draw [color={rgb, 255:red, 0; green, 0; blue, 0 }  ,draw opacity=1 ]   (302.43,157.09) -- (326.43,171.79) ;
        \draw [shift={(319.55,167.57)}, rotate = 211.48] [color={rgb, 255:red, 0; green, 0; blue, 0 }  ,draw opacity=1 ][line width=0.75]    (10.93,-3.29) .. controls (6.95,-1.4) and (3.31,-0.3) .. (0,0) .. controls (3.31,0.3) and (6.95,1.4) .. (10.93,3.29)   ;
        \draw [color={rgb, 255:red, 0; green, 0; blue, 0 }  ,draw opacity=1 ]   (199.14,122.09) -- (223.1,104.6) ;
        \draw [shift={(215.96,109.81)}, rotate = 143.87] [color={rgb, 255:red, 0; green, 0; blue, 0 }  ,draw opacity=1 ][line width=0.75]    (10.93,-3.29) .. controls (6.95,-1.4) and (3.31,-0.3) .. (0,0) .. controls (3.31,0.3) and (6.95,1.4) .. (10.93,3.29)   ;
        \draw [color={rgb, 255:red, 0; green, 0; blue, 0 }  ,draw opacity=1 ]   (223.1,104.6) -- (247.76,120.7) ;
        \draw [shift={(240.45,115.93)}, rotate = 213.13] [color={rgb, 255:red, 0; green, 0; blue, 0 }  ,draw opacity=1 ][line width=0.75]    (10.93,-3.29) .. controls (6.95,-1.4) and (3.31,-0.3) .. (0,0) .. controls (3.31,0.3) and (6.95,1.4) .. (10.93,3.29)   ;
        \draw [color={rgb, 255:red, 0; green, 0; blue, 0 }  ,draw opacity=1 ]   (225.1,174.59) -- (246.43,159.19) ;
        \draw [shift={(240.63,163.38)}, rotate = 144.18] [color={rgb, 255:red, 0; green, 0; blue, 0 }  ,draw opacity=1 ][line width=0.75]    (10.93,-3.29) .. controls (6.95,-1.4) and (3.31,-0.3) .. (0,0) .. controls (3.31,0.3) and (6.95,1.4) .. (10.93,3.29)   ;
        \draw [color={rgb, 255:red, 0; green, 0; blue, 0 }  ,draw opacity=1 ]   (200.43,158.49) -- (225.1,174.59) ;
        \draw [shift={(217.79,169.82)}, rotate = 213.13] [color={rgb, 255:red, 0; green, 0; blue, 0 }  ,draw opacity=1 ][line width=0.75]    (10.93,-3.29) .. controls (6.95,-1.4) and (3.31,-0.3) .. (0,0) .. controls (3.31,0.3) and (6.95,1.4) .. (10.93,3.29)   ;
        \draw [color={rgb, 255:red, 0; green, 0; blue, 0 }  ,draw opacity=1 ]   (118.43,107.4) -- (144.9,126.3) ;
        \draw [shift={(136.55,120.33)}, rotate = 215.51] [color={rgb, 255:red, 0; green, 0; blue, 0 }  ,draw opacity=1 ][line width=0.75]    (10.93,-3.29) .. controls (6.95,-1.4) and (3.31,-0.3) .. (0,0) .. controls (3.31,0.3) and (6.95,1.4) .. (10.93,3.29)   ;
        \draw [color={rgb, 255:red, 0; green, 0; blue, 0 }  ,draw opacity=1 ]   (121.76,175.29) -- (146.9,159.19) ;
        \draw [shift={(139.39,164)}, rotate = 147.37] [color={rgb, 255:red, 0; green, 0; blue, 0 }  ,draw opacity=1 ][line width=0.75]    (10.93,-3.29) .. controls (6.95,-1.4) and (3.31,-0.3) .. (0,0) .. controls (3.31,0.3) and (6.95,1.4) .. (10.93,3.29)   ;
        \draw   (172.74,99.74) .. controls (175.47,103.81) and (178.53,106.69) .. (181.93,108.39) .. controls (178.2,107.87) and (174.14,108.52) .. (169.74,110.36) ;
        \draw   (275.41,101.14) .. controls (278.14,105.21) and (281.2,108.09) .. (284.6,109.79) .. controls (280.87,109.26) and (276.81,109.92) .. (272.4,111.76) ;
        \draw   (372.08,98.35) .. controls (374.8,102.4) and (377.87,105.29) .. (381.27,106.99) .. controls (377.54,106.47) and (373.47,107.12) .. (369.07,108.96) ;
        \draw   (472.74,98.35) .. controls (475.47,102.41) and (478.53,105.29) .. (481.93,106.99) .. controls (478.2,106.47) and (474.14,107.12) .. (469.74,108.96) ;
        \draw (517.07,146.47) node [anchor=north west][inner sep=0.75pt]  [font=\scriptsize]  {$Re\ z$};
        \draw (167.45,143.6) node [anchor=north west][inner sep=0.75pt]  [font=\scriptsize]  {$\xi _{4}$};
        \draw (267.45,144.56) node [anchor=north west][inner sep=0.75pt]  [font=\scriptsize]  {$\xi _{3}$};
        \draw (368.78,142.46) node [anchor=north west][inner sep=0.75pt]  [font=\scriptsize]  {$\xi _{2}$};
        \draw (468.78,142.46) node [anchor=north west][inner sep=0.75pt]  [font=\scriptsize]  {$\xi _{1}$};
        \end{tikzpicture}
        \caption{Jump contour of $M^{(err)}(z)$.}\label{fig:jump of M^{err}}
        \end{center}
    \end{figure}
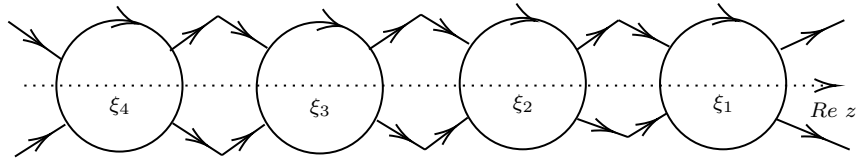

Taking into account Proposition \ref{outxi1} and Proposition \ref{outxi2}, we can know that $V^{(err)}(z)$ exponentially decay to $I$ for
$z\in\Sigma^{(2)}\backslash U(\xi)$ and $z\in \Sigma^{(1/2)}_{j\pm}$, $j=1,2,3,4$.
For $z\in \partial U(\xi)$, $M^{(\infty)}(z)$ is bounded, we obtain that
\begin{align}\label{verrest}
    \vert V^{(err)}-I \vert&=\vert M^{(\infty)}(z)M^{(PC)}(z){M^{(\infty)}(z)}^{-1}-I \vert  \nonumber \\
    &= \vert M^{(\infty)}(z)(M^{(PC)}(z)-I){M^{(\infty)}(z)}^{-1} \vert \nonumber\\
    &\overset{\eqref{compummod}}{=}\mathcal{O}(t^{-\frac{1}{2}}).
\end{align}
According to Beals-Coifman theory, the solution for $M^{(err)}$ can be given by
\begin{equation}\label{BCsoluformerr}
    M^{(err)}=I+\frac{1}{2\pi i}\int_{\Sigma^{(err)}}\frac{(I+\mu(s))(V^{(err)}(s)-I)}{s-z}ds,
\end{equation}
where $\mu\in L^{2}(\Sigma^{(err)})$ is the unique solution of $(1-C_{V^{(err)}})\mu=C_{V^{(err)}}I$. And $C_{V^{(err)}}$: $L^{2}(\Sigma^{(err)})\rightarrow L^{2}(\Sigma^{(err)})$ is
the Cauchy projection operator on $\Sigma^{(err)}$:
\begin{equation}
    C_{V^{(err)}} (f)(z)=C_{-}{f(V^{(err)}-I)}=\lim_{s\rightarrow z,\ z\in\Sigma^{(err)}}\int_{\Sigma^{(err)}}\frac{f(s)(V^{(err)}(s)-I)}{s-z}ds.
\end{equation}
Existence and uniqueness of $\mu$ follows from the boundedness of the Cauchy projection operator $C_{-}$, which implies
\begin{equation}
    \Vert C_{V^{(err)}}\Vert \leqslant \Vert C_{-} \Vert_{L^{2}(\Sigma^{(err)})\rightarrow L^{2}(\Sigma^{(err)})}\Vert V^{err}-I \Vert_{L^{2}(\Sigma^{(err)})}=\mathcal{O}(t^{-\frac{1}{2}}).
\end{equation}
Moreover,
\begin{equation}\label{erromuest}
    \Vert \mu \Vert_{L^{2}\Sigma^{(err)}}\lesssim \frac{\Vert C_{V^{(err)}}\Vert}{1-\Vert C_{V^{(err)}}\Vert}\lesssim t^{-\frac{1}{2}}.
\end{equation}

Now we present the following proposition, which is helpful for the last asymptotics.
\begin{proposition}\label{E1est}
As $t\rightarrow\infty$, we have
\begin{equation}\label{equ:M(err)_1}
M^{(err)}_1=t^{-1/2}\sum_{j=1}^{4}i\epsilon_j\left(2\epsilon_j\theta''(\xi_j) \right)^{-\frac{1}{2}}\left(1-\xi_j^{-2}\right)^{-1}
\begin{pmatrix}
    -\frac{i}{\xi_j}\left(\beta^{(\xi_j)}_{12}-\beta^{(\xi_j)}_{21}\right) & \beta^{(\xi_j)}_{12}-\frac{1}{\xi_j^2}\beta^{(\xi_j)}_{21}\\
    \frac{1}{\xi_j^2}\beta^{(\xi_j)}_{12}-\beta^{(\xi_j)}_{21} & \frac{i}{\xi_j}\left(\beta^{(\xi_j)}_{12}-\beta^{(\xi_j)}_{21}\right)
\end{pmatrix}
+\mathcal{O}{(t^{-1})}.
\end{equation}
where $M^{(err)}_1$ comes from the asymptotic expansion of $M^{(err)}$ as $z\rightarrow\infty$
\begin{align}
M^{(err)}=I+z^{-1}M^{(err)}_1+\mathcal{O}(z^{-2}).
\end{align}
Moreover, the $(1,2)$-entry of $M^{(err)}_1$, denoted by $\left(M^{(err)}_1\right)_{12}$, is given by
\begin{equation}\label{equ:M^(err)_1}
    \left(M^{(err)}_1\right)_{12}=t^{-1/2}\sum_{j=1}^{4}i\epsilon_j\left(2\epsilon_j\theta''(\xi_j) \right)^{-\frac{1}{2}}\left(1-\xi_j^{-2}\right)^{-1}
    \cdot \left(\beta^{(\xi_j)}_{12}-\frac{1}{\xi_j^2}\beta^{(\xi_j)}_{21}\right)+\mathcal{O}(t^{-1}).
\end{equation}
\end{proposition}
\begin{proof}Recalling \eqref{BCsoluformerr}, we know that
    \begin{equation}
        M^{(err)}_{1}=-\frac{1}{2\pi i}\int_{\Sigma^{(err)}}\left(I+\mu(s)\right)(V^{(err)}(s)-I)ds:=I_1+I_2+I_3.
    \end{equation}
    where
    \begin{align}
        &I_{1}=-\frac{1}{2\pi i}\oint_{\partial U(\xi)}\left(V^{err}(s)-I\right)ds,\\
        &I_{2}=-\frac{1}{2\pi i}\int_{\Sigma^{(err)}\backslash U(\xi)}\left(V^{err}(s)-I\right)ds,\\
        &I_{3}=-\frac{1}{2\pi i}\int_{\Sigma^{(err)}}\mu(s)(V^{err}(s)-I)ds.
    \end{align}
Using Proposition \ref{outxi1} and Proposition \ref{outxi2}, we obtain $\vert I_2 \vert=\mathcal{O}(t^{-1})$.

Using \eqref{erromuest} and \eqref{verrest}, we obtain
\begin{equation}
    \vert  I_3 \vert \lesssim \Vert\mu \Vert_{L^{2}} \Vert V^{err}-I \Vert_{L^{2}}\lesssim t^{-1}.
\end{equation}
Finally, we deal with $I_1$
\begin{align}
    I_{1}&=-\frac{1}{2\pi i}\oint_{\partial U(\xi)}M^{(\infty)}(s)\left(M^{(PC)}(s)-I\right)M^{(\infty)}(s)^{-1}ds\nonumber\\
    &\overset{\eqref{compummod}}{=}-\frac{1}{2\pi i}\sum_{j=1}^{4}\oint_{\partial U_{\varrho}(\xi_j)}\frac{i\epsilon_j}{(2t\epsilon_j\theta''(\xi_j))^{\frac{1}{2}}(z-\xi_j)}M^{(\infty)}(s)A^{mat}_{j}(s){M^{(\infty)}(s)}^{-1}ds+\mathcal{O}(t^{-1})\nonumber\\
    {\rm (residue\hspace{0.2em} theorem)}&=t^{-1/2}\sum_{j=1}^{4}i\epsilon_j\left(2\epsilon_j\theta''(\xi_j) \right)^{-\frac{1}{2}}M^{(\infty)}(\xi_j)A^{mat}_{j}{M^{(\infty)}(\xi_j)}^{-1}+\mathcal{O}(t^{-1})\nonumber\\
    &=t^{-1/2}\sum_{j=1}^{4}i\epsilon_j\left(2\epsilon_j\theta''(\xi_j) \right)^{-\frac{1}{2}}\left(1-\xi_j^{-2}\right)^{-1}\left(I+\frac{\sigma_2}{\xi_j}\right)A_j^{mat}\left(I-\frac{\sigma_2}{\xi_j}\right)+\mathcal{O}(t^{-1})\nonumber\\
    &=t^{-1/2}\sum_{j=1}^{4}i\epsilon_j\left(2\epsilon_j\theta''(\xi_j) \right)^{-\frac{1}{2}}\left(1-\xi_j^{-2}\right)^{-1}
    \begin{pmatrix}
        -\frac{i}{\xi_j}\left(\beta^{(\xi_j)}_{12}-\beta^{(\xi_j)}_{21}\right) & \beta^{(\xi_j)}_{12}-\frac{1}{\xi_j^2}\beta^{(\xi_j)}_{21}\\
        \frac{1}{\xi_j^2}\beta^{(\xi_j)}_{12}-\beta^{(\xi_j)}_{21} & \frac{i}{\xi_j}\left(\beta^{(\xi_j)}_{12}-\beta^{(\xi_j)}_{21}\right)
    \end{pmatrix}+\mathcal{O}(t^{-1})
\end{align}
Summarizing $I_1$, $I_2$ and $I_3$, we obtain \eqref{equ:M(err)_1}. $(M^{(err)}_1)_{12}$ follows immediately.
\end{proof}

\subsection{Analysis on pure $\bar{\partial}$ problem}\label{subsec: RLpuredbar}
Define
\begin{equation}\label{trans4}
    M^{(3)}(z)=M^{(2)}(z){M^{(PR)}(z)}^{-1}.
\end{equation}
Then $M^{(3)}$ satisfies the following $\bar{\partial}$ problem.
\begin{Dbar}\label{dbarproblem}
    Find a $2\times 2$ matrix-valued function $M^{(3)}(z)$ such that\\
        - $M^{(3)}(z)$ is continuous in $\mathbb{C}$ and analytic in $\mathbb{C}\backslash\overline{\Omega}$.\\
        - $M^{(3)}(z)=I+\mathcal{O}(z^{-1}), \quad z\rightarrow\infty$.\\
        - For $z\in \mathbb{C}$, $\bar{\partial}M^{(3)}(z)=M^{(3)}(z)W^{(3)}(z)$ with
        $W^{(3)}=M^{(PR)}(z)\bar{\partial}R^{(2)}(z){M^{(PR)}(z)}^{-1}$.
\end{Dbar}
\begin{proof}
    It's enough to prove the following claims.

    $M^{(3)}$ has no jumps. Indeed, since $M^{(2)}$ and $M^{(PR)}$ take the same jump matrix, we have
    \begin{align}
        {M^{(3)}_{-}(z)}^{-1}M^{(3)}_{+}(z)&=M^{(PR)}_{-}(M^{(2)}_{-})^{-1}M^{(2)}_{+}(M_{+}^{(PR)})^{-1} \nonumber\\
        &=M^{(PR)}_{-}V^{(2)}(M_{+}^{(PR)})^{-1}=I.
    \end{align}

    $M^{(3)}$ has no singularity at $z=0$. Near $z=0$, we have
    \begin{equation}
        (M^{(PR)})^{-1}=\frac{\sigma_2(M^{(PR)})^{\rm T}\sigma_2}{1-z^{-2}}.
    \end{equation}
    Thus
    \begin{equation}
        \lim_{z\rightarrow 0}M^{(3)}=\lim_{z\rightarrow 0}\frac{M^{(2)}\sigma_2(M^{(PR)})^{\rm T}\sigma_2}{1-z^{-2}}=\mathcal{O}(1).
    \end{equation}

    $M^{(3)}$ has no singularities at $z=\pm 1$. Indeed, as $z\rightarrow\pm 1$,
    \begin{equation}
        V(z)=\begin{pmatrix} c & \mp ic \\ \pm i\bar{c} & \bar{c} \end{pmatrix}, \quad
        {M^{(PR)}(z)}^{-1}=\frac{\pm 1}{2(z\mp 1)}\sigma_1\begin{pmatrix} \gamma & \pm i\bar{\gamma} \\ \mp i\gamma & \bar{\gamma} \end{pmatrix}\sigma_1+\mathcal{O}(1),
    \end{equation}
    for some constants $c$ and $\gamma$. Thus we have $\lim_{z\rightarrow\pm 1}M^{(3)}(z)=\mathcal{O}(1)$.
\end{proof}

The solution of $\bar{\partial}$-Problem \ref{dbarproblem} can be solved by the following integral equation
\begin{equation}\label{intm3}
    M^{(3)}(z)=I-\frac{1}{\pi}\iint_{\mathbb{C}}\frac{M^{(3)}(s)W^{(3)}(s)}{s-z}dA(s),
\end{equation}
where $A(s)$ is the Lebesgue measure on $\mathbb{C}$. Denote $S$ be the Cauchy-Green integral operator
\begin{equation}\label{C-Gop}
    S[f](z)=-\frac{1}{\pi}\iint\frac{f(s)W^{(3)}(s)}{s-z}dA(s),
\end{equation}
then \eqref{intm3} can be written as the following operator-valued equation
\begin{equation}
    (1-S)M^{(3)}(z)=I.
\end{equation}
To prove the existence of the operator at large time, we present the following proposition.
\begin{proposition}\label{estS} Consider the operator $S$ defined by \eqref{C-Gop},
    then we have $S:L^{\infty}(\mathbb{C})\rightarrow L^{\infty}(\mathbb{C})\cap C^{0}(\mathbb{C})$ and
    \begin{equation}
        \Vert S \Vert_{L^{\infty}(\mathbb{C})\rightarrow L^{\infty}(\mathbb{C})}\lesssim t^{-\frac{1}{4}}.
    \end{equation}
\end{proposition}
\begin{proof}
    For any $f\in L^{\infty}$, we have
    \begin{equation}
    \Vert Sf\Vert_{L^{\infty}}\leqslant\Vert f \Vert_{L^{\infty}}\frac{1}{\pi}\iint_{\mathbb{C}}\frac{|W^{(3)}(s)|}{|s-z|}dA(s).
    \end{equation}
Recalling the definition $W^{(3)}=M^{(PR)}(z)\bar{\partial}R^{(2)}(z)(M^{(PR)}(z))^{-1}$ and $\bar{\partial}R^{(2)}$, we know that
$W^{(3)}(z)\equiv 0$ for $z \in \mathbb{C}\backslash\overline{\Omega}$. Besides, we only take into account that matrix-valued functions have
support in sector $\Omega_{jk}$. Based on these conditions, what we need to do is to control
the boundedness of the integral $\iint_{\mathbb{C}}\frac{|W^{(3)}(s)|}{|s-z|}dA(s)$ for $z\in \Omega_{jk}$, $j=0^{\pm},1,2,3,4$,
$k=1,2,3,4$. We present details for $z\in\Omega_{0^{+}1}$, $z\in\Omega_{24}$ and $z\in\Omega_{11}$, the proofs for the rest regions
are similar.

Since det$M^{(PR)}(z)=1-z^{-2}$ and ${M^{(PR)}(z)}^{-1}=(1-z^{-2})^{-1}\sigma_2 (M^{(PR)})^{\rm T}\sigma_2$, we have
\begin{equation}
    \vert W^{(3)}(s)\vert \leqslant \vert M^{(PR)}(s)\vert^{2}\vert 1-s^{-2} \vert^{-1} \vert\bar{\partial} R^{(2)}(s)\vert.
\end{equation}
Next we estimate $M^{(PR)}$ as follows
\begin{equation}
    \vert M^{(PR)}(s) \vert\lesssim 1+|s|^{-1}\lesssim c\sqrt{(1+|s|^{-1})^2}\lesssim \sqrt{1+|s|^{-2}}=|s|^{-1}\sqrt{1+|s|^2}=|s|^{-1}\langle s\rangle.
\end{equation}
Since $z=1\in \Omega_{24}$, we have
\begin{equation}\label{<s>est}
\frac{|s|^{-2}\langle s\rangle^2}{\vert 1-s^{-2} \vert}=\frac{\langle s\rangle^2}{\vert 1-s^2\vert}
=\left\{
        \begin{aligned}
        &\mathcal{O}(1), \quad z\in\Omega_{0^{+}1}, \Omega_{11},\\
        &\frac{\langle s\rangle}{\vert s-1\vert}, \quad z\in\Omega_{24}.
        \end{aligned}
        \right.
\end{equation}
For $z\in \Omega_{0^{+}1}$ and $\Omega_{11}$, the estimations are similar to \cite{Dieng}. However, for $z\in\Omega_{24}$, the singularities at $z=\pm 1$
should be treated in a more delicate way. To some extent, how to deal with the singularity points plays a core role in our analysis.

Introduce an inequality which plays an vital role in our analysis. Making $s=z_{0}+le^{i\phi}=z_{0}+u+iv$, $z=x+iy$, $u,v,x,y>0$, we have
\begin{align}\label{vital role inequ}
\left\Vert \frac{1}{s-z} \right\Vert _{L^{q}(v,\infty)}&\lesssim\left(\int_{\mathbb{R}_{+}}\left[1+\left(\frac{u+z_0-x}{v-y}\right)^2\right]^{-\frac{q}{2}}(v-y)^{-q}du\right)^{\frac{1}{q}}\nonumber\\
&=|v-y|^{\frac{1}{q}-1}\left(\int_{\mathbb{R}_{+}}\left[1+\left(\frac{u+z_0-x}{v-y}\right)^2\right]^{-\frac{q}{2}}d\left(\frac{u+z_{0}-x}{v-y}\right)\right)^{1/q}\nonumber\\
&\overset{q>1}{\lesssim} \vert v-y \vert^{1/q-1}.
\end{align}
For $z\in \Omega_{0^{+}1}$, we make $z=x+iy$, $s=0+u+iv$. Thanks to \eqref{<s>est}, we have
\begin{equation}
    \frac{1}{\pi}\iint_{\Omega_{0^{+}1}}\frac{|W^{(3)}(s)|}{|s-z|}dA(s)\lesssim\frac{1}{\pi}\iint_{\Omega_{0^{+}1}}
    \frac{|\bar{\partial}f_{0^{+}1}e^{2it\theta}|}{|s-z|}dA(s)=\frac{1}{\pi}\iint_{\Omega_{0^{+}1}}
    \frac{|\bar{\partial}f_{0^{+}1}|e^{-2t\Im \theta}}{|s-z|}dA(s).
\end{equation}
Recalling Proposition \ref{estopenlesz=0}, we can divide the integral into two parts
\begin{equation}
    \frac{1}{\pi}\iint_{\Omega_{0^{+}1}}\frac{|\bar{\partial}f_{0^{+}1}|e^{-2t\Im \theta}}{|s-z|}dA(s)\lesssim I_{1}+I_{2},
\end{equation}
where
\begin{equation}
    I_1=\iint_{\Omega_{0^{+}1}}\frac{|p'_{0^{+}1}(s)|e^{-2t\Im\theta}}{|s-z|}dA(s), \quad
    I_2=\iint_{\Omega_{0^{+}1}}\frac{|s|^{-\frac{1}{2}}e^{-2t\Im\theta}}{|s-z|}dA(s).
\end{equation}
Notice that $\Omega_{0^{+}1}$ is a bounded area, $0<x,u<\xi_2/2$ and $0<y,v<\frac{\xi_2{\rm tan}\phi}{2}$($<\frac{\xi_2}{2}$).
Thus
\begin{align}
I_1& \overset{Cor\ref{corz=0}}{\leqslant} \int_{\mathbb{R}_{+}}\int_{\mathbb{R}_{+}}\frac{|p'_{0^{+}1}|}{|s-z|}e^{-ctv}dudv\nonumber\\
&\lesssim \int_{\mathbb{R}_{+}}\left\Vert \frac{1}{s-z} \right\Vert _{L^{2}(\mathbb{R}_{+})}\Vert p'_{0^{+}1} \Vert_{L^{2}(\mathbb{R}_{+})}e^{-ctv}dv\nonumber\\
&\lesssim \int_{\mathbb{R}_{+}}\vert v-y \vert^{-1/2}e^{-ctv}dv\nonumber\\
&\lesssim t^{-\frac{1}{2}}.
\end{align}
Next, we introduce the following inequality for $p>2$
\begin{align}\label{est-1/2}
    \left\Vert |s|^{-\frac{1}{2}} \right\Vert _{L^{p}(v,+\infty)}&=\left(\int_{v}^{+\infty}\left(\sqrt{u^2+v^2}\right)^{-p/2}du\right)^{1/p}\nonumber\\
    &\overset{l^2=u^2+v^2}{=}\left(\int_{v}^{+\infty}l^{-p/2}\cdot l\cdot u^{-1}dl\right)^{\frac{1}{p}}\nonumber\\
    &\lesssim v^{-\frac{1}{2}+\frac{1}{p}}.
\end{align}
By H\"older inequality with $1/p+1/q=1$,
\begin{align}\label{toolI2}
I_2&\leqslant \int_{\mathbb{R}_{+}} \left\Vert \frac{1}{s-z} \right\Vert _{L^{q}(\mathbb{R}_{+})} \Vert|s|^{-1/2} \Vert_{L^{p}(\mathbb{R}_{+})}e^{-ctv}dv\nonumber\\
&\lesssim \int_{\mathbb{R}_{+}}|v-y|^{1/q-1}v^{-\frac{1}{2}+\frac{1}{p}}e^{-ctv}dv\nonumber\\
&\lesssim t^{-\frac{1}{2}}.
\end{align}
For $z\in \Omega_{0^{\pm}k}$, $k=1,2$,
we can conclude that $\Vert S \Vert_{L^{\infty}(\mathbb{C})\rightarrow L^{\infty}(\mathbb{C})}\lesssim t^{-\frac{1}{2}}$.

For $z\in \Omega_{11}$, we make $z=x+iy$, $s=\xi_1+u+iv$.
Thanks to \eqref{<s>est}, we have
\begin{equation}
    \frac{1}{\pi}\iint_{\Omega_{11}}\frac{|W^{(3)}(s)|}{|s-z|}dA(s)\lesssim
    \frac{1}{\pi}\iint_{\Omega_{11}}\frac{|\bar{\partial}f_{11}|e^{-2t\Im \theta}}{|s-z|}dA(s).
\end{equation}
Recalling Proposition \ref{estopenlenssaddle}, we still divide the integral into two parts
\begin{equation}
    I_3=\iint_{\Omega_{11}}\frac{|p'_{11}(s)|e^{-2t\Im\theta}}{|s-z|}dA(s), \quad
    I_4=\iint_{\Omega_{11}}\frac{|s-\xi_1|^{-\frac{1}{2}}e^{-2t\Im\theta}}{|s-z|}dA(s).
\end{equation}
For $I_3$, we use Proposition \ref{lemz=xi1}
\begin{align}
    I_3&=\int_{0}^{\infty}\int_{\frac{v}{{\rm tan}\phi}}^{\infty}\frac{|p'_{11}|e^{-ctv(\Re s-\xi_1)}}{|s-z|}dudv\nonumber\\
    &\lesssim \int_{0}^{\infty}\int_{v}^{\infty}\frac{|p'_{11}|e^{-ctvu}}{|s-z|}dudv\nonumber\\
    &\lesssim \int_{0}^{\infty}e^{-ctv^2}\Vert p'_{11}\Vert_{L^{2}(v,+\infty)} \left\Vert\frac{1}{s-z}\right\Vert _{L^{2}(v,+\infty)}dv\nonumber\\
    &\overset{\eqref{vital role inequ}}{\lesssim}\int_{0}^{+\infty}e^{-4tv^2}|v-y|^{-1/2}dv\lesssim t^{-1/4}.
\end{align}
For $I_4$, we have
\begin{align}
    I_4&\lesssim \int_{0}^{\infty}e^{-ctv^2}dv \int_{v}^{+\infty}\frac{\vert s-\xi_2\vert^{-\frac{1}{2}}}{\vert s-z\vert}du\nonumber\\
    &\leqslant \int_{0}^{+\infty}e^{-ctv^2}\left\Vert \frac{1}{s-z} \right\Vert _{L^{q}(\mathbb{R}_{+})}\left\Vert |s-\xi_1|^{-\frac{1}{2}} \right\Vert _{L^{p}(\mathbb{R}_{+})}dv\nonumber\\
    &=\left(\int_{0}^{y}+\int_{y}^{+\infty}\right)v^{-\frac{1}{2}+\frac{1}{p}}|v-y|^{1/q-1}e^{-ctv^2}dv.
\end{align}
For the first integral, we have
\begin{align}
    \int_{0}^{y}v^{-\frac{1}{2}+\frac{1}{p}}|v-y|^{1/q-1}e^{-ctv^2}dv&=\int_{0}^{1}\sqrt{y}e^{-cty^2w^2}w^{1/p-1/2}\vert 1-w \vert^{1/q-1}\nonumber\\
    &\lesssim t^{-\frac{1}{4}}.
\end{align}
For the second integral, taking $v=y+w$, we have
\begin{align}
\int_{y}^{+\infty}e^{-ctv^2}v^{\frac{1}{p}-\frac{1}{2}}\vert v-y \vert^{1/q-1}dv&=\int_{0}^{+\infty}e^{-t(y+w)^2}(y+w)^{1/p-1/2}w^{1/q-1}dw\nonumber\\
&\leqslant \int_{0}^{+\infty}e^{-ctw^2}w^{1/p-1/2}w^{1/q-1}dw\nonumber\\
&=\int_{0}^{+\infty}e^{-tw^2}w^{-\frac{1}{2}}dw\lesssim t^{-\frac{1}{4}}.
\end{align}
For $z\in\Omega_{jk}$, $j=1,4$, $k=1,2,3,4$, we can conclude that
$\Vert S \Vert_{L^{\infty}(\mathbb{C})\rightarrow L^{\infty}(\mathbb{C})}\lesssim t^{-\frac{1}{4}}$.

For $z\in \Omega_{24}$, we make $z=x+iy$, $s=\xi_2+u+iv$. Since $\Omega_{24}$ is a bounded domain,
we find that $0<u<\frac{\xi_1-\xi_2}{2}$, $0<v<\frac{(\xi_1-\xi_2){\rm tan}\phi}{2}$($<\frac{\xi_1-\xi_2}{2}$).
Owe to \eqref{<s>est}, we have
\begin{equation}
    \frac{1}{\pi}\iint_{\Omega_{24}}\frac{|W^{(3)}(s)|}{|s-z|}dA(s)
    \lesssim\frac{1}{\pi}\iint_{\Omega_{24}}\frac{\langle s\rangle|\bar{\partial}f_{24}e^{-2it\theta}|}{|s-z||s-1|}dA(s)
    =\frac{1}{\pi}\iint_{\Omega_{24}}\frac{\langle s\rangle|\bar{\partial}f_{24}|e^{2t\Im \theta}}{|s-z||s-1|}dA(s).
\end{equation}
Furthermore,
\begin{equation}
    \frac{1}{\pi}\iint_{\Omega_{24}}\frac{\langle s\rangle|\bar{\partial}f_{24}|e^{2t\Im \theta}}{|s-z||s-1|}dA(s)\lesssim I_{5}+I_{6},
\end{equation}
where
\begin{align}
    &I_5=\frac{1}{\pi}\iint_{\Omega_{24}}\frac{\langle s\rangle|\bar{\partial}f_{24}|e^{2t\Im \theta}\chi_{[\xi_2,1)}(|s|)}{|s-z||s-1|}dA(s), \\
    &I_6=\frac{1}{\pi}\iint_{\Omega_{24}}\frac{\langle s\rangle|\bar{\partial}f_{24}|e^{2t\Im \theta}\chi_{\left[1,\frac{\xi_1-\xi_2}{\sqrt{2}}\right)}(|s|)}{|s-z||s-1|}dA(s).
\end{align}
where $\chi_{[\xi_2,1)}(|s|)+\chi_{\left[1,\frac{\xi_1-\xi_2}{\sqrt{2}}\right)}(|s|)$ is the partition of unity.

We consider $I_5$ firstly. Since $|s|\in [\xi_2,1)$, we have $\frac{\langle s \rangle}{|s-1|}=\mathcal{O}(1)$. Recalling Proposition \ref{estopenlenssaddle}, we have
\begin{equation}
    I_5\leqslant I_{5}^{(1)}+I_{5}^{(2)},
\end{equation}
where
\begin{equation}
    I_5^{(1)}=\iint_{\Omega_{24}}\frac{|p'_{24}(s)|e^{2t\Im\theta}}{|s-z|}dA(s), \quad
    I_5^{(2)}=\iint_{\Omega_{24}}\frac{|s-\xi_2|^{-\frac{1}{2}}e^{2t\Im\theta}}{|s-z|}dA(s).
\end{equation}
Proposition \ref{lemz=xi2} implies that
\begin{align}
    I_5^{(1)}&\leqslant \int_{0}^{+\infty}\int_{\frac{v}{{\rm tan}\phi}}^{+\infty}\frac{|p_{24}'(s)|}{|s-z|}e^{-ct(1+|s|^{-2})v^2}dudv\nonumber\\
    &\lesssim \int_{\mathbb{R}_{+}}\left\Vert \frac{1}{s-z} \right\Vert_{L^{2}(\mathbb{R}_{+})}\Vert p'_{24} \Vert_{L^{2}(\mathbb{R}_{+})}e^{-ctv^2}dv\nonumber\\
    &\lesssim \left(\int_{0}^{y}+\int_{y}^{+\infty}\right)|v-y|^{-\frac{1}{2}}e^{-ctv^2}dv.
\end{align}
Then
\begin{align}
    \int_{0}^{y}(y-v)^{-\frac{1}{2}}e^{-ctv^2}dv
    &\lesssim \int_{0}^{y}(y-v)^{-\frac{1}{2}}v^{-\frac{1}{2}}dv\cdot t^{-\frac{1}{4}}\nonumber\\
    &\lesssim t^{-\frac{1}{4}},
\end{align}
where we use $e^{-z}\lesssim z^{-1/4}$. Setting $w=v-y\vert$, we obtain
\begin{align}
    \int_{y}^{+\infty}(v-y)^{-\frac{1}{2}}e^{-ctv^2}dv&=\int_{0}^{+\infty}w^{-\frac{1}{2}}e^{-ct(w+y)^2}dw\lesssim e^{-cty^2}.
\end{align}
Thus $I_5^{(1)}\lesssim t^{-1/4}$. We turn to estimate $I_5^{(2)}$. For $p>2$, the similar analysis to \eqref{toolI2} implies that
\begin{align}
    I_5^{(2)}&\leqslant\int_{0}^{+\infty}\left\Vert \frac{1}{s-z} \right\Vert _{L^{q}(\mathbb{R}_{+})}\left\Vert |s-\xi_2|^{-\frac{1}{2}} \right\Vert _{L^{p}(\mathbb{R}_{+})}e^{-ctv^2}dv\nonumber\\
    &\lesssim\int_{0}^{+\infty}v^{-\frac{1}{2}+\frac{1}{p}}|v-y|^{1/q-1}e^{-ctv^2}dv\nonumber\\
    &=\left(\int_{0}^{y}+\int_{y}^{+\infty}\right)v^{-\frac{1}{2}+\frac{1}{p}}|v-y|^{1/q-1}e^{-ctv^2}dv.
\end{align}
The analysis for $I_5^{(1)}$ can be applied to bound the two integrals above.
\begin{align}
    \int_{0}^{y}(y-v)^{\frac{1}{q}-1}v^{-\frac{1}{2}+\frac{1}{p}}e^{-ctv^2}dv
    &\lesssim t^{-\frac{1}{4}}\int_{0}^{y}(y-v)^{-\frac{1}{2}}dv\nonumber\\
    &\lesssim t^{-\frac{1}{4}},
\end{align}
and
\begin{align}
    \int_{y}^{+\infty}v^{-\frac{1}{2}+\frac{1}{p}}|v-y|^{1/q-1}e^{-ctv^2}dv\leqslant
    \int_{0}^{+\infty}w^{-\frac{1}{2}}e^{-ct(w+y)^2}dw \lesssim e^{-cty^2}.
\end{align}

Next we consider $I_6$. Since $|s|\in\left[1,\frac{\xi_1-\xi_2}{\sqrt{2}}\right)$, we have $\langle s\rangle\leqslant c(\xi)$ for $\xi\in\mathcal{R}_L$.
The singularity at $z=1$ could be balanced by \eqref{estbalance},
\begin{align}
    I_6&\leqslant c(\xi)\frac{1}{\pi}\iint_{\Omega_{24}}\frac{|p'_{24}|e^{2t\Im \theta}}{|s-z|}dA(s)\nonumber\\
    &\lesssim \int_{0}^{+\infty}\int_{\frac{v}{{\rm tan}\phi}}^{+\infty}\frac{|p'_{24}|}{|s-z|}e^{-ct(1+|s|^{-2})v^2}dudv\nonumber\\
    &\lesssim \int_{\mathbb{R}_{+}}\left\Vert \frac{1}{s-z} \right\Vert _{L^{2}(\mathbb{R}_{+})}\Vert p'_{24} \Vert_{L^{2}(\mathbb{R}_{+})}e^{-ctv^2}dv\nonumber\\
    &\lesssim \left(\int_{0}^{y}+\int_{y}^{+\infty}\right)|v-y|^{-\frac{1}{2}}e^{-ctv^2}dv.\nonumber\\
    ({\rm similar \hspace{0.2em} to}\hspace{0.2em} I_5^{(1)})&\lesssim t^{-1/4}
\end{align}
From estimatations for $I_5$ and $I_6$, we conclude that $\Vert S \Vert_{L^{\infty}(\mathbb{C})\rightarrow L^{\infty}(\mathbb{C})}\lesssim t^{-\frac{1}{4}}$
for $z\in\Omega_{jk}$, $j=2,3$, $k=1,2,3,4$.
Based on the three cases we discuss, $\Vert S \Vert_{L^{\infty}(\mathbb{C})\rightarrow L^{\infty}(\mathbb{C})}\lesssim t^{-\frac{1}{4}}$ as $t\rightarrow\infty$.
\end{proof}

Following from Proposition \ref{estS},  $(1-S)^{-1}$ exists as $t\rightarrow\infty$.
Finally, we turn to evaluate $M^{(3)}$ as $t\rightarrow\infty$. Make the asymptotic expansion as follows.
\begin{equation}
    M^{(3)}=I+z^{-1}M^{(3)}_{1}(x,t)+\mathcal{O}(z^{-2}), \quad {\rm as} \quad z\rightarrow \infty,
\end{equation}
where
\begin{equation}\label{m31}
    M^{(3)}_{1}(x,t)=\frac{1}{\pi}\iint_{\mathbb{C}}M^{(3)}(s)W^{(3)}(s)dA(s).
\end{equation}
To recover the solution of defocusing mKdV \eqref{dmkdv}, we shall discuss the asymptotic behavior of $M^{(3)}_{1}(x,t)$.
\begin{proposition} \label{estm31}As $t\rightarrow \infty$ for $\xi\in\mathcal{R}_L$,
    \begin{equation}
    \vert M^{(3)}_{1}(x,t) \vert \lesssim t^{-\frac{3}{4}}.
    \end{equation}
\end{proposition}
\begin{proof} Noticing the boundedness of $M^{(PR)}$ and $M^{(3)}$, we have
    \begin{align}
    \vert M^{(3)}_{1}\vert &\leqslant \iint_{\Omega_{jk}}\left\vert M^{(3)}M^{(PR)}\bar{\partial}R^{(2)}\left(M^{(PR)}\right)^{-1} \right\vert dA(s)\nonumber\\
    &\lesssim \iint_{\Omega_{jk}} \frac{\langle s\rangle}{|s-1|}\left\vert \bar{\partial}f_{jk}e^{\pm 2it\theta} \right\vert dA(s)\nonumber\\
    &=\iint_{\Omega_{jk}} \frac{\langle s\rangle}{|s-1|}\left\vert \bar{\partial}f_{jk}\right\vert e^{\mp 2t\Im\theta} dA(s)
    \end{align}
Similar to the proof of Proposition \ref{estS}, we only take into account that matrix-valued functions have
support in the sector $\Omega_{jk}$. What we need to do is to evaluate the
integral $\iint_{\Omega_{jk}} \frac{\langle s\rangle}{|s-1|}\left\vert \bar{\partial}f_{jk}\right\vert e^{\mp 2t\Im\theta} dA(s)$ for $z\in \Omega_{jk}$, $j=0^{\pm},1,2,3,4$,
$k=1,2,3,4$. We exhibit details for $z\in\Omega_{0^{+}1}$, $z\in\Omega_{11}$ and $z\in\Omega_{24}$.
We point out that the analysis for $z\in\Omega_{24}$ is a bit different from $z\in\Omega_{0^{+}1}, \Omega_{11}$, because we should deal with
the singularity at $z=1$ as what we do in the proof of Proposition \ref{estS}.

For $z\in \Omega_{0^{+}1}$, we make $z=x+iy$, $s=u+iv$ which satisfy $0<x,u<\xi_2/2$, $0<y,v<\frac{\xi_2{\rm tan}\phi}{2}$($<\frac{\xi_2}{2}$).
Owe to \eqref{<s>est}, $\langle s\rangle/|s-1|=\mathcal{O}(1)$ for $z\in\Omega_{0^{+}1}$. And we can divide the integral into two parts
\begin{align}
    \iint_{\Omega_{0^{+}1}} \left\vert \bar{\partial}f_{0^{+}1}(s)\right\vert e^{-2t\Im\theta} dA(s)\lesssim I_1+I_2,
\end{align}
where
\begin{align}
    &I_1=\iint_{\Omega_{0^{+}1}} |p'_{0^{+}1}(s)|e^{-2t\Im\theta}dA(s), \\
    &I_2=\iint_{\Omega_{0^{+}1}} |s|^{-\frac{1}{2}}e^{-2t\Im\theta}dA(s).
\end{align}
Since $\Vert p'_{0^{+}1}\Vert_{L^{2}}$ is bounded, we bound $I_1$ by Cauchy-Schwarz inequality
\begin{align}
    |I_1|&\leqslant\iint_{\Omega_{0^{+}1}} \vert p'_{0^{+}1}\vert e^{-2ctv}dA(s)\nonumber\\
    & \leqslant\int_{0}^{\frac{\xi_2{\rm tan} \phi}{2}}\Vert p'_{0^{+}1} \Vert_{L^{2}(v,\xi_2/2)}
    \left(\int_{v}^{\xi_2/2}e^{-4ctv}du\right)^{1/2}dv\nonumber\\
    &\lesssim\int_{0}^{+\infty}v^{\frac{1}{2}}e^{-2ctv}dv\nonumber\\
    &\overset{w=v^{\frac{1}{2}}}{=}2\int_{0}^{+\infty}w^{2}e^{-2ctw^2}dw\lesssim t^{-\frac{3}{2}}.
\end{align}
For $I_2$, we use H\"older inequality and \eqref{est-1/2} to obtain
\begin{align}
    |I_2|& \leqslant\int_{0}^{\frac{\xi_2{\rm tan} \phi}{2}}\Vert |s|^{-\frac{1}{2}} \Vert_{L^{p}(v,\xi_2/2)}
    \left(\int_{v}^{\xi_2/2}e^{-2ctqv}du\right)^{1/q}dv\nonumber\\
    &\lesssim\int_{0}^{+\infty}v^{-\frac{1}{2}+\frac{1}{p}}v^{\frac{1}{q}}e^{-2ctv}dv\nonumber\\
    &\overset{\frac{1}{p}+\frac{1}{q}=1}{=}\int_{0}^{+\infty}v^{\frac{1}{2}}e^{-2ctv}dv\nonumber\\
    &\lesssim t^{-\frac{3}{2}}.
\end{align}
For $z\in \Omega_{0^{\pm}k}$, $k=1,2$, we conclude that $|M^{(3)}_{1}|\lesssim t^{-\frac{3}{2}}$.

For $z\in \Omega_{11}$, $\langle s\rangle/|s-1|=\mathcal{O}(1)$. And we make $z=x+iy$, $s=\xi_1+u+iv$.
\begin{align}
    \iint_{\Omega_{11}} \left\vert \bar{\partial}f_{11}(s)\right\vert e^{-2t\Im\theta} dA(s)\lesssim I_3+I_4,
\end{align}
where
\begin{align}
    &I_3=\iint_{\Omega_{11}} |p'_{11}(s)|e^{-2t\Im\theta}dA(s), \\
    &I_4=\iint_{\Omega_{11}} |s-\xi_1|^{-\frac{1}{2}}e^{-2t\Im\theta}dA(s).
\end{align}
With the help of Proposition \ref{lemz=xi1}, we can bound $I_3, I_4$. Using Cauchy-Schwarz inequality,
\begin{align}
    |I_3|&\leqslant\iint_{\Omega_{11}} \vert p'_{11}(s) \vert e^{-2tv(\Re s-\xi_1)}dA(s)\nonumber\\
    &=\iint_{\Omega_{11}} \vert p'_{11}(s)\vert e^{-2tvu}dA(s)\nonumber\\
    &\overset{r \in H^{1}}{\leqslant} \int_{0}^{+\infty}\Vert p'_{11}(s) \Vert_{L^{2}(v+\xi_1,\infty)}
    \left(\int_{v}^{+\infty}e^{-4tuv}du\right)^{\frac{1}{2}}dv\nonumber\\
    &\lesssim t^{-\frac{1}{2}}\int_{0}^{+\infty}v^{-\frac{1}{2}}e^{-2tv^2}\lesssim t^{-\frac{3}{4}}.
\end{align}
As for $I_4$, we take the advantage of H\"older inequality and \eqref{est-1/2} again
\begin{align}
    |I_4|&\leqslant\int_{0}^{+\infty}v^{-\frac{1}{2}+\frac{1}{p}}\left(\int_{v}^{+\infty}e^{-2tquv}du\right)^{\frac{1}{q}}dv\nonumber\\
     &\lesssim\int_{0}^{+\infty}v^{-\frac{1}{2}+\frac{1}{p}}(qtv)^{-\frac{1}{q}}e^{-2tv^2}dv\nonumber\\
     &\overset{\frac{1}{p}+\frac{1}{q}=1}{\lesssim} t^{-\frac{1}{q}}\int_{0}^{+\infty}v^{2/p-3/2}e^{-2tv^2}dv\nonumber\\
     &=t^{-\frac{3}{4}}\int_{0}^{+\infty}w^{\frac{2}{p}-\frac{3}{2}}e^{-2w^2}dw\nonumber\\
     &\lesssim t^{-\frac{3}{4}},
\end{align}
where we use the substitution $w = t^{1/2}v$.
For $z\in \Omega_{jk}$, $j=1,4$, $k=1,2,3,4$, we can conclude that $|M^{(3)}_{1}|\lesssim t^{-\frac{3}{4}}$.

For $z\in \Omega_{24}$, we make $z=x+iy$, $s=\xi_2+u+iv$ which satisfy $0<u<\frac{\xi_1-\xi_2}{2}$,
$0<v<\frac{\xi_1-\xi_2}{2}{\rm tan}\phi$($<\frac{\xi_1-\xi_2}{2}$).
\begin{align}
    \iint_{\Omega_{24}} \left\vert \bar{\partial}f_{24}(s)\right\vert e^{2t\Im\theta} dA(s)\lesssim I_5+I_6,
\end{align}
where
\begin{align}
    &I_5=\frac{1}{\pi}\iint_{\Omega_{24}}\frac{\langle s\rangle|\bar{\partial}f_{24}|e^{2t\Im \theta}\chi_{[\xi_2,1)}(|s|)}{|s-1|}dA(s), \\
    &I_6=\frac{1}{\pi}\iint_{\Omega_{24}}\frac{\langle s\rangle|\bar{\partial}f_{24}|e^{2t\Im \theta}\chi_{\left[1,\frac{\xi_1-\xi_2}{\sqrt{2}}\right)}(|s|)}{|s-1|}dA(s),
\end{align}
and $\chi_{[\xi_2,1)}(|s|)+\chi_{\left[1,\frac{\xi_1-\xi_2}{\sqrt{2}}\right)}(|s|)$ is the partition of unity.

Notice $\langle s\rangle/|s-1|=\mathcal{O}(1)$ for $|s|\in [\xi_2,1)$. Combining Proposition \ref{estopenlenssaddle}, we divide
$I_5$ into two parts
\begin{equation}
    I_5\lesssim I_5^{(1)}+I_5^{(2)},
\end{equation}
where
\begin{align}
    &I_5^{(1)}=\iint_{\Omega_{24}} |p'_{24}(s)|e^{2t\Im\theta}dA(s), \\
    &I_5^{(2)}=\iint_{\Omega_{24}} |s-\xi_2|^{-\frac{1}{2}}e^{2t\Im\theta}dA(s).
\end{align}
With the help of Proposition \ref{lemz=xi2}, we can bound $I_5^{(1)}, I_5^{(2)}$.
We bound $I_5^{(1)}$ by Cauchy-Schwarz inequality
\begin{align}
    |I_5^{(1)}|&\leqslant\iint_{\Omega_{24}} \vert p'_{24}(s)\vert e^{-2c(1+|z|^{-2})tv^2}dA(s)\nonumber\\
    &\leqslant \iint_{\Omega_{24}} \vert p'_{24}(s)\vert e^{-2ctv^2}dA(s) \nonumber\\
    &\leqslant\int_{0}^{\frac{\xi_1-\xi_2}{2}{\rm tan}\phi}\Vert p'_{24}(s) \Vert_{L^{2}(v+\xi_2,\frac{\xi_1-\xi_2}{2})}
    \left(\int_{v}^{\frac{\xi_1-\xi_2}{2}}e^{-4ctv^2}du\right)^{1/2}dv\nonumber\\
    &\lesssim\int_{0}^{+\infty}v^{\frac{1}{2}}e^{-2ctv^2}dv\nonumber\\
    &\overset{w=t^{\frac{1}{4}}v^{\frac{1}{2}}}{=}t^{-\frac{3}{4}}\int_{0}^{+\infty}w^2e^{-2cw^4}dw\nonumber\\
    &\lesssim t^{-\frac{3}{4}}.
\end{align}
For $I_5^{(2)}$, the H\"older inequality and \eqref{est-1/2} are used to obtain
\begin{align}
    |I_5^{(2)}|& \leqslant\int_{0}^{\frac{\xi_1-\xi_2}{2}{\rm tan}\phi}\Vert |s-\xi_2|^{-\frac{1}{2}} \Vert_{L^{p}(v+\xi_2,\frac{\xi_1-\xi_2}{2})}
    \left(\int_{v}^{\frac{\xi_1-\xi_2}{2}}e^{-2ctqv^2}du\right)^{1/q}dv\nonumber\\
    &\lesssim\int_{0}^{+\infty}v^{-\frac{1}{2}+\frac{1}{p}}v^{\frac{1}{q}}e^{-2ctv^2}dv\nonumber\\
    &\overset{\frac{1}{p}+\frac{1}{q}=1}{=}\int_{0}^{+\infty}v^{\frac{1}{2}}e^{-2ctv^2}dv\nonumber\\
    &\lesssim t^{-\frac{3}{4}}.
\end{align}
We finally deal with $I_6$. Thanks to \eqref{estbalance}, the singularity at $z=1$ can be balanced. Additionally,
for $|s|\in\left[1,\frac{\xi_1-\xi_2}{\sqrt2}\right)$, $\langle s\rangle\leqslant c(\xi)$. As a consequence,
\begin{align}
    |I_6|&\leqslant c(\xi)\iint_{\Omega_{24}}|p'_{24}| e^{2t\Im \theta}dA(s)\nonumber\\
    &\lesssim\int_{0}^{+\infty}\int_{\frac{v}{{\rm tan}\phi}}^{+\infty}|p'_{24}|e^{-ct(1+|s|^{-2})v^2}dudv\nonumber\\
    &\leqslant\int_{0}^{+\infty}\Vert p'_{24}\Vert_{L^{2}(\mathbb{R}_{+})}\left(\int_{v}^{+\infty}e^{-2ctv^2}du\right)^{\frac{1}{2}}dv\nonumber\\
    &\lesssim\int_0^{+\infty}v^{\frac{1}{2}}e^{-ctv^2}dv\lesssim t^{-3/4}.
\end{align}
Summarizing the estimations for $I_5$ and $I_6$, we conclude that $|M^{(3)}_{1}|\lesssim t^{-\frac{3}{4}}$
for $z\in \Omega_{jk}$, $j=2,3$, $k=1,2,3,4$.

In a conclusion,
$|M^{(3)}_{1}|\lesssim t^{-\frac{3}{2}}+t^{-\frac{3}{4}}+t^{-\frac{3}{4}}\lesssim t^{-\frac{3}{4}}$ as $t\rightarrow\infty$.
\end{proof}

\subsection{Proof of Theorem \ref{mainthm}(a)}\label{subsec:pfofa}
Reviewing the transformation \eqref{trans1}, \eqref{trans2}, \eqref{equ:Merr} and \eqref{trans4}:
\begin{equation}
    M^{(1)}=MG\delta^{\sigma_3}, \quad M^{(2)}=M^{(1)}R^{(2)}, \quad M^{(3)}=M^{(2)}(M^{(PR)})^{-1}, \quad  M^{(PR)}=M^{(err)}M^{(\infty)},
\end{equation}
we have
\begin{equation}
    M(z)=M^{(3)}(z)M^{(err)}(z)M^{(\infty)}(z){R^{(2)}(z)}^{-1}{\delta(z)}^{-\sigma_3}{G(z)}^{-\sigma_3}, \quad z\in \mathbb{C}\backslash U(\xi).
\end{equation}
Take $z\rightarrow \infty$ out $\overline{\Omega}$ ($R^{(2)}=I$, $G(z)=I$), we obtain
\begin{equation}
    M(z)=\left(I+z^{-1}M^{(3)}_1+\cdots\right)\left(I+z^{-1}M^{(err)}_1+\cdots\right)\left(I+z^{-1}M^{(\infty)}_1+\cdots\right)\left(I-z^{-1}\delta_1\sigma_3+\cdots\right),
\end{equation}
thus
\begin{equation}
    M_1=M^{(\infty)}_1+M^{(err)}_1+M^{(3)}_1-\delta_1\sigma_3.
\end{equation}
Using the formulae \eqref{recover}, we have
\begin{equation}
    q(x,t)=-i\left(M^{(\infty)}_1\right)_{12}-i\left(M^{(err)}_1\right)_{12}+\mathcal{O}\left(t^{-\frac{3}{4}}\right).
\end{equation}
Combining Proposition \ref{prop:global} and Proposition \ref{E1est}, Theorem \ref{mainthm}(a) follows.

\section{Asymptotics for $\xi\in\mathcal{R}_R$: right field}\label{sec:pfof1c}
\subsection{First transformation: $M\rightarrow M^{(1)}$}
By the Figure \ref{figtheta}(c), the jump factorization
\begin{equation*}
    V(z)=\begin{pmatrix} 1 & 0 \\ \frac{r(z)e^{-2it\theta}}{1-|r(z)|^2} & 1 \end{pmatrix}\left(1-|r(z)|^2\right)^{\sigma_3}\begin{pmatrix} 1 & -\frac{\overline{r(z)}e^{2it\theta}}{1-|r(z)|^2}\\ 0 & 1\end{pmatrix}, \quad z\in\Sigma,
\end{equation*}
plays a key role in our analysis. Therefore, we choose
\begin{equation}\label{equ:newdelta}
    \delta(z):=\delta(z,\xi)={\rm exp}\left[-\frac{1}{2\pi i}\int_{\Sigma}{\rm log}\left(1-|r(s)|^2\right)\frac{1}{s-z}ds\right].
\end{equation}
\begin{remark}\rm
    The difference between the $\delta(z)$ defined by \eqref{delta} and the $\delta(z)$ defined by \eqref{equ:newdelta} is the
    integral interval. The interval of the former is $\Gamma:=(-\infty, \xi_4)\cup(\xi_3,0)\cup(0,\xi_2)\cup(\xi_1, +\infty)$,
    however, the latter is $\Sigma$.
\end{remark}
Define
\begin{equation}
    M^{(1)}(z)=M(z)G(z)\delta(z)^{\sigma_3},
\end{equation}
with
\begin{align}
    G(z)=\left\{
    \begin{aligned}
    &\begin{pmatrix} 1 & -\frac{z-\eta_n}{c_ne^{-2it\theta(\eta_n)}} \\  0 & 1\end{pmatrix}, & z\in D(\eta_n,h),\\
    &\begin{pmatrix} 1 &  0 \\ -\frac{z-\bar{\eta}_n}{\bar{c}_ne^{2it\theta(\bar{\eta}_n)}}  & 1\end{pmatrix}, & z\in D(\bar{\eta}_n,h),\\
    &I  &elsewhere.
    \end{aligned}
    \right.
\end{align}
The jump matrix $V^{(1)}(z)$ is as follows
\begin{align}\label{rhp1jump}
    V^{(1)}(z)=\left\{
        \begin{aligned}
        &\begin{pmatrix} 1 & 0 \\ \frac{r(z)\delta_{-}^{2}(z)}{1-|r(z)|^2}e^{-2it\theta} & 1 \end{pmatrix}\begin{pmatrix} 1 & -\frac{\overline{r(z)}\delta_{+}^{-2}(z)}{1-|r(z)|^2}e^{2it\theta}\\ 0 & 1\end{pmatrix}, & z\in\Sigma,\\
        &\begin{pmatrix} 1 & -\frac{z-\eta_n}{c_n\delta^2e^{-2it\theta(\eta_n)}} \\ 0 & 1 \end{pmatrix}, & z\in \partial D(\eta_n,h) \ {\rm oriented \ counterclockwise},\\
        &\begin{pmatrix} 1 & 0 \\ \frac{z-\bar{\eta}_n}{\bar{c}_n\delta^{-2}(z)e^{2it\theta(\bar{\eta}_n)}} & 1 \end{pmatrix}, & z\in \partial D(\bar{\eta}_n,h) \ {\rm oriented \ clockwise}.\\
        \end{aligned}
        \right.
    \end{align}
The asymptotic behavior of $M^{(1)}$ is the same as the former section.

\subsection{Opening $\bar{\partial}$ lenses: $M^{(1)}\rightarrow M^{(2)}$}
Find a small sufficiently angle $\phi: \phi<\theta_0$ and define a new region $\Omega=\cup_{j=1,2,3,4}\Omega_{j}$, where
\begin{align}
    &\Omega_{1}=\left\{z\in\mathbb{C}: 0<{\rm arg}z<\phi\right\}, &\Omega_{2}=\left\{z\in\mathbb{C}: \pi-\phi<{\rm arg}z<\phi\right\},\\
    &\Omega_{3}=\left\{z\in\mathbb{C}: -\pi<{\rm arg}z<-\pi+\phi\right\}, &\Omega_{4}=\left\{z\in\mathbb{C}: -\phi<{\rm arg}z<0\right\}.
\end{align}
Some paths are denoted by
\begin{align}
    &\Sigma_1=e^{i\phi}\mathbb{R}_{+}, &\Sigma_2=e^{i(\pi-\phi)}\mathbb{R}_{+},\\
    &\Sigma_3=e^{-i(\pi-\phi)}\mathbb{R}_{+}, &\Sigma_4=e^{-i\phi}\mathbb{R}_{+},
\end{align}
with the left-to-right oriented boundaries of $\Omega$, see Figure \ref{figR2}.
\begin{figure}[htbp]
	\centering
		\begin{tikzpicture}[node distance=2cm]
		\draw[yellow!30, fill=yellow!20] (0,0)--(3,-0.5)--(3,0)--(0,0)--(-3,-0.5)--(-3,0)--(0,0);
		\draw(0,0)--(3,0.5)node[above]{\scriptsize $\Sigma_1$};
		\draw(0,0)--(-3,0.5)node[left]{\scriptsize $\Sigma_2$};
		\draw(0,0)--(-3,-0.5)node[left]{\scriptsize $\Sigma_3$};
		\draw(0,0)--(3,-0.5)node[right]{\scriptsize $\Sigma_4$};
		\draw[->](-4,0)--(4,0)node[right]{\small $\Re z$};
		\draw[->](0,-3)--(0,3)node[above]{\small $\Im z$};
		\draw[-latex](-3,-0.5)--(-1.5,-0.25);
		\draw[-latex](-3,0.5)--(-1.5,0.25);
		\draw[-latex](0,0)--(1.5,0.25);
		\draw[-latex](0,0)--(1.5,-0.25);
		\coordinate (C) at (-0.2,2.2);
		\coordinate (D) at (2.2,0.2);
		\fill (D) circle (0pt) node[right] {\scriptsize $\Omega_1$};
		\coordinate (J) at (-2.2,-0.2);
		\fill (J) circle (0pt) node[left] {\scriptsize $\Omega_3$};
		\coordinate (k) at (-2.2,0.2);
		\fill (k) circle (0pt) node[left] {\scriptsize $\Omega_2$};
		\coordinate (k) at (2.2,-0.2);
		\fill (k) circle (0pt) node[right] {\scriptsize $\Omega_4$};
		\coordinate (I) at (0.2,0);
		\fill (I) circle (0pt) node[below] {$0$};
		\draw[thick,dotted,red] (2,0) arc (0:360:2);
		\coordinate (A) at (2,3);
		\coordinate (B) at (2,-3);
		\coordinate (C) at (-0.5546996232,0.8320505887);
		\coordinate (D) at (-0.5546996232,-0.8320505887);
		\coordinate (E) at (0.5546996232,0.8320505887);
		\coordinate (F) at (0.5546996232,-0.8320505887);
		\coordinate (G) at (-2,3);
		\coordinate (H) at (-2,-3);
		\coordinate (I) at (2,0);
		\coordinate (J) at (1.7320508075688774,1);
		\coordinate (K) at (1.7320508075688774,-1);
		\coordinate (L) at (-1.7320508075688774,1);
		\coordinate (M) at (-1.7320508075688774,-1);
        \coordinate (N) at (0,2);
        \coordinate (O) at (0,-2);
		\fill (I) circle (1pt) node[above] {$1$};
		\fill (J) circle (1pt) node[right] {\footnotesize $z_n$};
		\fill (K) circle (1pt) node[right] {\footnotesize $\bar{z}_n$};
		\fill (L) circle (1pt) node[left] {\footnotesize $-\bar{z}_n$};
		\fill (M) circle (1pt) node[left] {\footnotesize $-z_n$};
        \fill (N) circle (1pt) node[right] {$i$};
        \fill (O) circle (1pt) node[left] {$-i$};
        \draw[->] (0.18,2.05) arc(0:360:0.2);
        \draw[->] (1.98,0.98) arc(0:360:0.2);
        \draw[->] (-1.58,0.98) arc(0:360:0.2);
        \draw[->] (0.2,-2.05) arc(360:0:0.2);
        \draw[->] (-1.58,-0.98) arc(360:0:0.2);
        \draw[->] (1.98,-0.98) arc(360:0:0.2);
		\end{tikzpicture}
	\caption{\footnotesize For $\xi\in\mathcal{R}_R$, there are no phase points on the jump contour. The white regions imply that
    $e^{2it\theta}\rightarrow 0$, however, the yellow regions imply that $e^{-2it\theta}\rightarrow 0$.}
	\label{figR2}
\end{figure}
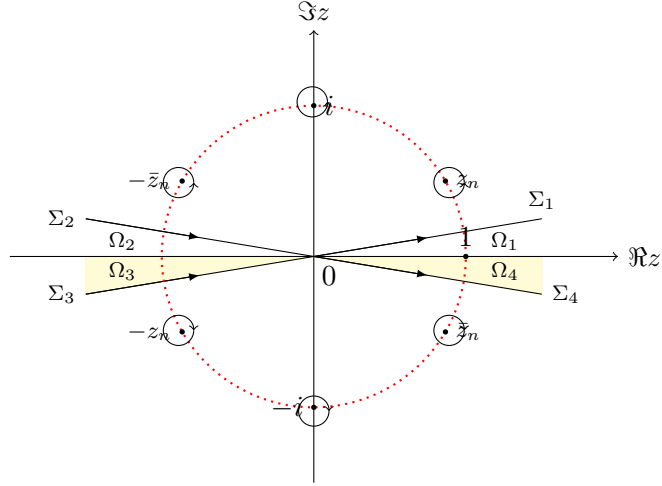
\begin{proposition} For $\xi\in\mathcal{R}_R$, $z=le^{i\phi}$, and $F(l)=l+l^{-1}$, the phase function $\theta(z)$ defined by \eqref{phasefunc} satisfies
    \begin{align}
        &\Im\theta(z)\geqslant \frac{1}{2}F(l)\vert{\rm sin} \phi \vert[\xi+F^2(l)], \quad z\in\Omega_{j}, \quad j=1,2,\\
        &\Im\theta(z)\leqslant -\frac{1}{2}F(l)\vert{\rm sin} \phi \vert[\xi+F^2(l)], \quad z\in\Omega_{j}, \quad j=3,4.
    \end{align}
\end{proposition}
\begin{proof}
    We give the details for $z\in\Omega_{1}$, the proof for the other regions is similar.
    Recalling \eqref{imthetaxingtai1}, we have
    \begin{align}
        \Im\theta(z)&=\frac{1}{2}F(l){\rm sin}\phi\left[\left(2F^2(l)-6\right){\rm cos}\left(2\phi\right)+\xi+F^2(l)\right]\nonumber\\
        &\overset{F(l)\geqslant 2}{\geqslant}\frac{1}{2}F(l){\rm sin}\phi[\xi+F^2(l)]>0.
    \end{align}
\end{proof}
We choose $R^{(2)}(z)$ as
\begin{equation}
    R^{(2)}(z)=\left\{
        \begin{aligned}
        &\begin{pmatrix} 1 & f_{j}e^{2it\theta} \\ 0 & 1 \end{pmatrix}, \quad &z\in\Omega_{j}, \quad j=1,2,\\
        &\begin{pmatrix} 1 & 0 \\ f_{j}e^{-2it\theta} & 1 \end{pmatrix}\quad &z\in\Omega_{j}, \quad j=3,4,\\
        &I, \quad &elsewhere,
        \end{aligned}
        \right.
\end{equation}
where $f_{j}$ is given by the following proposition.
\begin{proposition}[Opening lens at $z=0$ for $\xi\in\mathcal{R}_R$]\label{estopenlesz=0xi>6}
        $f_{j}: \overline{\Omega}_{j}\rightarrow \mathbb{C}$, $j=1,2,3,4$ are continuous on $\overline{\Omega}_{j}$
        with boundary values:
        \begin{align}
            f_{j}(z)=\left\{
                            \begin{aligned}
                            &\frac{\overline{r(z)}}{1-|r(z)|^2}\delta^{-2}_{+}(z), \quad &z\in\mathbb{R},\\
                            &0, \quad &z\in \Sigma_{j}, \quad j=1,2
                            \end{aligned}
                            \right. \label{bdryz=01xi>6}\\
            f_{j}(z)=\left\{
                            \begin{aligned}
                            &\frac{r(z)}{1-|r(z)|^2}\delta^{2}_{-}(z), \quad &z\in\mathbb{R},\\
                            &0, \quad &z\in \Sigma_{j},\quad j=3,4.
                            \end{aligned}
                            \right.\label{bdryz=02xi>6}
            \end{align}
    And $f_{j}, j=1,2,3,4$ have following estimation:
    \begin{equation}\label{dbarfor0no1xi>6}
        \vert \bar{\partial}f_{j}(z)\vert \leqslant c|z|^{-\frac{1}{2}}+c|r'(|z|)|+c\varphi(|z|), \quad j=1,2,3,4,
    \end{equation}
    where $\varphi\in C_{0}^{\infty}(\mathbb{R},[0,1])$ is a cutoff function with small support near 1.

    Moreover
    \begin{align}\label{dbarfor0no2xi>6}
        &\vert \bar{\partial}f_{j}(z)\vert \leqslant c|z-1|, \quad z\in \Omega_{j}, \quad j=1,4,\\
        &\vert \bar{\partial}f_{j}(z)\vert \leqslant c|z+1|, \quad z\in \Omega_{j}, \quad j=2,3.
    \end{align}
\end{proposition}
\begin{proof}
The proof is an analogue of \cite[Proposition 5.5]{Zhang&Xu mkdv} or \cite[Lemma 6.5]{Cuc}.
A sketch proof for $f_1$ is exhibited as follows. By $r(z)\rightarrow\mp i$ as $z\rightarrow\pm 1$, we know that $|r(z)|\rightarrow 1$ as $z\rightarrow\pm 1$. This implies
that $f_1(z)$ is singular at $z=1$. However, the singular behavior is exactly balanced by the factor $\delta^2(z)$. With the
help of \eqref{rz}-\eqref{relazbz}, we have
\begin{equation}\label{singularopen}
    \frac{\overline{r(z)}}{1-|r(z)|^2}\delta^{-2}_{+}(z)=\frac{\overline{b(z)}}{a(z)}\left(\frac{a(z)}{\delta_{+}(z)}\right)^2
    =\frac{\overline{J_b(z)}}{J_a(z)}\left(\frac{a(z)}{\delta_{+}(z)}\right)^2,
\end{equation}
where $J_a(z)=$det$(\varPhi_{+,1}, \varPhi_{-,2})$, $J_b(z)=$det$(\varPhi_{-,1}, \varPhi_{+,1})$. It's not difficult to know that
the denominator of each factor in the r.h.s of \eqref{singularopen} is nonzero and analytic in $\Omega_1$, with a well defined
nonzero limit on $\partial\Omega_1$. Notice also that in $\Omega_1$ away from the point $z=1$ the factors in the l.h.s of \eqref{singularopen}
are well behaved.

We introduce the cutoff functions $\chi_0, \chi_1 \in C_{0}^{\infty}(\mathbb{R},[0,1])$ with small support near $z=0$ and $z=1$
respectively, such that for any sufficiently small $s$, $\chi_0(s)=\chi_{1}(s+1)=1$. Additionally, we impose the condition
$\chi_1(s)=\chi_1(s^{-1})$ to preserve symmetry. Then we can rewrite the function $f_{1}(z)$ in $\mathbb{R}_{+}$ as
$f_1(z)=f_1^{(1)}(z)+f_1^{(2)}(z)$, where
\begin{equation}\label{dividef1}
    f_1^{(1)}(z)=\left(1-\chi_1(z)\right)\frac{\overline{r(z)}}{1-|r(z)|^2}\delta^{-2}_{+}(z),\quad
    f_1^{(2)}(z)=\chi_1(z)\frac{\overline{J_b(z)}}{J_a(z)}\left(\frac{a(z)}{\delta_{+}(z)}\right)^2.
\end{equation}
The aim of \eqref{dividef1} is to balance the effect raised by the singularity $z=1$ due to $|r(1)|=1$. Fix a small
$\kappa_0>0$, we extend $f_1^{(1)}(z)$ and $f_1^{(2)}(z)$ in $\Omega_{1}$ by
\begin{align}
    &f_1^{(1)}(z)=(1-\chi_1(|z|))\frac{\overline{r(|z|)}}{1-|r(|z|)|^2}\delta^{-2}(z){\rm cos}\left(\kappa {\rm arg}z\right).\\
    &f_1^{(2)}(z)=h(|z|)g(z){\rm cos}\left(\kappa {\rm arg}z\right)+\frac{i|z|}{\kappa}\chi_0\left(\frac{{\rm arg}z}{\kappa_0}\right)
    h'(|z|)g(z){\rm sin}\left(\kappa {\rm arg}z\right),
\end{align}
where
\begin{equation}
    \kappa:=\frac{\pi}{2\theta_0}, \quad h(z):=\chi_1(z)\frac{\overline{J_b(z)}}{J_a(z)},\quad
    g(z):=\left(\frac{a(z)}{\delta(z)}\right)^2.
\end{equation}
Notice that the definition of $f_1$ preserves the symmetry $f_1(s)=-\overline{f_1({\bar{s}}^{-1})}$.

Firstly we bound the $\bar{\partial}f_{1}^{(1)}$.
\begin{equation}\label{f1(1)}
    \bar{\partial}f_1^{(1)}(z)=-\frac{\bar{\partial}\chi_1(|z|)}{\delta^2(z)}\frac{\overline{r(|z|)}{\rm cos}(\kappa{\rm arg}z)}{1-|r(|z|)|^2}
    +\frac{1-\chi_1(|z|)}{\delta^2(z)}\bar{\partial}\left(\frac{\overline{r(|z|)}{\rm cos}(\kappa{\rm arg}z)}{1-|r(|z|)|^2}\right).
\end{equation}
We know that $1-|r(z)|^2>c>0$ as $z\in{\rm supp}(1-\chi_1(|z|))$ and $\delta^{-2}(z)$ is bounded as $z\in\Omega_1\cap{\rm supp}(1-\chi_1(|z|))$.
Taking $z=le^{i\gamma}:=u+iv$, we still have the equality $\bar{\partial}=\frac{e^{i\gamma}}{2}(\partial_{l}+il^{-1}\partial_{\gamma})$ and
apply it to the first term of \eqref{f1(1)}
\begin{equation}
    \left\vert -\frac{\bar{\partial}\chi_1(|z|)}{\delta^2(z)}\frac{\overline{r(|z|)}{\rm cos}(\kappa{\rm arg}z)}{1-|r(|z|)|^2}\right\vert
    =\left\vert \frac{e^{i\gamma}\chi'_1\bar{r}{\rm cos}(\kappa\gamma)}{2\delta^2(z)(1-|r(|z|)|^2)} \right\vert
    \lesssim \varphi(|z|)
\end{equation}
for a appropriate $\varphi\in C_{0}^{\infty}(\mathbb{R},[0,1])$ with a small support near $1$ and with $\varphi=1$ on supp$\chi_1$.
As $r(0)=0$ and $r(z)\in H^{1}(\mathbb{R})$ it follows that $|r(|z|)|\leqslant |z|^{1/2}\Vert r' \Vert_{L^{2}(\mathbb{R})}$, we have
\begin{equation}
\left\vert \frac{1-\chi_1(|z|)}{\delta^2(z)}\bar{\partial}\left(\frac{\overline{r(|z|)}{\rm cos}(\kappa{\rm arg}z)}{1-|r(|z|)|^2}\right) \right\vert
\lesssim |r'(z)|+\frac{|r(z)|}{|z|}\lesssim |r'(z)|+|z|^{-\frac{1}{2}}.
\end{equation}
So we obtain that
\begin{equation}
    \vert f_{1}^{(1)}(z)\vert\lesssim \varphi(|z|)+|r'(z)|+|z|^{-\frac{1}{2}}.
\end{equation}
Next we bound $\bar{\partial}f_1^{(2)}$,
\begin{align}
    \bar{\partial}f_1^{(2)}&=\frac{1}{2}e^{i\gamma}g(z)\left[h'{\rm cos}(\kappa\gamma)\left(1-\chi_0\left(\frac{\gamma}{\kappa_0}\right)\right)
    -\frac{i\kappa h(l)}{l}{\rm sin}(\kappa\gamma)\nonumber\right.\\
    &\left. +\frac{i}{\kappa}(lh'(l))'{\rm sin}(\kappa\gamma)\chi_{0}\left(\frac{\gamma}{\kappa_0}\right)+
    \frac{i}{\kappa\kappa_0}h'(l){\rm sin}(\kappa\gamma)\chi'_{0}\left(\frac{\gamma}{\kappa_0}\right)\right]
\end{align}
in which $g(z)$ is bounded, $q\in L^{1,2}(\mathbb{R})$ and $q'\in W^{1,1}(\mathbb{R})$. So we claim that $\bar{\partial}f_1^{(2)}(z)\lesssim \varphi(|z|)$
for a $\varphi\in C_{0}^{\infty}(\mathbb{R},[0,1])$ with a small support near $1$, thus yielding \eqref{dbarfor0no1xi>6}.

Finally, as $z\rightarrow 1$, we have
\begin{equation}
    \vert\bar{\partial}f_{1}^{(2)}(z)\vert=\mathcal{O}(\gamma)
\end{equation}
from which \eqref{dbarfor0no2xi>6} follows immediately.
\end{proof}
We now use $R^{(2)}$ to define transformation $M^{(2)}=M^{(1)}R^{(2)}$, which help us set up the following mixed
$\bar{\partial}$-RH problem for $\xi\in\mathcal{R}_R$
\begin{RHP}\label{m2rhpxi>6}
    Find a $2\times 2$ matrix-valued function $M^{(2)}(z)$ such that\\
    - $M^{(2)}(z)$ is continuous in $\mathbb{C}\backslash \Sigma^{(2)}$, where $\Sigma^{(2)}=\Sigma^{pole}\cup(\cup_{j=1,2,3,4}\Sigma_j)$ (see Figure \ref{figR2}).\\
    - $M^{(2)}(z)$ takes continuous boundary values $M^{(2)}_{\pm}(z)$ on $\Sigma^{(2)}$ with jump relation
    \begin{equation}
        M^{(2)}_{+}(z)=M^{(2)}_{-}(z)V^{(2)}(z),
    \end{equation}
    where $V^{(2)}=I$ for $z\in\Sigma_j, j=1,2,3,4$.\\
    - Asymptotic behavior
    \begin{align}
        &M^{(2)}(x,t;z)=I+\mathcal{O}(z^{-1}), \quad  z\rightarrow\infty,\\
        &M^{(2)}(x,t;z)=\frac{\sigma_2}{z}+\mathcal{O}(1), \quad  z\rightarrow 0.
    \end{align}
    - For $z\in\mathbb{C}$, we have $\bar{\partial}$-derivative equality
    \begin{equation}
        \bar{\partial}M^{(2)}=M^{(2)}\bar{\partial}R^{(2)},
    \end{equation}
    where
    \begin{equation}
        \bar{\partial}R^{(2)}=\left\{
            \begin{aligned}
            &\begin{pmatrix} 1 & \bar{\partial}f_{j}e^{2it\theta} \\ 0 & 1 \end{pmatrix}, \quad &z\in\Omega_{j}, \quad j=1,2\\
            &\begin{pmatrix} 1 & 0 \\ \bar{\partial}f_{j}e^{-2it\theta} & 1 \end{pmatrix}\quad &z\in\Omega_{j}, \quad j=3,4,\\
            &0, \quad &elsewhere.
            \end{aligned}
            \right.
    \end{equation}
\end{RHP}

\subsection{Analysis on pure $\bar{\partial}$ problem}
The error order mainly comes from the $\bar{\partial}$-problem for $\xi\in\mathcal{R}_R$, which is different from the previous Section \ref{sec:proofof1a}. We focus our insights on the
estimates for the Cauchy-Green operator $S$ defined by \eqref{C-Gop} and $M^{(3)}_1(x,t)$ defined by \eqref{m31}.
Then we have the following two propositions.
\begin{proposition}\label{estSxi>6}
    Consider the operator $S$ defined in \eqref{C-Gop}, then we have
    $S:L^{\infty}(\mathbb{C})\rightarrow L^{\infty}(\mathbb{C})\cap C^{0}(\mathbb{C})$ and
    \begin{equation}
        \Vert S \Vert_{L^{\infty}(\mathbb{C})\rightarrow L^{\infty}(\mathbb{C})}\lesssim t^{-\frac{1}{2}}.
    \end{equation}
\end{proposition}
\begin{proof}
    The proof is an analogue of Proposition \ref{estS}.
\end{proof}
\begin{proposition}\label{estm31xi>6}As $t\rightarrow\infty$ for $\xi\in\mathcal{R}_R$,
    \begin{equation}
    \vert M^{(3)}_{1}(x,t) \vert \lesssim t^{-1}, \quad {\rm as} \quad t\rightarrow\infty.
    \end{equation}
\end{proposition}
\begin{proof}
We present the details for $z\in\Omega_1$. By the standard procedure, as Proposition \ref{estm31}, we have
\begin{equation}
    \vert M^{(3)}_1(x,t)  \vert\lesssim I_1+I_2+I_3,
\end{equation}
where
\begin{align}
    &I_1=\iint_{\Omega_1}\frac{\langle s\rangle \vert\bar{\partial}f_1\vert e^{-2t\Im\theta}\chi_{[0,1)}(|s|)}{s-1}dA(s),\\
    &I_2=\iint_{\Omega_1}\frac{\langle s\rangle \vert\bar{\partial}f_1\vert e^{-2t\Im\theta}\chi_{[1,2)}(|s|)}{s-1}dA(s),\\
    &I_3=\iint_{\Omega_1}\frac{\langle s\rangle \vert\bar{\partial}f_1\vert e^{-2t\Im\theta}\chi_{[2,+\infty)}(|s|)}{s-1}dA(s).
\end{align}
For the term with $\chi_{[2,+\infty)}(|s|)$ the factor $\langle s\rangle|s-1|^{-1}=\mathcal{O}(1)$, and fixing a $p>2$, $q\in(1,2)$
we obtain the superabound for $I_3$
\begin{align}
    I_3&\lesssim\iint\left[|r'(|s|)|+\varphi(|s|)+|s|^{-\frac{1}{2}}\right]e^{-2t\Im\theta}\chi_{[1,+\infty)}(|s|)dA(s)\nonumber\\
    &\lesssim\int_{\mathbb{R}_{+}}\Vert e^{-ctuv} \Vert_{L^{2}({\rm max}\{v, 1/\sqrt{2}\}, \infty)}+\Vert e^{-ctuv} \Vert_{L^{2}({\rm max}\{v, 1/\sqrt{2}\}, \infty)}
    \Vert |s|^{-1/2} \Vert_{L^q(v, \infty)}dv\nonumber\\
    &\lesssim\int_{\mathbb{R}_{+}} e^{-ctv}\left((tv)^{-1/2}+t^{-1/p}v^{-1/p+1/q-1/2}\right)dv\lesssim t^{-1}.
\end{align}
For $s\in [0,2)$, $\langle s\rangle\leqslant 5$, so it could be omitted from the remaining estimates. For the $\chi_{[1,2]}(|s|)$,
we use \eqref{dbarfor0no1xi>6} to obtain that $I_2\lesssim t^{-1}$ at once. For the $\chi_{[0,1)}(|s|)$, the changes of
variables $w={\bar{z}}^{-1}$ and $r(s)=-\overline{r({\bar{s}}^{-1})}$ imply that
\begin{equation}
    I_1=\iint_{\Omega_1}\vert\bar{\partial}f_1\vert e^{-2t\Im\theta(w)}|w-1|^{-1}\chi_{[1,\infty]}(|w|)|w|^{-1}dA(s)\lesssim t^{-1}.
\end{equation}
Finally, we get the desired estimate.
\end{proof}

\subsection{Proof of Theorem \ref{mainthm}(c)}
Similar to the Subsection \ref{subsec:pfofa},
\begin{equation}
    M(z)=M^{(3)}(z)M^{(\infty)}(z){R^{(2)}(z)}^{-1}{\delta(z)}^{-\sigma_3},
\end{equation}
for $z$ outside $\overline{\Omega}$.

Furthermore, we obtain
\begin{equation}
    M_1=M^{(\infty)}_1+M^{(3)}_1-\delta_1\sigma_3,
\end{equation}
which yields Theorem \ref{mainthm}(c) by using the formulae \eqref{recover}.

\section*{Acknowledgements}
Taiyang Xu and Engui Fan thankfully acknowledge the support from the National Science Foundation of China (Grant No.11671095, No.51879045).

\appendix \label{phasereductionapp}
\renewcommand\thefigure{\Alph{section}\arabic{figure}}
\setcounter{figure}{0}
\renewcommand{\theequation}{\thesection.\arabic{equation}}
\setcounter{equation}{0}
\section{Parabolic cylinder model near $\xi_j$, $j=1,2,3,4$}\label{pcmodel}
This appendix is based on the methods developed by  A. Its'  fundamental work \cite{ItsDAN}.
\subsection{Local model near $\xi_j$, $j=1,3$}\label{pc13}
We take $\xi_1$ as an example to present this standard model.
\begin{RHP}\label{RHPpc13}
Find a matrix-valued function $M^{(PC,\xi_1)}(\zeta):=M^{(PC,\xi_1)}(\zeta;\xi)$ such that\\
- $M^{(PC,\xi_1)}(\zeta;\xi)$ is analytical  in $\mathbb{C}\backslash \Sigma^{pc} $ with $\Sigma^{pc}=\left\{\mathbb{R}e^{i\phi}\right\}\cup
\left\{\mathbb{R}e^{i(\pi-\phi) } \right\}$ shown in Figure \ref{sigpc13}.\\
- $M^{(PC,\xi_1)}$ has continuous boundary values $M^{(PC,\xi_1)}_{\pm}$ on $\Sigma^{pc}$ and
\begin{equation}
M^{(PC,\xi_1)}_+(\zeta)=M^{(PC,\xi_1)}_{-}(\zeta)V^{pc}(\zeta),\quad \zeta \in \Sigma^{pc},
\end{equation}
where
\begin{align}
V^{pc}(\zeta)=\left\{\begin{array}{ll}
\zeta^{-i\nu \hat{\sigma}_3}e^{\frac{i\zeta^2}{4}\hat{\sigma}_3}\left(\begin{array}{cc}
1 & -\frac{\bar{r}_{\xi_1}}{1-|r_{\xi_1}|^2}\\
0 & 1
\end{array}\right),  & \zeta\in\mathbb{R}_{+}e^{\phi i},\\[10pt]
\zeta^{-i\nu \hat{\sigma}_3}e^{\frac{i\zeta^2}{4}\hat{\sigma}_3}\left(\begin{array}{cc}
1 & 0\\
\frac{r_{\xi_1}}{1-|r_{\xi_1}|^2} & 1
\end{array}\right),   & \zeta\in \mathbb{R}_{+}e^{(2\pi-\phi) i},\\[10pt]
\zeta^{-i\nu \hat{\sigma}_3}e^{\frac{i\zeta^2}{4}\hat{\sigma}_3}\left(\begin{array}{cc}
1& -\bar{r}_{\xi_1}\\
0&1
\end{array}\right),   & \zeta\in \mathbb{R}_{+}e^{(\pi+\phi) i},\\[10pt]
\zeta^{-i\nu \hat{\sigma}_3}e^{\frac{i\zeta^2}{4}\hat{\sigma}_3}\left(\begin{array}{cc}
1 & 0\\
r_{\xi_1} & 1
\end{array}\right),   & \zeta\in \mathbb{R}_{+}e^{ (\pi-\phi)i},
\end{array}\right.
\end{align}
with $\nu=\nu(\xi_1)$.\\
- Asymptotic behavior:
\begin{align}
    M^{(PC,\xi_1)}(\zeta) =& I+M^{(PC,\xi_1)}_1\zeta^{-1}+\mathcal{O}(\zeta^{-2}),\hspace{0.5cm}\zeta \rightarrow \infty.
\end{align}
\end{RHP}

The RHP \ref{RHPpc13} has an explicit solution, which can be expressed in terms of Webber equation
$(\frac{\partial^2}{\partial z^2}+(\frac{1}{2}-\frac{z^2}{2}+a))D_{a}(z)=0$.
\begin{figure}[htbp]
	\centering
		\begin{tikzpicture}[node distance=2cm]
		\draw[->](0,0)--(2,1.5)node[above]{$\mathbb{R}_{+}e^{\phi i}$};
		\draw(0,0)--(-2,1.5)node[above]{$\mathbb{R}_{+}e^{(\pi-\phi)i}$};
		\draw(0,0)--(-2,-1.5)node[below]{$\mathbb{R}_{+}e^{(\pi+\phi)i}$};
		\draw[->](0,0)--(2,-1.5)node[below]{$\mathbb{R}_{+}e^{(2\pi-\phi) i}$};
		\draw[dashed,red](0,0)--(3,0)node[right]{\textcolor{black}{${\rm arg}\zeta\in (0,2\pi)$}};
		\draw[->](-2,-1.5)--(-1,-0.75);
		\draw[->](-2,1.5)--(-1,0.75);
		\coordinate (A) at (1,0.5);
		\coordinate (B) at (1,-0.5);
		\coordinate (G) at (-1,0.5);
		\coordinate (H) at (-1,-0.5);
		\coordinate (I) at (0,0);
		\fill (I) circle (1pt) node[below] {$0$};
		\end{tikzpicture}
	\caption{\footnotesize The contour $\Sigma^{pc}$ for the case of $\xi_j$, $j=1,3$.}\label{sigpc13}
\end{figure}
Taking the transformation
\begin{equation}\label{transpc13}
    M^{(PC,\xi_1)}=\mathcal{\psi}(\zeta)\mathcal{P}\zeta^{i\nu\sigma_3}e^{-\frac{i}{4}\zeta^2\sigma_3},
\end{equation}
where
\begin{align}
    \mathcal{P}(\xi)=\left\{\begin{array}{ll}
    \left(\begin{array}{cc}
    1 & \frac{\bar{r}_{\xi_1}}{1-|r_{\xi_1}|^2}\\
    0 & 1
    \end{array}\right),  & \arg\zeta\in(0,\phi) ,\\[10pt]
    \left(\begin{array}{cc}
    1 & 0\\
    \frac{r_{\xi_1}}{1-|r_{\xi_1}|^2} & 1
    \end{array}\right),   & \arg\zeta\in(2\pi-\phi,2\pi) ,\\[10pt]
    \left(\begin{array}{cc}
    1 & 0\\
    -r_{\xi_1} & 1
    \end{array}\right),   & \arg\zeta\in(\pi-\phi,\pi),\\[10pt]
    \left(\begin{array}{cc}
    1& -\bar{r}_{\xi_1}\\
    0&1
    \end{array}\right),   & \arg\zeta\in (\pi,\pi+\phi) ,\\[10pt]
    I,   & else.
    \end{array}\right.
\end{align}
The function $\mathcal{\psi}$ satisfies the following RH conditions.
\begin{RHP}\label{webberequ'}
    Find a $2\times 2$ matrix-valued function $\mathcal{\psi}(\zeta)$ such that\\
        - $\mathcal{\psi}$ is analytical in $\mathbb{C}\backslash\mathbb{R}$; \\
        - Due to the branch cut along $\mathbb{R}_{+}$, $\mathcal{\psi}(\zeta)$ takes continuous boundary values $\mathcal{\psi}_{\pm}$ on $\mathbb{R}$ and
        \begin{equation}\label{psijump13}
            \mathcal{\psi}_{+}(\zeta)=\mathcal{\psi}_{-}(\zeta)V^{\mathcal{\psi}}, \quad \zeta\in\mathbb{R},
        \end{equation}
        where
        \begin{equation}
            V^{\mathcal{\psi}}(\xi)=\left(\begin{array}{cc}
            1-|r_{\xi_1}|^2 & -\bar{r}_{\xi_1}\\
            r_{\xi_1} & 1
            \end{array}\right).
        \end{equation}
        - Asymptotic behavior:
        \begin{equation}\label{pc13psiasy}
            \mathcal{\psi}=\zeta^{-i\nu\sigma_3}e^{\frac{i}{4}\zeta^2\sigma_3}\left(I+M^{(PC,\xi_1)}_1\zeta^{-1}+\mathcal{O}(\zeta^{-2})\right), \quad {\rm as} \quad \zeta\rightarrow\infty.
        \end{equation}
\end{RHP}

Differentiating \eqref{psijump13} with respect to $\zeta$,
and combining $\frac{i\zeta}{2}\sigma_3\psi_{+}=\frac{i\zeta}{2}\sigma_3\psi_{-}V^{\psi}$, we obtain
\begin{equation}
    \left(\frac{d\psi}{d\zeta}-\frac{i\zeta}{2}\sigma_3\psi\right)_{+}=\left(\frac{d\psi}{d\zeta}-\frac{i\zeta}{2}\sigma_3\psi\right)_{-}V^{\psi}.
\end{equation}
Notice that det$V^{\psi}=1$, thus we have det$\psi_{+}$=det$\psi_{-}$. Moreover, we know that det$\psi$ is holomorphic
in $\mathbb{C}$ by Painlev\'e analytic continuation theorem. It follows that $\psi^{-1}$ exists and is bounded. The matrix function
$\left(\frac{d\psi}{d\zeta}-\frac{i\zeta}{2}\sigma_3\psi\right)\psi^{-1}$ has no jump along the real axis and is an entire function
with respect to $\zeta$. Combining \eqref{transpc13}, we can directly calculate that
\begin{equation}\label{dpsiequ}
    \left(\frac{d\psi}{d\zeta}-\frac{i\zeta}{2}\sigma_3\psi\right)\psi^{-1}=
    \left[\frac{dM^{(PC,\xi_1)}}{d\zeta}-M^{(PC,\xi_1)}\frac{i\nu}{\zeta}\sigma_3\right]\left(M^{(PC,\xi_1)}\right)^{-1}+\frac{i\zeta}{2}\left[M^{(PC,\xi_1)},\sigma_3\right]\left(M^{(PC,\xi_1)}\right)^{-1}.
\end{equation}
The first term in the R.H.S of \eqref{dpsiequ} tends to zero as $\zeta\rightarrow\infty$. We use $M^{(PC,\xi_1)}(\zeta)=I+M^{(PC,\xi_1)}_{1}\zeta^{-1}+\mathcal{O}(\zeta^{-2})$
as well as Liouville theorem to obtain that there exists a constant matrix $\beta^{mat}_1$ such that
\begin{equation}
    \left(\begin{array}{cc}
        0 & \beta^{(\xi_1)}_{12}\\
        \beta^{(\xi_1)}_{21} & 0
        \end{array}\right)=
\beta^{mat}_{1}=\frac{i}{2}\left[M^{(PC,\xi_1)}_{1},\sigma_3\right]=\left(\begin{array}{cc}
    0 & -i[M^{(PC,\xi_1)}_{1}]_{12}\\
    i[M^{(PC,\xi_1)}_{1}]_{21} & 0
    \end{array}\right),
\end{equation}
which implies that $[M^{(PC,\xi_1)}_{1}]_{12}=i\beta^{(\xi_1)}_{12}$, $[M^{(PC,\xi_1)}_{1}]_{21}=-i\beta^{(\xi_1)}_{21}$.
Using Liouville theorem again, we have
\begin{equation}
    \left(\frac{d\psi}{d\zeta}-\frac{i\zeta}{2}\sigma_3\psi\right)=\beta^{mat}_1\psi.
\end{equation}
Rewrite the above equality to the following ODE systems
\begin{align}
    \frac{d\psi_{11}}{d\zeta}-\frac{i\zeta}{2}\psi_{11}=\beta^{(\xi_1)}_{12}\psi_{21},\label{pc13ODE1}\\
    \frac{d\psi_{21}}{d\zeta}+\frac{i\zeta}{2}\psi_{21}=\beta^{(\xi_1)}_{21}\psi_{11},\label{pc13ODE2}
\end{align}
as well as
\begin{align}
    \frac{d\psi_{12}}{d\zeta}-\frac{i\zeta}{2}\psi_{12}=\beta^{(\xi_1)}_{12}\psi_{22},\label{pc13ODE3}\\
    \frac{d\psi_{22}}{d\zeta}+\frac{i\zeta}{2}\psi_{22}=\beta^{(\xi_1)}_{21}\psi_{12}.\label{pc13ODE4}
\end{align}
From \eqref{pc13ODE1} to \eqref{pc13ODE4}, we solve that
\begin{align}
    \frac{d^2\psi_{11}}{d\zeta^2}+\left(-\frac{i}{2}+\frac{\zeta^2}{4}-\beta^{(\xi_1)}_{12}\beta^{(\xi_1)}_{21}\right)\psi_{11}=0,\label{pc13psi11}\quad
    \frac{d^2\psi_{21}}{d\zeta^2}+\left(\frac{i}{2}+\frac{\zeta^2}{4}-\beta^{(\xi_1)}_{12}\beta^{(\xi_1)}_{21}\right)\psi_{21}=0,\\
    \frac{d^2\psi_{12}}{d\zeta^2}+\left(-\frac{i}{2}+\frac{\zeta^2}{4}-\beta^{(\xi_1)}_{12}\beta^{(\xi_1)}_{21}\right)\psi_{12}=0,\quad
    \frac{d^2\psi_{22}}{d\zeta^2}+\left(\frac{i}{2}+\frac{\zeta^2}{4}-\beta^{(\xi_1)}_{12}\beta^{(\xi_1)}_{21}\right)\psi_{22}=0.
\end{align}
The Webber equation is
\begin{equation}\label{orpcequ}
    y''+\left(\frac{1}{2}-\frac{z^2}{4}+a\right)y=0.
\end{equation}
The parabolic cylinder functions $D_a(z)$, $D_a(-z)$, $D_{-a-1}(iz)$, $D_{-a-1}(-iz)$ all satisfy \eqref{orpcequ} and are entire $\forall a$.
The large $z$ behavior of $D_a(z)$ can be uniquely given by the following formulaes.
\begin{equation*}
D_{a}(z)=\left\{
             \begin{aligned}
             &z^a e^{-\frac{z^2}{4}}\left(1+\mathcal{O}(z^{-2})\right),  \quad \vert{\rm arg}z\vert<\frac{3\pi}{4},\\
             &z^a e^{-\frac{z^2}{4}}\left(1+\mathcal{O}(z^{-2})\right)-\frac{\sqrt{2\pi}}{\Gamma(-a)}e^{ia\pi}z^{-a-1}e^{z^2/4}\left(1+\mathcal{O}(z^{-2})\right),  \quad \frac{\pi}{4}<{\rm arg}z<\frac{5\pi}{4},\\
             &z^a e^{-\frac{z^2}{4}}\left(1+\mathcal{O}(z^{-2})\right)-\frac{\sqrt{2\pi}}{\Gamma(-a)}e^{-ia\pi}z^{-a-1}e^{z^2/4}\left(1+\mathcal{O}(z^{-2})\right), \quad -\frac{5\pi}{4}<{\rm arg}z<-\frac{\pi}{4}.\\
             \end{aligned}
             \right.
\end{equation*}
We set $\nu=\beta^{(\xi_1)}_{12}\beta^{(\xi_1)}_{21}$. For $\psi_{11}$, $\Im \zeta>0$, we introduce a new variable $\eta=\zeta e^{-\frac{i\pi}{4}}$, and
the first equation of \eqref{pc13psi11} becomes
\begin{equation}
        \frac{d^2\psi_{11}}{d\eta^2}+\left(\frac{1}{2}-\frac{\eta^2}{4}-i\nu\right)\psi_{11}=0.
\end{equation}
For $\zeta\in\mathbb{C}_{+}$, $0<{\rm Arg}\zeta<\pi$, $-\frac{\pi}{4}<{\rm Arg}\eta<\frac{3\pi}{4}$.
We have $\psi_{11}=e^{\frac{\pi}{4}\nu(\xi_1)}D_{-i\nu(\xi_1)}(e^{-\frac{\pi}{4} i}\zeta)\sim\zeta^{-iv}e^{\frac{i}{4}\zeta^2}$, which corresponds
to the $(1,1)$-entry of \eqref{pc13psiasy}.

To limit the length of paper, we present the other results for $\psi$ below without delicate calculation. The unique solution to RH problem \ref{webberequ'} is\\
when $\zeta\in\mathbb{C}_{+}$,
\begin{align}
\mathcal{\psi}(\zeta)=	\left(\begin{array}{cc}
e^{\frac{\pi}{4}\nu(\xi_1)}D_{-i\nu(\xi_1)}(e^{-\frac{\pi}{4} i}\zeta) & \frac{i\nu(\xi_1)}{\beta^{(\xi_1)}_{21}}e^{-\frac{3\pi}{4}(\nu(\xi_1)+i)}D_{i\nu(\xi_1)-1}(e^{-\frac{3\pi i}{4} }\zeta)\\
-\frac{i\nu(\xi_1)}{\beta^{(\xi_1)}_{12}}e^{\frac{\pi}{4}(\nu(\xi_1)-i)}D_{-i\nu(\xi_1)-1}(e^{-\frac{\pi i}{4} }\zeta) & e^{-\frac{3\pi}{4}\nu(\xi_1)}D_{i\nu(\xi_1)}(e^{-\frac{3\pi}{4} i}\zeta)
\end{array}\right),
\end{align}
when $\zeta\in\mathbb{C}_{-}$,
\begin{align}
\mathcal{\psi}(\zeta)=	\left(\begin{array}{cc}
e^{\frac{5\pi}{4}\nu(\xi_1)}D_{-i\nu(\xi_1)}(e^{-\frac{5\pi}{4} i}\zeta) & \frac{i\nu(\xi_1)}{\beta^{(\xi_1)}_{21}}e^{-\frac{7\pi}{4}(\nu(\xi_1)+i)}D_{i\nu(\xi_1)-1}(e^{\frac{-7\pi i}{4} }\zeta)\\
-\frac{i\nu(\xi_1)}{\beta^{(\xi_1)}_{12}}e^{\frac{5\pi}{4}(\nu(\xi_1)-i)}D_{-i\nu(\xi_1)-1}(e^{-\frac{5\pi i}{4} }\zeta) & e^{-\frac{7}{4}\pi\nu(\xi_1)}D_{i\nu(\xi_1)}(e^{-\frac{7}{4}\pi i}\zeta)
\end{array}\right),
\end{align}
Which is similar to \cite[Appendix C.3]{LiuDeriNLS}.

From \eqref{psijump13}, we know that $(\psi_{-})^{-1}\psi_{+}=V^{\psi}$ and
\begin{align}
r_{\xi_1}&=\psi_{-,11}\psi_{+,21}-\psi_{-,21}\psi_{+,11}\nonumber\\
&=e^{\frac{5\pi}{4}\nu(\xi_1)}D_{-i\nu(\xi_1)}(e^{-\frac{5\pi}{4} i}\zeta)\cdot \frac{e^{\frac{\pi\nu(\xi_1)}{4}}}{\beta^{(\xi_1)}_{12}}\left[\partial_{\zeta}(D_{-i\nu(\xi_1)}(e^{-\frac{\pi i}{4}}\zeta))-\frac{i\zeta}{2}D_{-i\nu(\xi_1)}(e^{-\frac{\pi i}{4}}\zeta)\right]\nonumber\\
&\hspace{1em}-e^{\frac{\pi}{4}\nu(\xi_1)}D_{-i\nu(\xi_1)}(e^{-\frac{\pi}{4} i}\zeta)\cdot \frac{e^{\frac{5\pi\nu(\xi_1)}{4}}}{\beta^{(\xi_1)}_{12}}\left[\partial_{\zeta}(D_{-i\nu(\xi_1)}(e^{-\frac{5\pi i}{4}}\zeta))-\frac{i\zeta}{2}D_{-i\nu(\xi_1)}(e^{-\frac{5\pi i}{4}}\zeta)\right]\nonumber\\
&=\frac{e^{\frac{3\pi}{2}\nu(\xi_1)}}{\beta^{(\xi_1)}_{12}}{\rm Wr}\left(D_{-i\nu(\xi_1)}(e^{-\frac{5\pi}{4} i}\zeta), D_{-i\nu(\xi_1)}(e^{-\frac{\pi}{4} i}\zeta)\right)\nonumber\\
&=\frac{e^{\frac{3\pi}{2}\nu(\xi_1)}}{\beta^{(\xi_1)}_{12}}\cdot\frac{\sqrt{2\pi}e^{-\frac{5\pi}{4}i}}{\Gamma(i\nu(\xi_1))}.
\end{align}
The second ``$=$'' we use the equality $D'_{a}(z)+\frac{z}{2}D_{a}(z)=aD_{a-1}(z)$. As for the last ``$=$'', we use the Wronskian identity
Wr$(D_{a}(z),D_{a}(-z))=\frac{\sqrt{2\pi}}{\Gamma(-a)}$.

And
\begin{align}
    &\beta^{(\xi_1)}_{12}=\frac{\sqrt{2\pi}e^{-\frac{5i\pi}{4}}e^{\frac{3\pi \nu(\xi_1)}{2}}}{r_{\xi_1}\Gamma(i\nu(\xi_1))}, \label{betaxi1-1}\\
    &\beta^{(\xi_1)}_{12}\beta^{(\xi_1)}_{21}=\nu(\xi_1),\label{betaxi1-2}\\
    &{\rm arg}\beta^{(\xi_1)}_{21}=-\frac{5\pi}{4}-{\rm arg}r_{\xi_1}-{\rm arg}\Gamma(i\nu(\xi_1))\nonumber\\
    &\hspace*{3em}=-\frac{5\pi}{4}-2\beta(\xi_1,\xi)+2t\theta(\xi_1)+\nu(\xi_1){\rm log}(2t\theta''(\xi_1))-{\rm arg}\Gamma(i\nu(\xi_1)).\label{betaxi1-3}
\end{align}
Finally we have
\begin{equation}\label{mpcxi1}
    M^{(PC,\xi_1)}=I+\frac{1}{\zeta}
        \left(\begin{array}{cc}
        0 & i\beta^{(\xi_1)}_{12}\\
        -i\beta^{(\xi_1)}_{21} & 0
        \end{array}\right)+\mathcal{O}(\zeta^{-2}).
\end{equation}
The results of Appendix \ref{pc13} can also be applied to the local model near $\xi_3$.

\subsection{Local model near $\xi_j$, $j=2,4$}\label{pc24}
\begin{RHP}\label{RHPpc24}
    Find a matrix-valued function $M^{(PC,\xi_2)}(\zeta):=M^{(PC,\xi_2)}(\zeta;\xi)$ such that\\
    - $M^{(PC,\xi_2)}(\zeta;\xi)$ is analytical  in $\mathbb{C}\backslash \Sigma^{pc} $ with $\Sigma^{pc}=\left\{\mathbb{R}e^{i\phi}\right\}\cup
    \left\{\mathbb{R}e^{i(\pi-\phi) } \right\}$ shown in Figure.\ref{sigpc24}.\\
    - $M^{(PC,\xi_2)}$ has continuous boundary values $M^{(PC,\xi_2)}_{\pm}$ on $\Sigma^{pc}$ and
    \begin{equation}
    M^{(PC,\xi_2)}_+(\zeta)=M^{(PC,\xi_2)}_{-}(\zeta)V^{pc}(\zeta),\quad \zeta \in \Sigma^{pc},
    \end{equation}
    where
    \begin{align}
    V^{pc}(\zeta)=\left\{\begin{array}{ll}
    \zeta^{i\nu \hat{\sigma}_3}e^{-\frac{i\zeta^2}{4}\hat{\sigma}_3}\left(\begin{array}{cc}
    1 & 0\\
    r_{\xi_2} & 1
    \end{array}\right),  & \zeta\in\mathbb{R}_{+}e^{\phi i},\\[10pt]
    \zeta^{i\nu \hat{\sigma}_3}e^{-\frac{i\zeta^2}{4}\hat{\sigma}_3}\left(\begin{array}{cc}
    1 & -\bar{r}_{\xi_2}\\
    0 & 1
    \end{array}\right),   & \zeta\in \mathbb{R}_{+}e^{-\phi i},\\[10pt]
    \zeta^{i\nu \hat{\sigma}_3}e^{-\frac{i\zeta^2}{4}\hat{\sigma}_3}\left(\begin{array}{cc}
    1& 0\\
    \frac{r_{\xi_2}}{1-|r_{\xi_2}|^2}&1
    \end{array}\right),   & \zeta\in \mathbb{R}_{+}e^{(\phi-\pi) i},\\[10pt]
    \zeta^{i\nu \hat{\sigma}_3}e^{-\frac{i\zeta^2}{4}\hat{\sigma}_3}\left(\begin{array}{cc}
    1 & -\frac{\bar{r}_{\xi_2}}{1-|r_{\xi_2}|^2}\\
    0 & 1
    \end{array}\right),   & \zeta\in \mathbb{R}_{+}e^{ (-\phi+\pi)i}.
    \end{array}\right.
    \end{align}
    with $\nu=\nu(\xi_2)$. \\
    - Asymptotic behavior:
    \begin{align}
        M^{(PC,\xi_2)}(\zeta) =& I+M^{(PC,\xi_2)}_1\zeta^{-1}+\mathcal{O}(\zeta^{-2}),\hspace{0.5cm}\zeta \rightarrow \infty.
    \end{align}
    \end{RHP}
    The RH problem \ref{RHPpc24} has an explicit solution, which can be expressed in terms of Webber equation
    $(\frac{\partial^2}{\partial z^2}+(\frac{1}{2}-\frac{z^2}{2}+a))D_{a}(z)=0$.
    \begin{figure}
        \centering
            \begin{tikzpicture}[node distance=2cm]
            \draw[->](0,0)--(2,1.5)node[above]{$\mathbb{R}_{+}e^{\phi i}$};
            \draw(0,0)--(-2,1.5)node[above]{$\mathbb{R}_{+}e^{(-\phi+\pi)i}$};
            \draw(0,0)--(-2,-1.5)node[below]{$\mathbb{R}_{+}e^{(\phi-\pi)i}$};
            \draw[->](0,0)--(2,-1.5)node[below]{$\mathbb{R}_{+}e^{-\phi i}$};
            \draw[dashed,red](0,0)--(-3,0)node[left]{\textcolor{black}{${\rm arg}\zeta\in (-\pi,\pi)$}};
            \draw[->](-2,-1.5)--(-1,-0.75);
            \draw[->](-2,1.5)--(-1,0.75);
            \coordinate (A) at (1,0.5);
            \coordinate (B) at (1,-0.5);
            \coordinate (G) at (-1,0.5);
            \coordinate (H) at (-1,-0.5);
            \coordinate (I) at (0,0);
            \fill (I) circle (1pt) node[below] {$0$};
            \end{tikzpicture}
        \caption{\footnotesize The contour $\Sigma^{pc}$ for the case of $\xi_j$, $j=2,4$.}
    \end{figure}\label{sigpc24}
    Taking the transformation
    \begin{equation}\label{transpc24}
        M^{(PC,\xi_2)}=\mathcal{\psi}(\zeta)\mathcal{P}\zeta^{-i\nu\sigma_3}e^{\frac{i}{4}\zeta^2\sigma_3},
    \end{equation}
    where
    \begin{align}
        \mathcal{P}(\xi)=\left\{\begin{array}{ll}
        \left(\begin{array}{cc}
        1 & 0\\
        -r_{\xi_2} & 1
        \end{array}\right),  & \arg\zeta\in(0,\phi) ,\\[10pt]
        \left(\begin{array}{cc}
        1 & -\bar{r}_{\xi_2}\\
        0 & 1
        \end{array}\right),   & \arg\zeta\in(-\phi,0) ,\\[10pt]
        \left(\begin{array}{cc}
        1 & 0\\
        \frac{r_{\xi_2}}{1-|r_{\xi_2}|^2} & 1
        \end{array}\right),   & \arg\zeta\in (\phi-\pi,-\pi),\\[10pt]
        \left(\begin{array}{cc}
        1& \frac{\bar{r}_{\xi_2}}{1-|r_{\xi_2}|^2}\\
        0&1
        \end{array}\right),   & \arg\zeta\in(-\phi+\pi,\pi) ,\\[10pt]
        I,   & else.
        \end{array}\right.
    \end{align}
    The function $\mathcal{\psi}$ satisfies the following RH conditions.
    \begin{RHP}\label{webberequ}
        Find a $2\times 2$ matrix-valued function $\mathcal{\psi}(\zeta)$ such that\\
           - $\mathcal{\psi}$ is analytical in $\mathbb{C}\backslash\mathbb{R}$.\\
           - Due to the branch cut along $\mathbb{R}_{-}$, $\mathcal{\psi}(\zeta)$ takes continuous boundary values $\mathcal{\psi}_{\pm}$ on $\mathbb{R}$ and
            \begin{equation}\label{psijump24}
                \mathcal{\psi}_{+}(\zeta)=\mathcal{\psi}_{-}(\zeta)V^{\mathcal{\psi}}, \quad \zeta\in\mathbb{R},
            \end{equation}
            where
            \begin{equation}
                V^{\mathcal{\psi}}(\xi)=\left(\begin{array}{cc}
                1-|r_{\xi_2}|^2 & -\bar{r}_{\xi_2}\\
                r_{\xi_2} & 1
                \end{array}\right).
            \end{equation}
            - Asymptotic behavior:
            \begin{equation}\label{pc24psiasy}
                \mathcal{\psi}=\zeta^{i\nu\sigma_3}e^{-\frac{i}{4}\zeta^2\sigma_3}\left(I+M^{(PC,\xi_2)}_1\zeta^{-1}+\mathcal{O}(\zeta^{-2})\right), \quad {\rm as} \quad \zeta\rightarrow\infty.
            \end{equation}
    \end{RHP}

    Differentiating \eqref{psijump24} with respect to $\zeta$,
    and combining $\frac{i\zeta}{2}\sigma_3\psi_{+}=\frac{i\zeta}{2}\sigma_3\psi_{-}V^{\psi}$, we obtain
    \begin{equation}
        \left(\frac{d\psi}{d\zeta}+\frac{i\zeta}{2}\sigma_3\psi\right)_{+}=\left(\frac{d\psi}{d\zeta}+\frac{i\zeta}{2}\sigma_3\psi\right)_{-}V^{\psi}.
    \end{equation}
    Since the same reasons presented in Appendix \ref{pc13}, the matrix function
    $\left(\frac{d\psi}{d\zeta}+\frac{i\zeta}{2}\sigma_3\psi\right)\psi^{-1}$ has no jump along the real axis and is an entire function
    with respect to $\zeta$. Combining \eqref{transpc24}, we can directly calculate that
    \begin{equation}\label{dpsiequ24}
        \left(\frac{d\psi}{d\zeta}+\frac{i\zeta}{2}\sigma_3\psi\right)\psi^{-1}=
        \left[\frac{dM^{(PC,\xi_2)}}{d\zeta}+M^{(PC,\xi_2)}\frac{i\nu}{\zeta}\sigma_3\right]\left(M^{(PC,\xi_2)}\right)^{-1}+\frac{i\zeta}{2}\left[\sigma_3,M^{(PC,\xi_2)}\right]\left(M^{(PC,\xi_2)}\right)^{-1}.
    \end{equation}
    The first term in the R.H.S of \eqref{dpsiequ24} tends to zero as $\zeta\rightarrow\infty$. We use $M^{(PC,\xi_2)}(\zeta)=I+M^{(PC,\xi_2)}_{1}\zeta^{-1}+\mathcal{O}(\zeta^{-2})$
    as well as Liouville theorem to obtain that there exists a constant matrix $\beta^{mat}_2$ such that
    \begin{equation}
        \left(\begin{array}{cc}
            0 & \beta^{(\xi_2)}_{12}\\
            \beta^{(\xi_2)}_{21} & 0
            \end{array}\right)=
    \beta^{mat}_{2}=\frac{i}{2}\left[\sigma_3, M^{(PC,\xi_2)}_{1}\right]=\left(\begin{array}{cc}
        0 & i[M^{(PC,\xi_2)}_{1}]_{12}\\
        -i[M^{(PC,\xi_2)}_{1}]_{21} & 0
        \end{array}\right),
    \end{equation}
    which implies that $[M^{(PC,\xi_2)}_{1}]_{12}=-i\beta^{(\xi_1)}_{12}$, $[M^{(PC,\xi_2)}_{1}]_{21}=i\beta^{(\xi_1)}_{21}$.
    Using Liouville theorem again, we have
    \begin{equation}
        \left(\frac{d\psi}{d\zeta}+\frac{i\zeta}{2}\sigma_3\psi\right)=\beta^{mat}_{2}\psi.
    \end{equation}
    We rewrite the above equality to the following ODE system
    \begin{align}
        \frac{d\psi_{11}}{d\zeta}+\frac{i\zeta}{2}\psi_{11}=\beta^{(\xi_2)}_{12}\psi_{21},\label{pc24ODE1}\\
        \frac{d\psi_{21}}{d\zeta}-\frac{i\zeta}{2}\psi_{21}=\beta^{(\xi_2)}_{21}\psi_{11},\label{pc24ODE2}
    \end{align}
    as well as
    \begin{align}
        \frac{d\psi_{12}}{d\zeta}+\frac{i\zeta}{2}\psi_{12}=\beta^{(\xi_2)}_{12}\psi_{22},\label{pc24ODE3}\\
        \frac{d\psi_{22}}{d\zeta}-\frac{i\zeta}{2}\psi_{22}=\beta^{(\xi_2)}_{21}\psi_{12},\label{pc24ODE4}.
    \end{align}
    From \eqref{pc24ODE1} to \eqref{pc24ODE4}, we solve that
    \begin{align}
        \frac{d^2\psi_{11}}{d\zeta^2}+\left(\frac{i}{2}+\frac{\zeta^2}{4}-\beta^{(\xi_2)}_{12}\beta^{(\xi_2)}_{21}\right)\psi_{11}=0,\label{pc24psi11}\quad
        \frac{d^2\psi_{21}}{d\zeta^2}+\left(-\frac{i}{2}+\frac{\zeta^2}{4}-\beta^{(\xi_2)}_{12}\beta^{(\xi_2)}_{21}\right)\psi_{21}=0,\\
        \frac{d^2\psi_{12}}{d\zeta^2}+\left(\frac{i}{2}+\frac{\zeta^2}{4}-\beta^{(\xi_2)}_{12}\beta^{(\xi_2)}_{21}\right)\psi_{12}=0,\quad
        \frac{d^2\psi_{22}}{d\zeta^2}+\left(-\frac{i}{2}+\frac{\zeta^2}{4}-\beta^{(\xi_2)}_{12}\beta^{(\xi_2)}_{21}\right)\psi_{22}=0.
    \end{align}
    We set $\nu=\beta^{(\xi_2)}_{12}\beta^{(\xi_2)}_{21}$. For $\psi_{11}$, $\Im\zeta>0$, we introduce the new variable $\eta=\zeta e^{-\frac{3i\pi}{4}}$, and
    the first equation of \eqref{pc24psi11} becomes
    \begin{equation}
        \frac{d^2\psi_{11}}{d\eta^2}+\left(\frac{1}{2}-\frac{\eta^2}{4}+i\nu\right)\psi_{11}=0.
    \end{equation}
    For $\zeta\in\mathbb{C}_{+}$, $0<{\rm Arg}\zeta<\pi$, $-\frac{3\pi}{4}<{\rm Arg}\eta<\frac{\pi}{4}$.
    We have $\psi_{11}=e^{-\frac{3\pi}{4}\nu(\xi_2)}D_{i\nu(\xi_2)}(e^{-\frac{3\pi}{4} i}\zeta)\sim\zeta^{iv}e^{-\frac{i}{4}\zeta^2}$, which
    corresponds to the $(1,1)$-entry of \eqref{pc24psiasy}. The other results for $\psi$ are presented below.

    The unique solution to RH problem \ref{webberequ} is\\
    when $\zeta\in\mathbb{C}_{+}$,
    \begin{align}
    \mathcal{\psi}(\zeta)=	\left(\begin{array}{cc}
    e^{-\frac{3\pi}{4}\nu(\xi_2)}D_{i\nu(\xi_2)}(e^{-\frac{3\pi}{4} i}\zeta) & -\frac{i\nu(\xi_2)}{\beta^{(\xi_2)}_{21}}e^{\frac{\pi}{4}(\nu(\xi_2)-i)}D_{-i\nu(\xi_2)-1}(e^{-\frac{\pi i}{4} }\zeta)\\
    \frac{i\nu(\xi_2)}{\beta^{(\xi_2)}_{12}}e^{-\frac{3\pi}{4}(\nu(\xi_2)+i)}D_{i\nu(\xi_2)-1}(e^{-\frac{3\pi i}{4} }\zeta) & e^{\frac{\pi}{4}\nu(\xi_2)}D_{-i\nu(\xi_2)}(e^{-\frac{\pi}{4} i}\zeta)
    \end{array}\right),
    \end{align}
    when $\zeta\in\mathbb{C}_{-}$,
    \begin{align}
    \mathcal{\psi}(\zeta)=	\left(\begin{array}{cc}
    e^{\frac{\pi\nu(\xi_2)}{4}}D_{i\nu(\xi_2)}(e^{\frac{\pi}{4} i}\zeta) & -\frac{i\nu(\xi_2)}{\beta^{(\xi_2)}_{21}}e^{-\frac{3\pi}{4}(\nu(\xi_2)-i)}D_{-i\nu(\xi_2)-1}(e^{\frac{3\pi i}{4} }\zeta)\\
    \frac{i\nu(\xi_2)}{\beta^{(\xi_2)}_{12}}e^{\frac{\pi}{4}(\nu(\xi_2)+i)}D_{i\nu(\xi_2)-1}(e^{\frac{\pi i}{4} }\zeta) & e^{-\frac{3\pi}{4}\nu(\xi_2)}D_{-i\nu(\xi_2)}(e^{\frac{3\pi}{4} i}\zeta)
    \end{array}\right),
    \end{align}
    Which is derived in \cite[Section 4]{DZAnn} and verified in \cite[Proposition 5.5]{LiuDeriNLS}.

    From \eqref{psijump24}, we know that $(\psi_{-})^{-1}\psi_{+}=V^{\psi}$ and
    \begin{align}
    r_{\xi_2}&=\psi_{-,11}\psi_{+,21}-\psi_{-,21}\psi_{+,11}\nonumber\\
    &=e^{\frac{\pi}{4}\nu(\xi_2)}D_{i\nu(\xi_2)}(e^{\frac{\pi}{4} i}\zeta)\cdot \frac{e^{-\frac{3\pi\nu(\xi_2)}{4}}}{\beta^{(\xi_2)}_{12}}\left[\partial_{\zeta}(D_{i\nu(\xi_2)}(e^{-\frac{3\pi i}{4}}\zeta))+\frac{i\zeta}{2}D_{i\nu(\xi_2)}(e^{-\frac{3\pi i}{4}}\zeta)\right]\nonumber\\
    &\hspace{1em}-e^{-\frac{3\pi}{4}\nu(\xi_2)}D_{i\nu(\xi_2)}(e^{-\frac{3\pi}{4} i}\zeta)\cdot \frac{e^{\frac{\pi\nu(\xi_2)}{4}}}{\beta^{(\xi_2)}_{12}}\left[\partial_{\zeta}(D_{i\nu(\xi_2)}(e^{\frac{\pi i}{4}}\zeta))+\frac{i\zeta}{2}D_{i\nu(\xi_2)}(e^{\frac{\pi i}{4}}\zeta)\right]\nonumber\\
    &=\frac{e^{-\frac{\pi}{2}\nu(\xi_2)}}{\beta^{(\xi_2)}_{12}}{\rm Wr}\left(D_{i\nu(\xi_2)}(e^{\frac{\pi}{4} i}\zeta), D_{i\nu(\xi_2)}(e^{-\frac{3\pi}{4} i}\zeta)\right)\nonumber\\
    &=\frac{e^{-\frac{\pi}{2}\nu(\xi_2)}}{\beta^{(\xi_2)}_{12}}\cdot\frac{\sqrt{2\pi}e^{\frac{\pi}{4}i}}{\Gamma(-i\nu(\xi_2))}.
    \end{align}

    And
    \begin{align}
        &\beta^{(\xi_2)}_{12}=\frac{\sqrt{2\pi}e^{\frac{i\pi}{4}}e^{-\frac{\pi \nu(\xi_2)}{2}}}{r_{\xi_2}\Gamma(-i\nu(\xi_2))},\label{betaxi2-1}\\
        &\beta^{(\xi_2)}_{12}\beta^{(\xi_2)}_{21}=\nu(\xi_2),\label{betaxi2-2}\\
        &{\rm arg}\beta^{(\xi_2)}_{12}=\frac{\pi}{4}-{\rm arg}r_{\xi_2}-{\rm arg}\Gamma(-i\nu(\xi_2))\nonumber\\
        &\hspace*{3em}=\frac{\pi}{4}-2\beta(\xi_2,\xi)+2t\theta(\xi_2)-\nu(\xi_2){\rm log}(-2t\theta''(\xi_2))-{\rm arg}\Gamma(-i\nu(\xi_2)).\label{betaxi2-3}
    \end{align}
    As a consequence,
    \begin{equation}\label{mpcxi2}
        M^{(PC,\xi_2)}=I+\frac{1}{\zeta}
            \left(\begin{array}{cc}
            0 & -i\beta^{(\xi_2)}_{12}\\
            i\beta^{(\xi_2)}_{21} & 0
            \end{array}\right)+\mathcal{O}(\zeta^{-2}).
    \end{equation}
    The results of Appendix \ref{pc24} also can be applied to the local model near $\xi_4$.

\end{document}